\documentclass[a4paper,11pt]{article}
\usepackage{amsmath,amssymb,graphicx,amsthm}

\usepackage{algorithm}
\usepackage{algpseudocode}
\usepackage{subcaption}
\usepackage{tikz}
\usetikzlibrary{calc}
\usepackage{tkz-berge}
\usepackage{placeins}
\usepackage{array}
\usetikzlibrary{arrows.meta,decorations.markings,decorations}
\usepackage{mathtools}
\usepackage{multirow}
\usepackage{enumitem}
\usepackage{makecell}
\usepackage{multicol}

\usepackage{pgfplots}
\pgfplotsset{width=8.5cm}
\pgfplotsset{compat=1.18}

\usepackage[
top    = 2cm,
bottom = 2.25cm,
left   = 1.75cm,
right  = 1.75cm]{geometry}

\usepackage[margin=1cm]{caption}

\newtheorem{THM}{THEOREM}[section]

\newtheorem{theorem}[THM]{Theorem}
\newtheorem{lemma}[THM]{Lemma}
\newtheorem{claim}[THM]{Claim}
\newtheorem{corollary}[THM]{Corollary}

\theoremstyle{definition}
\newtheorem{remark}{Remark}[section]
\newtheorem{problem}{Problem}   %

\newcommand{\Input}{\item[\textbf{Input:}]}
\newcommand{\Output}{\item[\textbf{Output:}]}

\pgfdeclarelayer{background}
\pgfdeclarelayer{foreground}
\pgfsetlayers{background,main,foreground}

\usetikzlibrary{decorations.pathreplacing,calc}
\tikzset{draw half paths/.style 2 args={%
    decoration={show path construction,
      lineto code={
        \draw [#1] (\tikzinputsegmentfirst) --
        ($(\tikzinputsegmentfirst)!0.2!(\tikzinputsegmentlast)$);
        \draw [#2] ($(\tikzinputsegmentfirst)!0.2!(\tikzinputsegmentlast)$)
        -- (\tikzinputsegmentlast);
      }
    }, decorate
  }}

\tikzset{
  my box/.style = {
    , line cap = round
    , line join = round
  }
}
\newcommand{\highlight}[3]{
  \path [my box, line width = #1, draw = #2, transparency group, opacity=1] #3;
}

\newenvironment{subproof}[1][\proofname]{%
  \begin{proof}[#1]%
  }{%
  \end{proof}%
}

\usepackage[textsize=tiny]{todonotes}

\def\R{\mathbb R}

\def\Z{\mathbb Z}

\def\1{\mathbb 1}

\DeclareMathOperator*{\argmin}{arg\,min}
\DeclareMathOperator{\poly}{poly}
\DeclareMathOperator{\opt}{opt}

\def\niceArrow{-{Stealth[length=2.25mm]}}
\def\NiceArrow{-{Stealth[length=2.5mm]}}

\usepackage{accents}
\usepackage{esvect}
\makeatletter
\DeclareRobustCommand{\cev}[1]{%
  {\mathpalette\do@cev{#1}}%
}
\newcommand{\do@cev}[2]{%
  \vbox{\offinterlineskip
    \sbox\z@{$\m@th#1 x$}%
    \ialign{##\cr
      \hidewidth\reflectbox{$\m@th#1\vec{}\mkern0mu$}\hidewidth\cr
      \noalign{\kern-\ht\z@}
      $\m@th#1#2$\cr
    }
  }
}
\makeatother

\newcolumntype{x}[1]{>{\centering\let\newline\\\arraybackslash\hspace{0pt}}m{#1}}

\usetikzlibrary{decorations.pathmorphing}
\tikzset{wavy/.style={decorate,decoration={snake,amplitude=.4mm,segment length=2mm,post length=0mm,pre length=0mm},line width=.5}}

\tikzset{draw third paths/.style 2 args={%
    decoration={show path construction,
      lineto code={
        \draw [#1] (\tikzinputsegmentfirst) -- ($(\tikzinputsegmentfirst)!0.05!(\tikzinputsegmentlast)$);
        \draw [#2] ($(\tikzinputsegmentfirst)!0.05!(\tikzinputsegmentlast)$) -- ($(\tikzinputsegmentfirst)!0.95!(\tikzinputsegmentlast)$);
      }
    }, decorate
  }}

\tikzset{paralleledge/.style={to path={
      \pgfextra{%
        \pgfmathsetmacro{\startf}{-(#1-1)/2}
        \pgfmathsetmacro{\endf}{-\startf}
        \pgfmathsetmacro{\stepf}{\startf+1}}
      \ifnum 1=#1 -- (\tikztotarget)  \else
        let \p{mid}=($(\tikztostart)!0.5!(\tikztotarget)$)
        in
        \foreach \i in {\startf,\stepf,...,\endf}
        {%
          (\tikztostart) .. controls ($ (\p{mid})!\i*6pt!90:(\tikztotarget) $) .. (\tikztotarget)
        }
      \fi
      \tikztonodes
    }
  }
}

\tikzset{circleAroundEdges/.style n args={3}{
    decorate, decoration = {markings, mark=at position #3 with {\draw[double=black,white,double distance=.6pt] (0,#2) arc [x radius = #1, y radius = #2, start angle = 90, end angle = -125];} , mark=at position #3 with{\draw[double=black,white,double distance=.6pt] (0,-#2) arc [x radius = #1, y radius = #2, start angle = 270, end angle = 360+125];}}
  }
}

\usepackage{titling}
\usepackage[hang,flushmargin]{footmisc}
\usepackage[yyyymmdd]{datetime}
\usepackage[hidelinks]{hyperref}

\begin{document}

\title{Separable convex optimization over indegree polytopes %
}

\author{N\'ora A.\ Borsik\thanks{Department of Operations Research, ELTE E\"otv\"os Lor\'and University, P\'azm\'any P.\ s.\ 1/c, Budapest H-1117, Hungary. E-mail: {\tt borsiknora@student.elte.hu}}
  \and P\'eter Madarasi\thanks{HUN-REN Alfr\'ed R\'enyi Institute of Mathematics, Re\'altanoda u.\ 13--15, Budapest H-1053, Hungary; and Department of Operations Research, ELTE E\"otv\"os Lor\'and University, P\'azm\'any P.\ s.\ 1/c, Budapest H-1117, Hungary. E-mail: \texttt{madarasi@renyi.hu}, corresponding author}
}

\date{\vspace*{-32pt}}

\maketitle
\begin{abstract}
  We study separable discrete convex optimization over indegree vectors of orientations of an undirected graph, focusing on the effect of imposing acyclicity.
  Since the extreme points of the indegree polytope are precisely the indegree vectors of acyclic orientations, this amounts to comparing optimization over its integer points with optimization over its extreme points.
  Over all orientations, minimizing a vertex-separable discrete convex objective is strongly polynomial-time solvable by minimum-cost flow.
  Over acyclic orientations, for every fixed discrete convex function shared by all vertices, we obtain a dichotomy: if the function is linear on the positive integers, the problem is solvable in linear time; otherwise exact optimization is NP-hard on connected simple graphs, and APX-hard when the function is non-negative.
  We also give strongly polynomial-time constant-factor approximation algorithms for broad classes of convex and concave objectives.
  For the $p$\nobreakdash-power sum of indegrees, this yields a $2^p$\nobreakdash-approximation for every fixed integer $p\geq1$, while the problem is APX-complete for every fixed integer $p\geq2$; if $p$ is part of the input, no deterministic polynomial-time $2^{p^{1-\varepsilon}}$\nobreakdash-approximation exists for any fixed $\varepsilon>0$ unless $\mathrm{P}=\mathrm{NP}$.
  We further give an exact $O(2^{|V|}|V|+\operatorname{poly}(|V|,|E|))$\nobreakdash-time dynamic program for general separable objectives, with a matching lower bound in the exponent under the Exponential Time Hypothesis.
  Finally, decreasingly minimal and increasingly maximal acyclic orientations need not coincide and are both NP-hard, while minimizing the maximum weighted indegree over acyclic orientations remains strongly polynomial-time solvable.

\end{abstract}

\section{Introduction}\label{sec:introduction}
This paper studies separable discrete convex objectives over the indegree vectors of orientations of an undirected graph, with particular emphasis on the effect of imposing acyclicity.
In particular, we investigate the following problem.

\begin{problem}[$\min\sum_{v\in V} \varphi_v(\varrho(v))$ orientation]\label{prob:sumPhi_VRhoV}
  Given a loop-free multigraph $G = (V, E)$ and, for each $v \in V$, a discrete convex function $\varphi_v: \Z_+ \to \R$ (i.e., a function satisfying $\varphi_v(z+1) + \varphi_v(z-1) \geq 2\varphi_v(z)$ for every positive integer $z$), find an orientation that minimizes $\sum_{v \in V} \varphi_v(\varrho(v))$, where $\varrho(v)$ denotes the indegree of vertex $v$.
\end{problem}

This problem includes several notable special cases, such as maximizing $\prod_{v \in V} (\varrho(v)+1)$, maximizing $\sum_{v\in V}\varrho(v)\delta(v)$, and minimizing $\sum_{v \in V} \left([\alpha(v) - \varrho(v)]^+ + [\varrho(v) - \beta(v)]^+\right)$ for $\alpha,\beta : V \to \R$, where $\delta(v)$ denotes the outdegree of vertex $v$ and $[x]^+=\max\{x,0\}$ for $x\in \R$.

Of distinguished interest is the case of Problem~\ref{prob:sumPhi_VRhoV} in which all vertices share the same discrete strictly convex function $\varphi : \Z_+ \to \R$, so the objective is to minimize the $\varphi$\nobreakdash-sum of indegrees (i.e., $\varphi(z+1) + \varphi(z-1) > 2\varphi(z)$ for every positive integer $z$).
As we will see in Section~\ref{sec:decMinIncMax}, this includes as special cases the \emph{decreasingly-minimal (dec-min)} and \emph{increasingly-maximal (inc-max)} objectives, aiming to lexicographically minimize the indegree vector sorted in non-increasing order and lexicographically maximize the indegree vector sorted in non-decreasing order, respectively.
These objectives enforce that the indegree vector becomes ``smooth'', ``equitable'', or ``egalitarian'', and have been studied by Frank and Murota~\cite{frank2022decreasing2, frank2022decreasing1} under the name \emph{egalitarian orientation problems}.
They proved that a single orientation simultaneously optimizes these three types of objectives.
However, under the acyclicity constraint these fairness objectives can diverge, as we show in Section~\ref{sec:decMinIncMax}.

\medskip
A significant part of the present paper is devoted to the investigation of such egalitarian orientation problems under the additional requirement that the orientation must be \emph{acyclic}:
\begin{problem}[$\min\sum_{v\in V} \varphi(\varrho(v))$ acyclic orientation]\label{prob:sumPhiRhoV}
  Given a loop-free multigraph $G = (V, E)$ and a discrete strictly convex function $\varphi: \Z_+ \to \R$, find an acyclic orientation that minimizes $\sum_{v \in V} \varphi(\varrho(v))$.
\end{problem}

\begin{problem}[decreasingly minimal (dec-min) acyclic orientation]\label{prob:decMin}
  Given a loop-free multigraph $G = (V, E)$, find an acyclic orientation in which the indegree vector sorted in non-increasing order is lexicographically minimal.
\end{problem}

\begin{problem}[increasingly maximal (inc-max) acyclic orientation]\label{prob:incMax}
  Given a loop-free multigraph $G = (V, E)$, find an acyclic orientation in which the indegree vector sorted in non-decreasing order is lexicographically maximal.
\end{problem}

An important observation is that each of the acyclic orientation problems introduced above can be equivalently reformulated as a \emph{vertex ordering problem}.
Rather than searching directly for an optimal acyclic orientation, one may instead seek a topological ordering of an optimal orientation.
In this reformulation, indegrees $\varrho$ correspond to left-degrees $\cev{d}$.
An ordering of the vertices is optimal for the vertex ordering problem if and only if the corresponding acyclic orientation --- obtained by directing each edge from the earlier to the later vertex in the order --- is optimal for the original orientation problem.
Throughout this paper, we do not explicitly distinguish between acyclic orientation problems and their vertex ordering counterparts, though we primarily adopt the vertex ordering perspective.

\medskip
We note that Problem~\ref{prob:sumPhi_VRhoV} has a self-generalizing nature, namely, one can pose lower and upper bounds $f, g : V \to \Z_+$ on the indegrees of the vertices, where $f(v) \leq g(v)$ for every $v \in V$.
To this end, we introduce a modified discrete convex function $\varphi_v^{fg}$ for each $v \in V$, which turns $\varphi_v$ into a fast-decreasing linear function before $f(v)$, and a fast-increasing linear function after $g(v)$.
Formally, we define the \emph{$(f, g)$-lifting} of $\varphi_v$ as

\begin{equation}\label{eq:phivfg}
  \varphi_v^{fg}(z)
  = \begin{cases}
    \varphi_v(z)                    &\text{if } f(v) \leq z \leq g(v),\\
    \varphi_v(f(v))+(f(v) - z) M    &\text{if } z < f(v),\\
    \varphi_v(g(v))+(z - g(v)) M    &\text{if } g(v) < z,
  \end{cases}
\end{equation}
where $M$ is a sufficiently large integer.
This modified function penalizes those orientations that violate the indegree bounds $f$ and $g$, restricting the search to $(f,g)$-indegree-bounded orientations if such an orientation exists.
Furthermore, our framework extends to optimizing over the orientations of mixed graphs.
For each $v \in V$, let $r_v$ be the number of fixed directed arcs entering $v$.
After removing all directed arcs, replace $\varphi_v(z)$ by $\varphi_v(z+r_v)$.
If lower and upper bounds $f$ and $g$ are imposed on the total indegrees, replace them by $\max\{0,f(v)-r_v\}$ and $g(v)-r_v$, respectively; if $g(v)<r_v$ for some $v$, the instance is infeasible.

\paragraph{Motivation}
We provide a purely mathematical motivation for Problems~\ref{prob:sumPhiRhoV}-\ref{prob:incMax}.
Consider the indegree vectors of all orientations of a loop-free undirected multigraph $G = (V, E)$, and take their convex hull in $\R^V$.
The resulting polytope is called the \emph{indegree polytope} and is denoted by $P_{\varrho}(G)$.
It is known that $P_{\varrho}(G)$ is a base polyhedron and its integer points are precisely the indegree vectors of orientations of $G$~\cite{frank2022decreasing2,frank2022decreasing1}.
Therefore, the unconstrained (possibly cyclic) versions of Problems~\ref{prob:sumPhiRhoV}-\ref{prob:incMax} aim to optimize the corresponding objective function ($\min \sum_{v\in V} \varphi(\varrho(v))$, dec-min, or inc-max) over the integer points of $P_{\varrho}(G)$.
These problems are known to be solvable in strongly polynomial time~\cite{borradaile2017egalitarian,frank2022decreasing2,frank2022decreasing1}, despite the fact that the objective functions are non-linear --- so standard linear programming theorems do not apply directly.
A standard fact about the indegree polytope is that its corners, i.e., its vertices or extreme points, are exactly the indegree vectors of \emph{acyclic} orientations of $G$.
Consequently, requiring acyclicity is precisely the restriction from the integer points of $P_{\varrho}(G)$ to its corners, yielding Problems~\ref{prob:sumPhiRhoV}-\ref{prob:incMax}.

\paragraph{Our contribution}
Our principal results concern the effect of imposing acyclicity on separable indegree optimization.
Section~\ref{sec:convexHullOfIndegrees} sets up the algorithmic baseline for this comparison.
We use the standard corner characterization above and, for completeness, record an elementary polynomial-size totally unimodular extended formulation based on the bipartite incidence graph.
Against this background, we give a simple, strongly polynomial-time minimum-cost flow algorithm for Problem~\ref{prob:sumPhi_VRhoV} over all orientations.

Our first main result shows that the acyclicity requirement causes a sharp change in complexity.
In Section~\ref{sec:minSumPhi}, we prove a complexity dichotomy for minimizing the $\varphi$\nobreakdash-sum for every fixed discrete convex function $\varphi$.
If $\varphi$ is linear on the positive integers, minimizing the $\varphi$\nobreakdash-sum over acyclic orientations is solvable in linear time; otherwise the problem is NP-hard on simple graphs, and it is APX-hard when $\varphi$ is non-negative.
The hardness proof first treats functions satisfying $\varphi(1)+\varphi(3)>2\varphi(2)$ and then extends the result to every integer $k\geq2$ satisfying $\varphi(k-1)+\varphi(k+1)>2\varphi(k)$.
For every such integer $k$, the hard instances can be chosen connected and simple so that $k=\left\lceil \frac{|E|}{|V|}\right\rceil$.
Since every orientation has total indegree $|E|$, this means that $k$ is the ceiling of the average indegree of the hard instance and hence a lower bound on the maximum indegree of every orientation.
The assumption that $\varphi$ is fixed is essential for this universal APX-hardness statement.
We exhibit a graph-dependent non-negative discrete strictly convex function for which the problem admits a PTAS.
Nevertheless, for every fixed integer $k\geq2$, exact optimization is NP-hard for arbitrary graph-dependent discrete convex functions $\varphi_G$ satisfying $\varphi_G(k-1)+\varphi_G(k+1)>2\varphi_G(k)$, even on connected simple graphs satisfying $\left\lceil \frac{|E|}{|V|}\right\rceil=k$.
In particular, this implies the NP-hardness of Problems~\ref{prob:decMin} and~\ref{prob:incMax}.

Our second main contribution is a general approximation framework for the acyclic problem.
Despite the hardness above, Section~\ref{sec:minSumPhi} gives strongly polynomial-time constant-factor approximation algorithms for broad classes of convex minimization and concave maximization objectives, including logarithmic and fractional-power concave objectives.

Our third main contribution concerns the $p$\nobreakdash-power objective $\varphi(z)=z^p$ for a positive integer $p$, studied in Section~\ref{sec:powerObjectives}.
For every fixed integer $p \geq 1$, we obtain a strongly polynomial-time $2^p$\nobreakdash-approximation for minimizing the $p$\nobreakdash-power sum of indegrees over acyclic orientations, while the problem is APX-complete for every fixed integer $p \geq 2$.
When the positive integer exponent $p$ is part of the input, we show that the exponential dependence on $p$ is near-tight in the exponent.
For every fixed $\varepsilon>0$, no deterministic polynomial-time $2^{p^{1-\varepsilon}}$\nobreakdash-approximation exists unless $\mathrm{P}=\mathrm{NP}$.
We also give two approximation bounds for the smallest-last greedy algorithm for the quadratic objective $\varphi(z)=z^2$.

We complement these results with exact exponential-time and lexicographic results.
Section~\ref{sec:minSumPhi} gives an $O\left(2^{|V|}|V|+\poly(|V|,|E|)\right)$-time dynamic programming algorithm for arbitrary vertex-specific objectives and, under the Exponential Time Hypothesis, a matching lower bound on the exponential dependence for every fixed discrete convex function that is not linear on the positive integers.
Section~\ref{sec:decMinIncMax} studies lexicographic objectives.
One motivation is the problem posed in~\cite{borradaile2017egalitarian} of approximating the dec-min objective under the acyclicity constraint.
We first show that dec-min and inc-max acyclic orientations need not coincide and give a strongly polynomial-time algorithm for minimizing the maximum weighted indegree over acyclic orientations.
We then prove APX-hardness for the next lexicographic refinement of the dec-min objective and for a related refinement of the inc-max objective.
We also prove APX-hardness for natural first steps of the complementary inc-min and dec-max objectives, both with and without the acyclicity requirement.

\paragraph{Applications}
First, we propose an application of Problem~\ref{prob:sumPhi_VRhoV}, which is a generalization of the scheduling problem investigated in~\cite{burcea2016scheduling}.
Consider a set $J$ of unit-sized jobs and a set $T$ of timeslots.
Each job $j \in J$ must be assigned to exactly one timeslot from a predefined feasible set $I_j \subseteq T$ associated with the job $j$.
The load $\ell(t)$ of a timeslot $t \in T$ is the total number of jobs assigned to it.
The goal is to minimize the total cost $\sum_{t \in T} h(\ell(t))$ for a discrete convex function $h$, which measures the cost of each timeslot $t$ based on the load $\ell(t)$.
This problem is known to be solvable in polynomial time~\cite{burcea2016scheduling}.

We show how this problem fits into the framework of Problem~\ref{prob:sumPhi_VRhoV}.
In fact, we consider a more general scheduling problem in which each timeslot $t \in T$ has its own discrete convex cost function $h_t$, and our goal is to minimize $\sum_{t \in T} h_t(\ell(t))$.
Consider a bipartite graph $G = (J, T; E)$, where $J$ is the set of jobs, $T$ is the set of timeslots, and $E$ contains an edge between the job $j$ and the timeslot $t$ if and only if $t \in I_j$.
Let $\varphi_t = h_t$, and define $\varphi_j$ as the $(f,g)$-lifting of the all-zero function with $f(j) = g(j) = |I_j| - 1$, so that in every optimal orientation exactly one arc leaves each job $j$.
That arc enters the timeslot in which $j$ is scheduled.
The indegree of each timeslot $t$ is exactly its load $\ell(t)$.
Thus, there is a natural cost-preserving correspondence between feasible schedules and feasible orientations for Problem~\ref{prob:sumPhi_VRhoV}.
In addition, lower and upper bounds on the load of each timeslot can be incorporated by applying the lifting defined above to $\varphi_t$ with the corresponding lower and upper bounds for each $t\in T$.

\smallskip
Furthermore, a practical application of egalitarian acyclic orientation problems arises in the context of network routing.
In this setting, an undirected graph models the network, where vertices represent switches and edges correspond to communication links that provide two channels in the two directions between the endpoints, each with specified capacity.
In \emph{wormhole routing}, each packet is divided into small units called \emph{flits}.
The first and last of these are the so-called \emph{header flit} and the \emph{tail flit}, respectively.
The header flit carries routing information, including source and destination identifiers.
All flits of a packet traverse the network in sequence.
As the header flit travels, it reserves channel capacity for the packet; the reservation is released when the tail flit passes through the given channel.
During routing, each switch attempts to forward any flit waiting at an incoming channel to the next channel along its path --- provided that the next channel has available capacity.

A major challenge is to prevent \emph{deadlock}, a condition in which flits are indefinitely blocked by cyclic dependencies among channels, thereby stalling the system.
Figure~\ref{fig:deadlock} illustrates such a scenario.
\begin{figure}[h]
  \centering
  \begin{tikzpicture}[scale=0.5]
    \tikzset{VertexStyle/.append style = {minimum size = 29pt,inner sep=0pt, shape=rectangle}}
    \SetVertexNoLabel
    \Vertex[x=0, y=0, L=A]{A}
    \Vertex[x=4, y=0, L=B]{B}
    \Vertex[x=4, y=-4, L=C]{C}
    \Vertex[x=0, y=-4, L=D]{D}

    \draw [blue] (0.75,2) node (p1) {$P_1$};
    \draw [\NiceArrow, line width=1.5pt, blue] (0.15,3)--(0.15,1);
    \draw [\NiceArrow, line width=1.5pt, black!35] (-0.15,1)--(-0.15,3);
    \draw [\niceArrow, blue] (0.15,1)--(0.15,-1);
    \draw [\NiceArrow, line width=1.5pt, blue] (0.15,-1)--(0.15,-3);
    \draw [\NiceArrow, line width=1.5pt, black!35] (-0.15,-3)--(-0.15,-1);
    \draw [\niceArrow, blue, dashed] (0.15,-3)--(1,-3.85);

    \draw [cyan] (-2,-3.25) node (p2) {$P_2$};
    \draw [\NiceArrow, line width=1.5pt, cyan] (-3,-3.85)--(-1,-3.85);
    \draw [\NiceArrow, line width=1.5pt, black!35] (-1,-4.15)--(-3,-4.15);
    \draw [\niceArrow, cyan] (-1,-3.85)--(1,-3.85);
    \draw [\NiceArrow, line width=1.5pt, cyan] (1,-3.85)--(3,-3.85);
    \draw [\NiceArrow, line width=1.5pt, black!35] (3,-4.15)--(1,-4.15);
    \draw [\niceArrow, cyan, dashed] (3,-3.85)--(3.85,-3);

    \draw [red] (3.25,-6) node (p3) {$P_3$};
    \draw [\NiceArrow, line width=1.5pt, red] (3.85,-7)--(3.85,-5);
    \draw [\NiceArrow, line width=1.5pt, black!35] (4.15,-5)--(4.15,-7);
    \draw [\niceArrow, red] (3.85,-5)--(3.85,-3);
    \draw [\NiceArrow, line width=1.5pt, red] (3.85,-3)--(3.85,-1);
    \draw [\NiceArrow, line width=1.5pt, black!35] (4.15,-1)--(4.15,-3);
    \draw [\niceArrow, red, dashed] (3.85,-1)--(3,-0.15);

    \draw [green] (6,-0.75) node (p4) {$P_4$};
    \draw [\NiceArrow, line width=1.5pt, green] (7,-0.15)--(5,-0.15);
    \draw [\NiceArrow, line width=1.5pt, black!35] (5,0.15)--(7,0.15);
    \draw [\niceArrow, green] (5,-0.15)--(3,-0.15);
    \draw [\NiceArrow, line width=1.5pt, green] (3,-0.15)--(1,-0.15);
    \draw [\NiceArrow, line width=1.5pt, black!35] (1,0.15)--(3,0.15);
    \draw [\niceArrow, green, dashed] (1,-0.15)--(0.15, -1);

    \tikzset{VertexStyle/.append style = {minimum size = 20pt,inner sep=0pt, shape=rectangle, fill=blue}}
  \end{tikzpicture}
  \caption{Illustration of a deadlock.
    The squares and the thick back-and-forth arcs represent the switches and the channels of the network, respectively.
    Each of the four colors illustrates the route of a packet.
    The gray channels remain unused.
    Each packet waits to be forwarded along the dashed arrow of its color to a channel blocked by another packet.}\label{fig:deadlock}
\end{figure}
A well-established method for avoiding deadlocks is \emph{up-down routing}~\cite{alfaro2000performance, sancho2000new, schroeder1991autonet, wittorff2009implementation}, which begins by selecting an acyclic orientation of the network graph.
Forwarding is then restricted: subpaths in which two consecutive edges are both directed into the same vertex are disallowed.

This restriction ensures that the resulting \emph{channel dependency graph} is acyclic, which in turn guarantees deadlock-free routing~\cite{dally1987deadlock}.
In this dependency graph, vertices represent channels, and there is a directed arc from channel $c_i$ to $c_j$ if a packet can move from $c_i$ to $c_j$.
A cycle in the channel dependency graph corresponds to a closed walk $W$ in the original network where all consecutive edges are traversable, that is, no two are oriented into the same vertex.
This, in turn, implies that $W$ is oriented cyclically in the chosen orientation, contradicting the assumption of acyclicity.
Therefore, no such $W$ exists, and the routing is deadlock-free.

The number of forbidden subpaths at a vertex depends solely on its indegree.
Thus, choosing an appropriate acyclic orientation is critical for minimizing the number of routing constraints.
For example, a \emph{dec-min} acyclic orientation minimizes the maximum number of forbidden subpaths at any vertex, while a $\min \sum_{v\in V} \binom{\varrho(v)}{2}$ orientation minimizes the total number of such subpaths.
The optimal solutions to this latter problem coincide with those that minimize $\sum_{v\in V} \varrho(v)^2$.
Theorem~\ref{thm:fixedConvexDichotomy} shows that minimizing the total number of forbidden subpaths is APX-hard even when the network graph is connected and simple.
This already holds when $\left\lceil \frac{|E|}{|V|}\right\rceil=2$.
Note that in real-world implementations, additional constraints --- such as rooted connectivity --- are often imposed to ensure that there exist feasible routes between certain pairs of switches~\cite{alfaro2000performance}.

\paragraph{Related work}
Several special cases of Problem~\ref{prob:sumPhi_VRhoV} have already been studied in the literature, for example,
if $\varphi_v$ is the $(f,g)$-lifting of the all-zero function for each $v \in V$, then we obtain the polynomial-time solvable indegree-bounded orientation problem~\cite{hakimi1965degrees}.
In contrast, the acyclic variant of this problem is known to be NP-hard~\cite{kiraly2018acyclic}.
Consequently, the acyclic version of Problem~\ref{prob:sumPhi_VRhoV} is NP-hard even if $\varphi_v$ is the $(f,g)$-lifting of the all-zero function for each $v\in V$.
More generally, the problem of finding an acyclic orientation in which the indegree of each vertex belongs to a prescribed list has been studied, with fixed-parameter tractability established for several parameterizations~\cite{garvardt2023finding}.
The unconstrained (possibly cyclic) counterparts of Problems~\ref{prob:sumPhiRhoV}-\ref{prob:incMax}, known as egalitarian orientation problems, along with graph orientations in general, have been extensively studied in the literature~\cite{borradaile2017egalitarian,frank1980orientation,frank2022decreasing2,frank2022decreasing1}.

A fundamental fairness objective is to minimize the maximum indegree in an (acyclic) orientation.
Graphs that admit an acyclic orientation with maximum indegree at most $k$ are called \emph{$k$\nobreakdash-degenerate}, a notion first introduced by Lick and White~\cite{lick1970k}.
The smallest such $k$ --- the \emph{degeneracy} of the graph --- is denoted by $\cev{d}_{\min}$.
A simple linear-time greedy algorithm can construct an acyclic orientation in which the indegree of every vertex is at most the degeneracy number~\cite{matula1983smallest}.
The problem of minimizing the maximum indegree has also been studied for orientations without the acyclicity constraint.
In this case, a simple augmenting-path approach yields a polynomial-time algorithm~\cite{asahiro2007graph,de1995regular,venkateswaran2004minimizing}.
However, when a weight function $w : E \to \R_+$ is introduced and the objective is to minimize the maximum \emph{weighted} indegree, the problem becomes NP-hard --- even when each weight is either $1$ or $k$ for some fixed $k \geq 2$~\cite{asahiro2011approximation, asahiro2007graph}.

A stricter notion of fairness is captured by the dec-min orientation and dec-min acyclic orientation problems introduced by Borradaile et al.~\cite{borradaile2017egalitarian}.
These problems seek orientations where the indegree vector sorted in non-increasing order is lexicographically minimal.
Notably, any dec-min acyclic orientation guarantees that the maximum indegree is at most the degeneracy number $\cev{d}_{\min}$.
They prove that an unconstrained (possibly cyclic) dec-min orientation can be found in polynomial time via a dipath-reversing method~\cite{borradaile2017egalitarian}, whereas the acyclic variant is NP-hard.
They also pose the problem of approximating the dec-min objective under the acyclicity constraint.

Frank and Murota~\cite{frank2022decreasing2,frank2022decreasing1} further investigated the dec-min orientation problem without the acyclicity constraint.
They proved that the optimal solutions to the dec-min and inc-max orientation problems coincide.
Moreover, these solutions also minimize the $\varphi$\nobreakdash-sum of indegrees for any discrete strictly convex function $\varphi: \Z_+ \to \R$.
Their framework leverages the theory of M-convex sets, providing a strongly polynomial-time algorithm for computing a dec-min element in such a set, particularly one arising from an integral base polyhedron --- extending well beyond graph orientations.

\paragraph{Notation}
Throughout this paper, we work with an undirected graph $G = (V, E)$.
Parallel edges are allowed, while loops are excluded unless explicitly stated otherwise.
For a vertex $v \in V$, its degree is denoted by $d_G(v)$, or simply $d(v)$ when the graph is clear from context.
If loops are present, then each loop contributes one to the degree of its incident vertex.
For a subset $V' \subseteq V$, we write $G[V']$ for the subgraph induced by $V'$. We write $d(v, V')$ for the degree of $v$ in $G[V'\cup\{v\}]$.
We denote a directed graph~(\emph{digraph} for short) by $D = (V, A)$.
The \emph{indegree} of a vertex $v$ in $D$ is denoted by $\varrho_D(v)$ (or simply $\varrho(v)$).
An \emph{ordering} of the vertices is denoted by $\sigma$, where $\sigma_i \in V$ is the vertex in position $i$.
For a vertex $v = \sigma_i$ in an ordering $\sigma$, the \emph{left-degree} is $\cev{d}_{\sigma}(v) = d(v, \{\sigma_1, \dots, \sigma_i\})$.
We often omit~$\sigma$ when the ordering is clear and write $\cev{d}(v)$.
When a weight function $w : E \to \R_+$ is given, we extend these notations with a superscript $w$ for the weighted degrees, for example, $d^w(v) = \sum_{e \in E : v \in e} w(e)$ denotes the \emph{weighted degree} of $v$.
An ordering $\sigma$ is called \emph{$k$\nobreakdash-bounded} if $\cev{d}(v) \leq k$ for all $v \in V$.
The \emph{degeneracy} of $G$, denoted by $\cev{d}_{\min}(G)$ or simply $\cev{d}_{\min}$, is the smallest $k$ for which a $k$\nobreakdash-bounded ordering exists.
A discrete function $\varphi : \Z_+\to \R$ is \emph{convex} if $\varphi(z+1) + \varphi(z-1) \geq 2 \varphi(z)$ holds for every positive integer $z$.
We say that $z$ is a \emph{breakpoint} of such a function $\varphi$ if the previous condition holds with strict inequality for $z$.
A discrete \emph{strictly convex} function is a discrete convex function that has a breakpoint at every positive integer $z$.
We use the notation $[x]^+ = \max\{x, 0\}$ for $x \in \R$.
Unless stated otherwise, running-time bounds involving real-valued objective functions use unit cost for evaluations of these functions and for arithmetic operations and comparisons involving their values.

\section{Separable convex optimization over all orientations}\label{sec:convexHullOfIndegrees}
For a graph $G = (V, E)$, let $P_{\varrho}(G)\subseteq\R^V$ denote the convex hull of the indegree vectors of all orientations of $G$.
It is known that $P_{\varrho}(G)$ is a base polyhedron and that its integer points are precisely the indegree vectors of orientations of $G$~\cite{frank2022decreasing2,frank2022decreasing1}.
We also use the standard fact that the corners of $P_{\varrho}(G)$ are precisely the indegree vectors of acyclic orientations of $G$.
Thus, the unconstrained problems studied here optimize over the integer points of $P_{\varrho}(G)$, whereas requiring acyclicity restricts the solution space to its corners.

For completeness, we record an elementary polynomial-size totally unimodular extended formulation of $P_{\varrho}(G)$ based on the bipartite incidence graph of $G$.

Let $G' = (V, E; E')$ be the bipartite incidence graph in which $E'$ contains an edge between $e \in E$ and $v \in V$ if $v$ is an endpoint of $e$ in $G$.
Let $A_{G'}$ denote the incidence matrix of this bipartite graph $G'$.
Extend the rows corresponding to $V$ with a $|V| \times |V|$ negated identity matrix, and the rows corresponding to $E$ with an $|E| \times |V|$ zero matrix.
Let $A$ denote the resulting matrix.
Since $A_{G'}$ is the incidence matrix of a bipartite graph, it is totally unimodular, and appending the negated identity columns preserves total unimodularity; hence $A$ is totally unimodular.
Let $y \in \R_+^{E'}$ be the variables corresponding to columns of $A_{G'}$, and $x \in \R_+^V$ the variables corresponding to the additional columns.
Define $b\in \{0,1\}^{V\cup E}$ such that $b(v)=0$ for $v\in V$ and $b(e)=1$ for $e\in E$.
This gives the following formulation.

\begin{claim}\label{cl:TUMatrixForIndegreeVectors}
  Consider the polytope $P = \{ (y,x) \in \R_+^{E'} \times \R_+^V : A(y,x) = b \}$.
  Then the projection
  $P|_x = \{ x \in \R_+^V : (y,x) \in P \text{ for some } y \in \R_+^{E'} \}$
  is the polytope $P_{\varrho}(G)$ of indegree vectors of orientations of $G = (V, E)$.
\end{claim}
\begin{proof}
  The edge equations imply $0 \leq y_{ev} \leq 1$ for every $ev \in E'$, and the vertex equations then bound $x$, so $P$ is bounded.
  Since $A$ is totally unimodular and $b$ is integral, $P$ is integral.
  At every integral point $(y,x) \in P$, the equation $y_{eu}+y_{ev}=1$ for an edge $e=uv$ selects exactly one of $u$ and $v$.
  Orient $e$ towards the selected endpoint.
  The equations indexed by $V$ then give $x_v=\varrho(v)$ for every $v\in V$.
  Conversely, every orientation of $G$ yields an integral point of $P$ by setting $y_{ev}=1$ exactly when $e$ is oriented towards $v$ and setting $x_v=\varrho(v)$.
  Thus, the integral points of $P$ project exactly to the indegree vectors of orientations of $G$.
  Since $P$ is bounded and integral, $P=\operatorname{conv}(P\cap\Z^{E'\cup V})$.
  Taking the $x$-projection and using the correspondence above gives $P|_x=P_{\varrho}(G)$.
\end{proof}

Meyer~\cite{meyer1977class} presented a method for minimizing a separable discrete convex objective $\sum_{i=1}^n \varphi_i(x_i)$ over $x\in\Z_+^n$ subject to $Ax=b$, where $A \in \Z^{m \times n}$ is totally unimodular and $b \in \Z^m$.
Applying this method to the extended formulation in Claim~\ref{cl:TUMatrixForIndegreeVectors}, with zero objective functions on the $y$-variables, solves Problem~\ref{prob:sumPhi_VRhoV} over unconstrained (possibly cyclic) orientations.
In contrast, Corollary~\ref{cor:minSumPhiStrictlyConvexNPH} shows that restricting the feasible indegree vectors to the corners of $P_{\varrho}(G)$, equivalently requiring the orientation to be acyclic, is NP-hard even when every vertex is assigned the same discrete strictly convex function $\varphi$.
Thus, the tractability provided by the totally unimodular formulation over all integer points of $P_{\varrho}(G)$ does not survive the restriction to its corners.

\medskip
In the rest of this section, we present an alternative, strongly polynomial-time algorithm for Problem~\ref{prob:sumPhi_VRhoV}, where the goal is to minimize $\sum_{v \in V} \varphi_v(\varrho(v))$ over the orientations of a graph $G$ for given discrete (not necessarily strictly) convex functions $\varphi_v$.
Our approach gives a reduction to a minimum-cost integral flow problem, for which a strongly polynomial algorithm is known~\cite{tardos1985strongly}.
Following an idea similar to Meyer’s, we linearize the objective by expressing each discrete convex function as a sum of its increments and leveraging that they are non-decreasing.
Our approach directly exploits the connection between orientations and flows and yields an explicit minimum-cost flow formulation.
\begin{theorem}\label{thm:nonAcyclicWithFlow}
  Let $G = (V, E)$ be a multigraph, and let $\varphi_v: \Z_+ \to \R$ be a discrete (not necessarily strictly) convex function for each $v \in V$.
  An orientation minimizing $\sum_{v \in V} \varphi_v(\varrho(v))$ can be computed in strongly polynomial time.
\end{theorem}
\begin{proof}
  We reduce the problem to a minimum-cost integral flow problem.
  Let $G = (V, E)$ denote the multigraph for which we want to find an optimal orientation, and let $V = \{v_1, \dots, v_n\}$ and $E = \{e_1, \dots, e_m\}$.
  If $\varphi_v(0) \neq 0$, then define $\varphi_v'(z) = \varphi_v(z) - \varphi_v(0)$ for $z \in \Z_+$.
  Since
  \[
    \sum_{v \in V} \varphi_v(\varrho(v)) = \sum_{v \in V} \varphi_v'(\varrho(v)) + \sum_{v \in V} \varphi_v(0),
  \]
  the set of optimal orientations remains unchanged.
  Thus, we may assume without loss of generality that $\varphi_v(0) = 0$ for all $v \in V$.

  Construct a network $D = (U, A)$ as follows.
  The vertex set $U$ consists of a source $s$, a sink $t$, a vertex $v'_i$ for each $v_i \in V$, and a vertex $e'_j$ for each $e_j \in E$.
  For each $v_i \in V$, add $d_G(v_i)$ parallel arcs from $s$ to $v'_i$.
  Add an arc from $v'_i$ to $e'_j$ if $e_j$ is incident to $v_i$ in $G$.
  Finally, add an arc from each $e'_j$ to $t$.
  Figure~\ref{fig:integralFlow} illustrates this construction.
  \begin{figure}[ht]
    \centering
    \begin{tikzpicture}[xscale=1.1,yscale=1.45]
      \tikzset{VertexStyle/.append style = {minimum size = 20pt,inner sep=0pt}}
      \SetVertexMath
      \Vertex[x=-3, y=0, L=s]{s}
      \Vertex[x=0, y=1.5, L=v'_1]{v1}
      \Vertex[x=0, y=0, L=v'_i]{vi}
      \Vertex[x=0, y=-1.5, L=v'_n]{vn}

      \Vertex[x=3, y=1.5, L=e'_1]{e1}
      \Vertex[x=3, y=0, L=e'_j]{ej}
      \Vertex[x=3, y=-1.5, L=e'_m]{em}
      \Vertex[x=6, y=0, L=t]{t}

      \node(d1) at (0,0.85) {$\vdots$};
      \node(d1) at (0,-0.7) {$\vdots$};
      \node(d1) at (3,0.85) {$\vdots$};
      \node(d1) at (3,-0.7) {$\vdots$};

      \draw[\niceArrow] (s) to [out=33,in=203,looseness=1] (v1);
      \draw[\niceArrow] (s) to [out=23,in=213,looseness=1] (v1);

      \draw[\niceArrow] (s) to [out=5,in=175,looseness=1] (vi);
      \draw[\niceArrow] (s) to [out=-5,in=-175,looseness=1] (vi);
      \path[circleAroundEdges={.08}{.20}{.5}] (s)--(v1) node [above=3mm, pos=.45] {$d(v_1)$};
      \path[circleAroundEdges={.08}{.20}{.5}] (s)--(vi) node [above=1.5mm, pos=.55] {$d(v_i)$};
      \path[circleAroundEdges={.08}{.20}{.5}] (s)--(vn) node [below=3mm, pos=.45] {$d(v_n)$};

      \draw[\niceArrow] (s) to [out=-33,in=-203,looseness=1] (vn);
      \draw[\niceArrow] (s) to [out=-23,in=-213,looseness=1] (vn);

      \draw[\niceArrow] (vi)--(ej);
      \node(if1) at (1.5,-0.3) {if $v_i$ and $e_j$};
      \node(if2) at (1.5,-0.6) {are incident};

      \draw[\niceArrow] (e1)--(t);
      \draw[\niceArrow] (ej)--(t);
      \draw[\niceArrow] (em)--(t);
      \draw[\niceArrow] (vi)--(ej);
    \end{tikzpicture}
    \caption{The digraph $D = (U, A)$ constructed in the reduction to the minimum cost flow problem.
      The vertices $v'_1,\dots,v'_n$ correspond to the vertices of $G$ and the vertices $e'_1,\dots,e'_m$ correspond to the edges of $G$.}\label{fig:integralFlow}
  \end{figure}
  Let the capacity of each arc be $1$.
  For each $i \in \{1, \dots, n\}$ and $\ell \in \{1, \dots, d(v_i)\}$, set the cost of the arc $a_\ell^i$ to $c(a_\ell^i) = \varphi_{v_i}(\ell) - \varphi_{v_i}(\ell - 1)$.
  Set the cost $c(a)$ of all other arcs to $0$.

  We now prove that the minimum cost of an integral $s$\nobreakdash-$t$ flow of amount $|E|$ is exactly $\sum_{v_i \in V} \varphi_{v_i}(\varrho(v_i))$ for an optimal orientation of $G$.
  First, observe that for each $i \in \{1, \dots, n\}$ and $z \in \{0, \dots, d(v_i)\}$, the total cost of sending $z$ units of flow along the first $z$ parallel arcs from $s$ to $v'_i$ is
  \[
    \sum_{\ell=1}^z c(a_\ell^i) = \sum_{\ell=1}^z \big(\varphi_{v_i}(\ell) - \varphi_{v_i}(\ell - 1)\big) = \varphi_{v_i}(z) - \varphi_{v_i}(0) = \varphi_{v_i}(z).
  \]

  On the one hand, consider an optimal orientation, and construct an $s$\nobreakdash-$t$ flow $f : A \to \Z_+$ of amount $|E|$ as follows.
  For each $i \in \{1, \dots, n\}$, set $f(a_\ell^i) = 1$ if $\ell \leq \varrho(v_i)$.
  For each $j \in \{1, \dots, m\}$, set $f(v'_i e'_j) = 1$ if $e_j$ is oriented toward $v_i$ in $G$.
  Set $f(e'_j t) = 1$ for all $j$.
  Set $f(a) = 0$ for all other arcs.
  At each $v'_i$, the number of incoming flow units equals the number of edges oriented toward $v_i$, and each $e'_j$ receives and sends exactly one unit, so $f$ is feasible.
  Its cost is $\sum_{v_i \in V} \varphi_{v_i}(\varrho(v_i))$, since it uses the first $\varrho(v_i)$ parallel $s v'_i$ arcs for each $i \in \{1,\dots,n\}$.

  On the other hand, consider a minimum-cost integral $s$\nobreakdash-$t$ flow $f : A \to \Z_+$ of value $|E|$.
  Since $f(e'_j t) = 1$ for all $j$, each $e'_j$ has exactly one incoming arc $v'_i e'_j$ with $f(v'_i e'_j) = 1$.
  We orient $e_j$ toward $v_i$ if $f(v'_i e'_j) = 1$.
  The indegree $\varrho(v_i)$ is thus the number of parallel $sv'_i$ arcs used by $f$.
  Since $\varphi_{v_i}$ is convex, the costs satisfy
  \[
    c(a_\ell^i) = \varphi_{v_i}(\ell) - \varphi_{v_i}(\ell - 1) \leq \varphi_{v_i}(\ell') - \varphi_{v_i}(\ell' - 1) = c(a_{\ell'}^i)
  \]
  whenever $\ell < \ell'$.
  Therefore, $f$ can be assumed to use the first $\varrho(v_i)$ parallel arcs from $s$ to $v'_i$ for each $i$.
  Thus, the cost of the flow is $\sum_{v_i \in V} \varphi_{v_i}(\varrho(v_i))$, matching the objective value of the corresponding orientation.

  This shows that an orientation minimizing $\sum_{v_i \in V} \varphi_{v_i}(\varrho(v_i))$ can be found by solving the minimum-cost integral $s$\nobreakdash-$t$ flow problem in the network above.
  The arc costs require $\sum_{v\in V}(d(v)+1)=2|E|+|V|$ function evaluations, and the network has size polynomial in $|V|+|E|$.
  A strongly polynomial minimum-cost flow algorithm~\cite{tardos1985strongly} therefore proves the theorem under the running-time convention above.
\end{proof}

Theorem~\ref{thm:nonAcyclicWithFlow} solves Problem~\ref{prob:sumPhi_VRhoV} and hence several special cases, including Problems~\ref{prob:sumPhiRhoV}-\ref{prob:incMax} without the acyclicity requirement, $\max \prod_{v \in V} (\varrho(v)+1)$, and $\min \sum_{v \in V} \left([\alpha(v) - \varrho(v)]^+ + [\varrho(v) - \beta(v)]^+\right)$ for every $\alpha,\beta : V \to \R$.
As already noted in the introduction, the self-generalizing nature of Problem~\ref{prob:sumPhi_VRhoV} immediately extends Theorem~\ref{thm:nonAcyclicWithFlow} to mixed graphs with lower and upper bounds for the indegrees of the vertices.

\section{General separable objectives over acyclic orientations}\label{sec:minSumPhi}
This section studies separable objectives over the vertices of the indegree polytope.
By the standard corner characterization recalled in Section~\ref{sec:convexHullOfIndegrees}, these vertices are precisely the indegree vectors of acyclic orientations, so this is equivalent to optimizing the same objectives over acyclic orientations of a graph $G=(V,E)$.
We first characterize the tractable case of vertex-specific linear objectives.
We then prove a dichotomy for minimizing the $\varphi$\nobreakdash-sum for fixed discrete convex functions $\varphi$, with a linear-time algorithm when $\varphi$ is linear on the positive integers and NP-hardness otherwise; under non-negativity, the hard side is APX-hard.
For every breakpoint $k\geq2$, the hard instances can be chosen connected and simple so that $k=\left\lceil \frac{|E|}{|V|}\right\rceil$.
We also give an example of a graph-dependent non-negative discrete strictly convex objective for which the problem admits a PTAS, and establish NP-hardness of exact optimization for arbitrary graph-dependent discrete convex functions $\varphi_G$ having any fixed breakpoint $k\geq2$.
Despite this hardness, we give a general strongly polynomial-time constant-factor approximation framework for broad classes of convex and concave objectives.
Finally, we present an exact exponential-time algorithm for arbitrary vertex-specific separable objectives and show that the exponential dependence on $|V|$ in the running time is unavoidable under the Exponential Time Hypothesis for every fixed discrete convex function on the hard side of the dichotomy.

From now on, we mostly adopt the vertex ordering perspective, that is, we seek a topological ordering of the desired orientation, as discussed in Section~\ref{sec:introduction}.

\subsection{Linear objectives}
We begin with the tractable linear case.
The problem is polynomial-time solvable when each vertex has its own linear objective function.
The standard corner characterization already implies that every vertex-specific linear objective has an acyclic optimum, since after omitting its constant term the objective is linear on $P_{\varrho}(G)$ and therefore has a minimizer at a corner.
We give a direct argument that yields an explicit optimal ordering.

\begin{theorem}\label{thm:minLinearFunctionAcyclic}
  Let $G = (V, E)$ be a loop-free multigraph and let $\varphi_v : \Z_+ \to \R$ be a discrete linear function for each $v \in V$.
  Then ordering the vertices in non-increasing order of the slopes of $\varphi_v$ minimizes $\sum_{v \in V} \varphi_v(\cev{d}(v))$.
\end{theorem}
\begin{proof}
  Let $a_v, b_v \in \R$ be such that $\varphi_v(z) = a_vz + b_v$ for every $v \in V, z \in \Z_+$.
  For an arbitrary order $\sigma$ of the vertices, define $\pi_\sigma(v) = b_v $ for each $v \in V$, and let
  \[
    \pi_\sigma(uv) =
    \begin{cases}
      a_v & \text{if $u$ precedes $v$ in $\sigma$},\\
      a_u & \text{otherwise}
    \end{cases}
  \]
  for each edge $uv \in E$.
  Then we can express the objective as
  \[
    \sum_{v \in V} \varphi_v(\cev{d}(v))
    = \sum_{v \in V} \left(a_v \cev{d}(v) + b_v\right)
    = \sum_{e \in E} \pi_\sigma(e) + \sum_{v \in V} \pi_\sigma(v).
  \]
  In any vertex order, for each edge $uv$, we have $\pi_\sigma(uv) \geq \min\{a_u, a_v\}$, with equality achieved if the vertex with the larger slope precedes the other.
  Thus, by ordering the vertices in non-increasing order of $a_v$, every edge contributes exactly $\min\{a_u,a_v\}$, yielding an optimal ordering.
\end{proof}

As an immediate corollary, if all vertices share the same linear function $\varphi$, then every vertex order attains the same objective value, and hence all orders are optimal.
This also implies that minimizing $\sum_{v\in V} \varrho(v)^2$ and $\sum_{v\in V} \binom{\varrho(v)}{2}$ are exactly the same problems, as we stated in relation to the routing application in Section~\ref{sec:introduction}.

\subsection{A complexity dichotomy for fixed discrete convex objectives}\label{sec:strictConvexApproximationHardness}\label{sec:minSumPhiComplexity}\label{sec:minSumPhiNPH}

The preceding theorem allows each vertex to have its own linear function.
In this section, we show that, when all vertices use the same discrete convex function $\varphi$, the problem remains tractable even if $\varphi(0)$ does not agree with the linear formula for the positive integers.
As the main result of this section, we now give the corresponding complexity dichotomy for every fixed discrete convex function.
For every non-empty graph $G=(V,E)$, each vertex order satisfies $\sum_{v\in V}\cev{d}(v)=|E|$, so $\frac{|E|}{|V|}$ is its average left-degree.
Hence $\left\lceil \frac{|E|}{|V|}\right\rceil$ is a lower bound on the maximum left-degree of every order, and therefore $\left\lceil \frac{|E|}{|V|}\right\rceil\leq\cev{d}_{\min}(G)$.
The theorem below shows that, for every breakpoint $k\geq2$, the hard instances can be chosen so that $k=\left\lceil \frac{|E|}{|V|}\right\rceil$.

\begin{theorem}\label{thm:fixedConvexDichotomy}
  Let $\varphi:\Z_+\to\R$ be a fixed discrete convex function.
  If there exist $a,b\in\R$ such that $\varphi(z)=az+b$ for every positive integer $z$, then an order minimizing the $\varphi$\nobreakdash-sum of left-degrees can be found in $O(|V|+|E|)$ time on loop-free multigraphs.
  Otherwise, for every breakpoint $k\geq2$ of $\varphi$, minimizing the $\varphi$\nobreakdash-sum of left-degrees over the vertex orders of connected simple graphs satisfying $k=\left\lceil \frac{|E|}{|V|}\right\rceil$ is NP-hard.
  If, in addition, $\varphi$ is non-negative, then for every such breakpoint $k$ the problem is APX-hard under the same restriction.
\end{theorem}

We prove the three parts separately.
Theorem~\ref{thm:positiveLinearSolvable} establishes the polynomial-time solvable case, Theorem~\ref{thm:breakpointKAPXH} establishes the APX-hard case, and Theorem~\ref{thm:graphDependentBreakpointNPH} proves the NP-hardness result even for graph-dependent objectives.

We begin with the polynomial-time solvable case when $\varphi$ is linear on the positive integers.

\begin{theorem}\label{thm:positiveLinearSolvable}
  Let $\varphi:\Z_+\to\R$ be a fixed discrete convex function for which there exist $a,b\in\R$ such that $\varphi(z)=az+b$ for every positive integer $z$.
  Then an order minimizing the $\varphi$\nobreakdash-sum of left-degrees can be found in $O(|V|+|E|)$ time on loop-free multigraphs.
\end{theorem}
\begin{proof}
  Put $c=\varphi(0)-b$.
  Since $\varphi(1)=a+b$ and $\varphi(2)=2a+b$, discrete convexity at $1$ gives $c=\varphi(0)+\varphi(2)-2\varphi(1)\geq0$.
  For a vertex order $\sigma$ of a loop-free multigraph $G=(V,E)$, let $Z(\sigma)$ denote the number of vertices with left-degree $0$.
  Since the sum of the left-degrees equals $|E|$, we have
  \[
    \sum_{v\in V}\varphi(\cev{d}_\sigma(v))
    =a|E|+b|V|+cZ(\sigma).
  \]
  Since $c\geq0$, any order minimizing $Z(\sigma)$ also minimizes the $\varphi$\nobreakdash-sum.
  Every connected component contributes at least one vertex to $Z(\sigma)$, namely its first vertex in the order.
  Conversely, choose a rooted spanning tree in each connected component and order the vertices so that every parent precedes its children.
  The root is then the only vertex with left-degree $0$ in its component.
  Hence the minimum value of $Z(\sigma)$ is the number of connected components, and an order attaining this value can be constructed in $O(|V|+|E|)$ time.
\end{proof}

We next focus on the APX-hard case of Theorem~\ref{thm:fixedConvexDichotomy}.
We first prove APX-hardness for non-negative functions when $2$ is a breakpoint, then show that adding a clique of $s$ new vertices adjacent to every original vertex increases every original left-degree by $s$ in an optimal order, allowing us to derive hardness for a breakpoint at any $k\geq2$ from the breakpoint-at-$2$ result.
Taking a fixed number of disjoint copies before adding the clique ensures that the same integer $k$ is also the ceiling of the average left-degree of the resulting connected graph.

\begin{theorem}\label{thm:fixedStrictlyConvexAPXHardness}
  For any fixed discrete convex function $\varphi : \Z_+ \to \R_+$ with a breakpoint at $2$, minimizing the $\varphi$\nobreakdash-sum of left-degrees over the vertex orders of connected simple graphs is APX-hard, even on instances $G = (V, E)$ satisfying $\left\lceil \frac{|E|}{|V|}\right\rceil = 2$.
\end{theorem}
\begin{proof}
  We reduce from Minimum Vertex Cover on $3$\nobreakdash-regular simple graphs, which is APX-complete~\cite{alimonti2000some}.
  Let $H = (U, F)$ be such an instance, and put $n = |U|$, $m = |F|$, and $\tau = \opt_{\mathrm{VC}}(H)$.
  If $n = 0$, the instance is trivial, so assume that $n > 0$; in particular, $\tau \geq 1$.
  Since $H$ is $3$\nobreakdash-regular, all $m$ edges are incident with any vertex cover of size $\tau$, and hence 
  \begin{equation}\label{eq:fixedConvexSourceBounds}
    m \leq 3\tau
    \qquad\text{and}\qquad
    n = \frac{2m}{3} \leq 2\tau.
  \end{equation}

  We first construct an auxiliary graph $G^\circ$ that contains loops.
  For each $u\in U$, take one copy $G_u$ of the five-vertex gadget shown in Figure~\ref{fig:fixedConvexAPXGadget}.
  The gadget consists of the path $t^u_1t^u_2t^u_3t^u_4$, a vertex $t^u_5$ adjacent to all four path vertices, and one loop at $t^u_1$.
  Assign one edge incident with $u$ to $t^u_4$ and the other two to $t^u_3$, thereby representing each edge $uv\in F$ by an edge between one of $\{t^u_3,t^u_4\}$ and one of $\{t^v_3,t^v_4\}$.
  The loop contributes one to the left-degree of its incident vertex.

  \begin{figure}[h]
    \centering
    \begin{tikzpicture}[scale=1]
      \centering
      \tikzset{VertexStyle/.append style = {minimum size = 20pt,inner sep=0pt}}
      \SetVertexMath
      \Vertex[x=-3, y=2, L=t^u_1]{t1}
      \Vertex[x=-1, y=2, L=t^u_2]{t2}
      \Vertex[x=1, y=2, L=t^u_3]{t3}
      \Vertex[x=3, y=2, L=t^u_4]{t4}
      \Vertex[x=0, y=0, L=t^u_5]{t5}

      \Vertex[x=.5, y=3.5, empty]{e1}
      \Vertex[x=1.5, y=3.5, empty]{e2}
      \Vertex[x=3, y=3.5, empty]{e3}

      \draw (t1) -- (t2) -- (t3) -- (t4);
      \draw (t5) -- (t1);
      \draw (t5) -- (t2);
      \draw (t5) -- (t3);
      \draw (t5) -- (t4);
      \draw (t1) edge[loop left, looseness=10] (t1);

      \draw[wavy] (t3) -- (e1);
      \draw[wavy] (t3) -- (e2);
      \draw[wavy] (t4) -- (e3);
    \end{tikzpicture}
    \caption{The five-vertex gadget $G_u$ used in the proof of Theorem~\ref{thm:fixedStrictlyConvexAPXHardness}.
      The loop is incident with $t^u_1$, and the wavy edges are the three non-gadget edges corresponding to the edges incident with $u$ in $H$.}\label{fig:fixedConvexAPXGadget}
  \end{figure}

  Fix an order $\pi$ of the vertices of $G^\circ$.
  For $u \in U$, say that a non-gadget edge enters $G_u$ if its endpoint in $G_u$ is later in $\pi$ than its other endpoint, and let $r_u$ be the number of non-gadget edges entering $G_u$.
  Every non-gadget edge is counted at its later endpoint, so
  \begin{equation}\label{eq:fixedConvexIncomingSum}
    \sum_{u \in U}r_u = m.
  \end{equation}
  Moreover,
  \begin{equation}\label{eq:fixedConvexExtractedCover}
    C_\pi = \{u \in U:r_u > 0\}
  \end{equation}
  is a vertex cover of $H$, because each edge of $F$ enters the gadget containing its later endpoint.

  Let $a_\varphi = 2\varphi(1)+3\varphi(2)$, $b_\varphi = \varphi(2)-\varphi(1)$, and  $c_{\varphi} = \varphi(1)-2\varphi(2)+\varphi(3)$.
  The breakpoint assumption gives $c_{\varphi} > 0$.

  \begin{claim}\label{cl:fixedConvexLocalGadget}
    Fix $u\in U$ and prescribe a set of $r\in\{0,1,2,3\}$ non-gadget edges incident with $G_u$ as entering.
    Among orders of $V(G_u)$, where each prescribed edge contributes one to the left-degree of its endpoint in $G_u$, the minimum objective value $\sum_{v\in V(G_u)}\varphi(\cev{d}(v))$ of the vertex orders of $G_u$ is
    \[
      \begin{cases}
        a_\varphi & \text{if }r=0,\\
        a_\varphi+b_\varphi r+c_{\varphi} & \text{if }r>0.
      \end{cases}
    \]
  \end{claim}
  \begin{subproof}
    Throughout this subproof, write $t_i=t^u_i$ for $i\in\{1,\dots,5\}$.
    Let $r_3\in\{0,1,2\}$ and $r_4\in\{0,1\}$ be the numbers of entering non-gadget edges at $t_3$ and $t_4$, respectively, so that $r=r_3+r_4$.
    The seven non-loop gadget edges, the loop at $t_1$, and the $r$ entering non-gadget edges each contribute once to the sum of the five left-degrees, so
    \begin{equation}\label{eq:fixedConvexLocalDegreeSum}
      \sum_{i=1}^5\cev{d}(t_i)=8+r.
    \end{equation}

    If $x+2\leq y$, then convexity implies $\varphi(x)+\varphi(y)-\varphi(x+1)-\varphi(y-1)\geq0$.
    Thus, transferring one unit from $y$ to $x$ does not increase the $\varphi$\nobreakdash-sum.
    Repeatedly applying this operation shows that, among non-negative integer vectors with a fixed sum, the minimum $\varphi$\nobreakdash-sum is attained by a vector whose entries differ by at most one.

    If $r=0$, balancing the five entries in~\eqref{eq:fixedConvexLocalDegreeSum} gives the multiset $1^2,2^3$, whose $\varphi$\nobreakdash-sum is $a_\varphi$.
    The order
    $
    (t_1,t_2,t_5,t_3,t_4)
    $
    attains this multiset.
    Now suppose that $r>0$.
    We first show that some gadget vertex has left-degree at least $3$.
    If $r_4=1$, then each gadget vertex is incident with at least three gadget edges, loops, or prescribed entering edges, counting the loop once.
    Since the other endpoint of every prescribed entering edge is earlier than its endpoint in $G_u$, the last gadget vertex has left-degree at least $3$.
    If $r_4=0$, then $r_3>0$.
    Each vertex in $\{t_1,t_2,t_3,t_5\}$ is incident with at least three gadget edges, loops, or prescribed entering edges whose other endpoint is either in this set or precedes its endpoint in $G_u$.
    Hence, the last vertex among these four has left-degree at least $3$.

    For a lower bound, ignore the gadget constraints and minimize the $\varphi$\nobreakdash-sum over non-negative integer vectors of length five, and sum $8+r$ having an entry of at least $3$.
    If an entry $y$ is at least $4$, then another entry $x$ is at most $1$, since otherwise the sum would be at least $12>8+r$.
    Transferring one unit from $y$ to $x$ does not increase the $\varphi$\nobreakdash-sum and preserves the existence of an entry of at least $3$.
    Repeating this operation yields a vector of no larger $\varphi$\nobreakdash-sum in which every entry is at most $3$ and at least one entry equals $3$.
    Keeping one entry equal to $3$ fixed and balancing the other four entries gives
    \[
      \begin{array}{c|c}
        r&\text{left-degree multiset}\\ \hline
        1&1^2,2^2,3\\
        2&1,2^3,3\\
        3&2^4,3.
      \end{array}
    \]
    The corresponding $\varphi$\nobreakdash-sum is $a_\varphi+b_\varphi r+c_{\varphi}$.
    The following orders attain the corresponding multisets for every possible pair $(r_3,r_4)$ with $r>0$, completing the proof of the claim:
    \[
      \begin{array}{c|c}
        (r_3,r_4)&\text{order}\\ \hline
        (1,0),(0,1)&(t_1,t_2,t_5,t_3,t_4)\\
        (1,1),(2,0),(2,1)&(t_3,t_4,t_5,t_1,t_2).
      \end{array}
    \]
  \end{subproof}

  Let
  \[
    F_\varphi^\circ(\pi) = \sum_{v \in V(G^\circ)}\varphi(\cev{d}_{\pi}(v)),
    \qquad \text{and} \qquad
    \opt_\varphi^\circ = \min_\pi F_\varphi^\circ(\pi).
  \]
  For each $u\in U$, the restriction of $\pi$ to $V(G_u)$ is a feasible local order in Claim~\ref{cl:fixedConvexLocalGadget} for the non-gadget edges entering $G_u$.
  Together with equations~\eqref{eq:fixedConvexIncomingSum} and~\eqref{eq:fixedConvexExtractedCover} and the inequality $|C_\pi| \geq \tau$, the claim gives
  \begin{equation}\label{eq:fixedConvexLoopedLowerBound}
    F_\varphi^\circ(\pi)
    \geq
    na_\varphi+b_\varphi m+c_{\varphi}|C_\pi|
    \geq
    na_\varphi+b_\varphi m+c_{\varphi}\tau.
  \end{equation}

  Let $C^*$ be a minimum vertex cover of $H$.
  Since $C^*$ is inclusion-minimal, every $u \in C^*$ has a neighbor in $U\setminus C^*$.
  Place all gadgets for $U\setminus C^*$ before all gadgets for $C^*$, with every gadget occupying a contiguous block.
  Since $U\setminus C^*$ is independent, every gadget indexed by $u \in U\setminus C^*$ has $r_u = 0$.
  Every $u \in C^*$ has a neighbor in the earlier block $U\setminus C^*$, so its gadget has $r_u > 0$.
  In each gadget, use an order attaining the objective value in Claim~\ref{cl:fixedConvexLocalGadget} for the entering edges determined by the block order.
  Since every gadget occupies a contiguous block, these internal orders do not change which non-gadget edges enter the gadget.
  Together with~\eqref{eq:fixedConvexLoopedLowerBound}, this proves
  \begin{equation}\label{eq:fixedConvexLoopedOptimum}
    \opt_\varphi^\circ
    = na_\varphi+b_\varphi m+c_{\varphi}\tau.
  \end{equation}
  Hence, every order $\pi$ satisfies
  \begin{equation}\label{eq:fixedConvexLoopedExcess}
    c_{\varphi}(|C_\pi|-\tau)
    \leq
    F_\varphi^\circ(\pi)-\opt_\varphi^\circ.
  \end{equation}

  We next eliminate the loops while preserving the difference from the optimum.
  Add a new vertex $z$ called \emph{anchor}. For each $u\in U$, remove the loop from $t^u_1$ and join $z$ to $t^u_1$ by one edge.
  For each $i\in\{0,1,2\}$, let $G^\circ_i$ denote the resulting graph after adding $i$ loops at $z$.

  \begin{claim}\label{cl:fixedConvexAnchorNormalization}
    For every fixed order $\pi$ of the vertices of $G^\circ$,
    \[
    \min_{\substack{\sigma\text{ an order of }G^\circ_2\\
                      \sigma|_{V(G^\circ)}=\pi}}
      \sum_{v\in V(G^\circ_2)}\varphi(\cev{d}_{\sigma}(v))
      =\varphi(2)+F_\varphi^\circ(\pi).
    \]
  \end{claim}
  \begin{subproof}
    Consider an order in which some neighbor of $z$ precedes $z$, and let $v=t^u_1$ be the rightmost such neighbor.
    Every vertex strictly between $v$ and $z$ is a non-neighbor of $z$, so moving $z$ leftward across these vertices changes no left-degree and preserves the relative order of all other vertices.
    Immediately before swapping the adjacent pair $v,z$, let $x$ and $y$ be their current left-degrees, respectively.
    The edge $vz$ does not contribute to the left-degree of $v$ before the swap, and the other two edges incident with $v=t^u_1$ for some $u\in U$ lie in $G_u$, so $x\leq2$.
    The two loops at $z$ and the edge $vz$ contribute to the left-degree of $z$, so $y\geq3$.
    Swapping $v$ and $z$ changes the pair of left-degrees from $(x,y)$ to $(x+1,y-1)$ and leaves all other left-degrees unchanged.
    Convexity gives
    $\varphi(x)+\varphi(y)-\varphi(x+1)-\varphi(y-1) \geq0$, because $y\geq3\geq x+1$.
    Repeating this operation moves $z$ before all of its neighbors without increasing the objective.
    It can then be moved to the first position across non-neighbors without changing any left-degree.

    When $z$ is first, its left-degree is $2$, and its edge to every $t^u_1$ reproduces the removed loop at that vertex.
    This extension has objective value $\varphi(2)+F_\varphi^\circ(\pi)$, and the preceding swaps show that no extension of $\pi$ has smaller value.
  \end{subproof}

  Claim~\ref{cl:fixedConvexAnchorNormalization} gives
  \begin{equation}\label{eq:fixedConvexAnchoredOptimum}
    \opt_\varphi(G^\circ_2)=\varphi(2)+\opt_\varphi^\circ.
  \end{equation}

  Take three disjoint copies of $G^\circ_0$ and make their anchors a clique.
  Denote the resulting graph by $G$.
  Each copy of $G^\circ_0$ is simple, and joining the anchors by a triangle preserves simplicity.
  In any order of $G$, index the copies so that the anchor $z_i$ of the $i^{\text{th}}$ copy has exactly $i$ earlier anchors, for $i\in\{0,1,2\}$.
  The $i$ edges from earlier anchors contribute $i$ to the left-degree of $z_i$, exactly as the $i$ loops in $G^\circ_i$, so the $\varphi$\nobreakdash-sum of left-degrees in the $i^{\text{th}}$ copy is the same as the $\varphi$\nobreakdash-sum of left-degrees in $G^\circ_i$ for the corresponding order.
  Conversely, in the $i^{\text{th}}$ copy, use the vertex order of an optimal order of $G^\circ_i$, and concatenate the three copies in increasing order of $i$.
  Then every copy attains its optimum simultaneously.
  Therefore,
  \begin{equation}\label{eq:fixedConvexLoopFreeOptimum}
    \opt_\varphi(G)=\sum_{i=0}^2\opt_\varphi(G^\circ_i).
  \end{equation}

  Let $\sigma$ be an arbitrary order of $G$, and consider the copy whose anchor is latest.
  Let $c_2$ be the $\varphi$\nobreakdash-sum of left-degrees in this copy, and let $\pi$ be the restriction of its order to the original gadget vertices.
  Claim~\ref{cl:fixedConvexAnchorNormalization} implies that $c_2\geq \varphi(2)+F_\varphi^\circ(\pi)$. Together with equation~\eqref{eq:fixedConvexAnchoredOptimum}, this yields
  $
    F_\varphi^\circ(\pi)-\opt_\varphi^\circ
    \leq c_2-\opt_\varphi(G^\circ_2).
  $
  The $\varphi$\nobreakdash-sum of left-degrees in the other copies is at least $\opt_\varphi(G^\circ_0)+\opt_\varphi(G^\circ_1)$.
  With~\eqref{eq:fixedConvexLoopFreeOptimum}, this gives
  \begin{equation}\label{eq:fixedConvexExcessPreservation}
    F_\varphi^\circ(\pi)-\opt_\varphi^\circ
    \leq
    \sum_{v\in V(G)}\varphi(\cev{d}_{\sigma}(v))
    -\opt_\varphi(G).
  \end{equation}
  Combining~\eqref{eq:fixedConvexLoopedExcess} and~\eqref{eq:fixedConvexExcessPreservation}, we obtain a vertex cover $C_\pi$ satisfying
  \begin{equation}\label{eq:fixedConvexCoverExcess}
    |C_\pi|-\tau
    \leq
    \frac{1}{c_{\varphi}}
    \left(
      \sum_{v\in V(G)}\varphi(\cev{d}_{\sigma}(v))
      -\opt_\varphi(G)
    \right).
  \end{equation}

  For every $i\in\{0,1,2\}$, placing the anchor first and then following an optimal order of $G^\circ$ gives an order of $G^\circ_i$ of objective value $\varphi(i)+\opt_\varphi^\circ$.
  Hence, with~\eqref{eq:fixedConvexLoopedOptimum} it implies
  $\opt_\varphi(G^\circ_i)\leq\varphi(i)+\opt_\varphi^\circ=\varphi(i)+(na_\varphi+b_\varphi m+c_{\varphi}\tau)$.
  Substituting this into~\eqref{eq:fixedConvexLoopFreeOptimum}, and using the non-negativity of $\varphi$, $a_\varphi\geq0$, $\tau\geq1$, and~\eqref{eq:fixedConvexSourceBounds}, gives
  \begin{equation}\label{eq:fixedConvexTargetOptimumBound}
    \opt_\varphi(G)
    \leq \sum_{i=0}^2\varphi(i)+3\bigl(na_\varphi+b_\varphi m+c_{\varphi}\tau\bigr)
    \leq \left(\sum_{i=0}^2\varphi(i)+6a_\varphi+9|b_\varphi|+3c_{\varphi}\right)\tau.
  \end{equation}
  Let $\lambda_\varphi
    =
    \sum_{i=0}^2\varphi(i)
    +6a_\varphi
    +9|b_\varphi|
    +3c_{\varphi}$.
  This is a positive constant depending only on $\varphi$.

  The construction uses $O(n)$ vertices and edges, and the cover $C_\pi$ can be extracted in polynomial time by scanning the edges representing $F$ in the copy with the latest anchor.
  Suppose that $\sigma$ is a $(1+\varepsilon)$\nobreakdash-approximate order of $G$.
  Equations~\eqref{eq:fixedConvexCoverExcess} and~\eqref{eq:fixedConvexTargetOptimumBound} give
  \[
    |C_\pi|
    \leq\tau+\frac{1}{c_{\varphi}}\left(\sum_{v\in V(G)}\varphi(\cev{d}_{\sigma}(v))-\opt_\varphi(G)\right)
    \leq\tau+\frac{\varepsilon}{c_{\varphi}}\opt_\varphi(G)
    \leq\left(1+\frac{\lambda_\varphi}{c_{\varphi}}\varepsilon\right)\tau.
  \]
  Thus, a $(1+\varepsilon)$\nobreakdash-approximation for the $\varphi$\nobreakdash-objective yields a $\left(1+\frac{\lambda_\varphi}{c_{\varphi}}\varepsilon\right)$\nobreakdash-approximation for Minimum Vertex Cover on $3$\nobreakdash-regular graphs.
  Given $\delta\in(0,1)$, choosing $\varepsilon\leq\frac{c_{\varphi}}{\lambda_\varphi}\delta$ would therefore yield a $(1+\delta)$\nobreakdash-ap\-prox\-i\-ma\-tion for Minimum Vertex Cover.
  The fixed linear relation between the approximation errors, together with the polynomial-time construction and extraction, gives a PTAS-reduction from Minimum Vertex Cover on $3$\nobreakdash-regular graphs, proving APX-hardness.

  It remains to verify that $\left\lceil \frac{|E(G)|}{|V(G)|}\right\rceil=2$.
  The graph $G^\circ$ has $5n$ vertices and $7n+n+m=\frac{19n}{2}$ edges, where loops are also counted.
  Replacing the loops at the vertices $t^u_1$ by the anchor edges preserves this edge count in each copy of $G^\circ_0$, while the anchor triangle adds three edges.
  Hence, $|V(G)|=15n+3$ and
    $|E(G)|=\frac{57n}{2}+3$.
  Since $n>0$,
    $|V(G)|<|E(G)|<2|V(G)|$,
  and hence $\left\lceil \frac{|E(G)|}{|V(G)|}\right\rceil=2$.
  Each anchor is adjacent to the vertices $t^u_1$ in its copy, every gadget $G_u$ is connected, and the three anchors form a triangle.
  Thus $G$ is connected.
\end{proof}

To derive APX-hardness for a breakpoint at an arbitrary integer $k\geq2$ from the preceding breakpoint-at-$2$ theorem, we need to increase every original left-degree by $k-2$.
The following lemma gives the required relation between the two optimization problems.

\begin{lemma}\label{lem:cliqueShift}
  Let $G=(V,E)$ be a simple graph, let $s\in\Z_+$, and let $G_s$ be obtained from $G$ by adding a clique $S$ of $s$ new vertices and joining every vertex of $S$ to every vertex of $V$.
  Let $\varphi:\Z_+\to\R$ be a discrete convex function, define $\widetilde{\varphi}(z)=\varphi(z+s)$ for $z\in\Z_+$, and put $A_{\varphi,s}=\sum_{i=0}^{s-1}\varphi(i)$, where $A_{\varphi,0}=0$.
  Then
  \[
    \opt_\varphi(G_s)=A_{\varphi,s}+\opt_{\widetilde{\varphi}}(G).
  \]
  Moreover, every vertex order of $G_s$ can be transformed, without increasing its $\varphi$\nobreakdash-sum, into an order in which all vertices of $S$ precede all vertices of $V$, while preserving the relative order of the vertices in $V$.
\end{lemma}
\begin{proof}
  Consider an order in which a vertex $v\in V$ immediately precedes a vertex $x\in S$, and let $r$ be the number of vertices preceding $v$.
  Write $a=\cev{d}(v)$ before swapping $v$ and $x$.
  Since $G_s$ is simple, $a\leq r$, while the vertex $x$ is adjacent to every preceding vertex and to $v$, so $\cev{d}(x)=r+1$.
  Swapping $v$ and $x$ changes their left-degrees from $(a,r+1)$ to $(a+1,r)$ and leaves every other left-degree unchanged.
  Discrete convexity gives $\varphi(a)+\varphi(r+1)\geq\varphi(a+1)+\varphi(r)$,
  because $a\leq r$.
  Hence the swap does not increase the $\varphi$\nobreakdash-sum.
  Repeating such swaps yields an order in which all vertices of $S$ precede all vertices of $V$, without changing the relative order of $V$.

  In the resulting order, the left-degrees of the vertices of $S$ are $0,1,\dots,s-1$.
  If $\sigma$ is the restriction of the order to $V$, then every $v\in V$ has left-degree $s+\cev{d}_\sigma(v)$ in $G_s$.
  Therefore the objective value is
    $A_{\varphi,s}+\sum_{v\in V}\widetilde{\varphi}(\cev{d}_\sigma(v))$.
  Prepending the vertices of $S$ in any order to an optimal order of $G$ for $\widetilde{\varphi}$ gives
  $\opt_\varphi(G_s)\leq A_{\varphi,s}+\opt_{\widetilde{\varphi}}(G)$.
  Conversely, applying the preceding swaps to an arbitrary order of $G_s$ produces an order with all vertices of $S$ first and therefore gives the reverse inequality.
\end{proof}

The next consequence records how to preserve the objective relation while prescribing the ceiling of the average left-degree after adding the clique.
\begin{corollary}\label{cor:cliqueShiftDensity}
  Let $G=(V,E)$ be a simple graph satisfying $\left\lceil \frac{|E|}{|V|}\right\rceil=2$, let $s$ be a positive integer, and let $t$ be an integer satisfying $t>\frac{s(s+3)}{2}$.
  Let $H$ be the disjoint union of $t$ copies of $G$, and let $H_s$ be obtained from $H$ by adding a clique $S$ of $s$ new vertices and joining every vertex of $S$ to every vertex of $H$.
  Then $H_s$ is a connected simple graph satisfying
    $\left\lceil \frac{|E(H_s)|}{|V(H_s)|}\right\rceil=s+2$.
  Moreover, for every discrete convex function $\varphi:\Z_+\to\R$, with $\widetilde{\varphi}(z)=\varphi(z+s)$ and $A_{\varphi,s}=\sum_{i=0}^{s-1}\varphi(i)$,
  \[
    \opt_\varphi(H_s)=A_{\varphi,s}+t\opt_{\widetilde{\varphi}}(G).
  \]
  Every vertex order of $H_s$ can be transformed, without increasing its $\varphi$\nobreakdash-sum, into an order with $S$ first, while preserving the relative order of the vertices in $V(H)$.
  After this transformation, the restriction to each copy of $G$ contributes its $\widetilde{\varphi}$\nobreakdash-sum independently.
\end{corollary}
\begin{proof}
  We have $|V(H_s)|=t|V|+s$ and $|E(H_s)|=t|E|+st|V|+\binom{s}{2}$.
  Since $|E|-|V|\geq1$ and $t>\frac{s(s+3)}{2}$,
  $|E(H_s)|-(s+1)|V(H_s)|
    =t(|E|-|V|)-\frac{s(s+3)}{2}>0$.
  
  Also, $2|V|-|E|\geq0$, and hence
  $(s+2)|V(H_s)|-|E(H_s)|
    =t(2|V|-|E|)+\frac{s(s+5)}{2}>0$.
  Therefore $(s+1)|V(H_s)|<|E(H_s)|<(s+2)|V(H_s)|$, which gives the stated value of $\left\lceil \frac{|E(H_s)|}{|V(H_s)|}\right\rceil$.
  The graph is simple, and it is connected because every vertex of $H$ is adjacent to every vertex of the non-empty clique $S$.

  Apply Lemma~\ref{lem:cliqueShift} to the graph $H$.
  Since the $t$ copies of $G$ are pairwise disjoint, $\opt_{\widetilde{\varphi}}(H)=t\opt_{\widetilde{\varphi}}(G)$, which gives the displayed equality.
  The final statement follows from the order transformation in Lemma~\ref{lem:cliqueShift} and the absence of edges between distinct copies of $G$.
\end{proof}

We are now ready to prove APX-hardness for any fixed non-negative discrete convex function $\varphi$ which is not linear on the positive integers.
\begin{theorem}\label{thm:breakpointKAPXH}
  Let $\varphi : \Z_+ \to \R_+$ be a fixed discrete convex function that is not linear on the positive integers.
  For every breakpoint $k\geq2$ of $\varphi$, minimizing the $\varphi$\nobreakdash-sum of left-degrees over the vertex orders of connected simple graphs satisfying $\left\lceil \frac{|E|}{|V|}\right\rceil=k$ is APX-hard.
\end{theorem}
\begin{proof}

  Let $k$ be a breakpoint of $\varphi$, put $s=k-2$, and define $\widetilde{\varphi}(z)=\varphi(z+s)$ for $z\in\Z_+$.
  Then $\widetilde{\varphi}$ is non-negative, discrete convex, and has a breakpoint at $2$.
  By Theorem~\ref{thm:fixedStrictlyConvexAPXHardness}, minimizing the $\widetilde{\varphi}$\nobreakdash-sum of left-degrees is APX-hard on connected simple graphs satisfying $\left\lceil \frac{|E|}{|V|}\right\rceil=2$.
  Let $G=(V,E)$ be an instance of this problem.
  If $s=0$, then $\varphi=\widetilde{\varphi}$, and the desired hardness follows directly.
  
  Suppose now that $s\geq1$.
  We first argue that if $\opt_{\widetilde{\varphi}}(G)=0$, then the instance can be solved in polynomial time, while if $\opt_{\widetilde{\varphi}}(G)>0$, the optimum admits a uniform positive lower bound.
  If $\widetilde{\varphi}(0)>0$, every order has a first vertex of left-degree $0$, so $\opt_{\widetilde{\varphi}}(G)\geq\widetilde{\varphi}(0)$ and the optimum is never zero.
  Suppose therefore that $\widetilde{\varphi}(0)=0$.
  Convexity and non-negativity then imply that $\widetilde{\varphi}$ is non-decreasing.
  Since $2$ is a breakpoint, there is a least $r\in\{1,2,3\}$ with $\widetilde{\varphi}(r)>0$.
  Hence $\opt_{\widetilde{\varphi}}(G)>0$ implies $\opt_{\widetilde{\varphi}}(G)\geq\widetilde{\varphi}(r)$.
  Moreover, $\opt_{\widetilde{\varphi}}(G)=0$ if and only if $G$ has an $(r-1)$\nobreakdash-bounded ordering, equivalently if and only if $\cev{d}_{\min}(G)\leq r-1$, and such an ordering can be found in polynomial time~\cite{matula1983smallest}.
  Thus, there is a constant $\alpha_{\widetilde{\varphi}}>0$, depending only on $\widetilde{\varphi}$, such that $\opt_{\widetilde{\varphi}}(G)>0$ implies $\opt_{\widetilde{\varphi}}(G)\geq\alpha_{\widetilde{\varphi}}$, while instances with $\opt_{\widetilde{\varphi}}(G)=0$ can be solved in polynomial time.

  Now suppose that $\opt_{\widetilde{\varphi}}(G)>0$.
  Fix an integer $t>\frac{s(s+3)}{2}$, which depends only on the fixed function $\varphi$ and the chosen breakpoint $k$.
  Since $\left\lceil \frac{|E|}{|V|}\right\rceil=2$, construct $H_s$ from $t$ copies of $G$ and a clique $S$ of $s$ new vertices as in Corollary~\ref{cor:cliqueShiftDensity}.
  Let $\pi$ be a $(1+\varepsilon)$\nobreakdash-approximate order of $H_s$ for the $\varphi$\nobreakdash-sum objective, and move the clique to the first $s$ positions without increasing the objective.
  Let $\sigma$ denote the resulting $(1+\varepsilon)$\nobreakdash-approximate order, and let $\sigma^i$ be its restriction to the $i^{\text{th}}$ copy of $G$.
  Corollary~\ref{cor:cliqueShiftDensity} gives
  \begin{equation}\label{eq:sumOfDifferences}
    \sum_{i=1}^t\left(
      \sum_{v\in V(G)}\widetilde{\varphi}(\cev{d}_{\sigma^i}(v))-\opt_{\widetilde{\varphi}}(G)
    \right)=\sum_{v\in V(H_s)}\varphi(\cev{d}_{\sigma}(v))-\opt_{\varphi}(H_s)
    \leq\varepsilon\bigl(A_{\varphi,s}+t\opt_{\widetilde{\varphi}}(G)\bigr),
  \end{equation}
  where $A_{\varphi,s}=\sum_{i=0}^{s-1}\varphi(i)\geq0$, since $\varphi$ is non-negative.
  Choose an index $i$ for which the $\widetilde{\varphi}$\nobreakdash-objective value of $\sigma^i$ is minimum among the $t$ copies.
  Equation~\eqref{eq:sumOfDifferences}, $A_{\varphi,s}\geq 0$, and $\opt_{\widetilde{\varphi}}(G)\geq\alpha_{\widetilde{\varphi}}$ imply
  \[
    \sum_{v\in V(G)}\widetilde{\varphi}(\cev{d}_{\sigma^i}(v))-\opt_{\widetilde{\varphi}}(G)
    \leq\varepsilon\left(\frac{A_{\varphi,s}}{t}+\opt_{\widetilde{\varphi}}(G)\right)
    \leq\varepsilon\left(1+\frac{A_{\varphi,s}}{t\alpha_{\widetilde{\varphi}}}\right)\opt_{\widetilde{\varphi}}(G).
  \]
  Given $\delta\in(0,1)$, set $\varepsilon=\delta/(1+\frac{A_{\varphi,s}}{t\alpha_{\widetilde{\varphi}}})$.
  The preceding inequality shows that a $(1+\varepsilon)$\nobreakdash-approximate order of $H_s$ for the $\varphi$\nobreakdash-sum objective yields a $(1+\delta)$\nobreakdash-approximate order of $G$ for the $\widetilde{\varphi}$\nobreakdash-sum objective.
  Together with the polynomial-time solution of instances with $\opt_{\widetilde{\varphi}}(G)=0$, the polynomial-time construction and extraction give a PTAS-reduction from the APX-hard $\widetilde{\varphi}$\nobreakdash-sum problem, proving the claimed APX-hardness.
\end{proof}

In particular, Problem~\ref{prob:sumPhiRhoV} is APX-hard for every fixed non-negative discrete strictly convex function $\varphi$.
Since every integer $k\geq2$ is a breakpoint of such a function, for every fixed $k\geq2$ this remains true on connected simple graphs satisfying $\left\lceil \frac{|E|}{|V|}\right\rceil=k$.
The function $\varphi(z)=\binom{z}{2}$ is non-negative and discrete strictly convex, so minimizing the total number of forbidden subpaths in the routing application from Section~\ref{sec:introduction} is APX-hard even on connected simple graphs.
For every fixed integer $k\geq2$, the non-negative discrete convex functions $\varphi(z)=|z-k|$ and $\varphi(z)=[z-k]^+$ have a breakpoint at $k$.
Therefore minimizing $\sum_{v\in V}|\cev{d}(v)-k|$ and minimizing $\sum_{v\in V}[\cev{d}(v)-k]^+$ are APX-hard even on connected simple graphs satisfying $\left\lceil \frac{|E|}{|V|}\right\rceil=k$.
For every graph $G=(V,E)$, every vertex order, and every integer $k$,
\[
  \sum_{v\in V}|\cev{d}(v)-k|
  =2\sum_{v\in V}[\cev{d}(v)-k]^+ + k|V|-|E|.
\]
Thus, for fixed $G$ and $k$, these two objectives have the same optimal orders.
On the hard instances above, $k$ is the ceiling of the average indegree, so the same reductions show that minimizing $\sum_{v\in V}\left|\cev{d}(v)-\left\lceil \frac{|E|}{|V|}\right\rceil\right|$ and minimizing $\sum_{v\in V}\left[\cev{d}(v)-\left\lceil \frac{|E|}{|V|}\right\rceil\right]^+$ are APX-hard.
For the positive-part objective, an optimum of zero is equivalent to the existence of an acyclic orientation with maximum indegree at most $k$.
For any prescribed integer $k$, this can be decided by computing the degeneracy~\cite{matula1983smallest}, and more generally one can decide in polynomial time whether there exists an acyclic orientation with $\varrho(v)\leq g(v)$ for every $v\in V$~\cite{kiraly2018acyclic}.

We next show that the APX-hardness statement in Theorem~\ref{thm:fixedConvexDichotomy} does not extend to arbitrary graph-dependent discrete convex objectives.
Let $\varphi_G:\Z_+\to\R_+$ now be allowed to depend on the input graph $G=(V,E)$.
Consider the function $\varphi_G(z)=\frac{z^2}{(m+1)^2}+z$,
where $m=|E|$, and let $\opt_{\varphi_G}(G)=\min_{\sigma}\sum_{v\in V}\varphi_G(\cev{d}_{\sigma}(v))$.
Its discrete second difference equals $\frac{2}{(m+1)^2}>0$ at every positive integer, so $\varphi_G$ is non-negative and discrete strictly convex.
Since $\varphi_G(z)\geq z$ for every $z\in\Z_+$, every vertex order satisfies
\[
\sum_{v\in V}\varphi_G(\cev{d}(v))
\geq \sum_{v\in V}\cev{d}(v)
=m.
\]
On the other hand, using $\sum_{v\in V}\cev{d}(v)^2 \leq m\sum_{v\in V}\cev{d}(v) = m^2$ gives the upper bound
\[
\sum_{v\in V}\varphi_G(\cev{d}(v))
=\frac{1}{(m+1)^2}\sum_{v\in V}\cev{d}(v)^2+m
\leq m+\frac{m^2}{(m+1)^2}
<m+1.
\]
These imply that $
m\leq \opt_{\varphi_G}(G)<m+1$.
We can therefore obtain a PTAS as follows. Fix any constant $\varepsilon>0$. If $m\geq \frac{1}{\varepsilon}$, then an arbitrary vertex order has objective value at most $m+1\leq (1+\varepsilon)m\leq (1+\varepsilon)\opt_{\varphi_G}(G)$, and hence is a $(1+\varepsilon)$-approximate solution. Otherwise, if $m<\frac{1}{\varepsilon}$, we can compute an optimal order in polynomial time by placing the isolated vertices at the first positions and enumerating all possible orders of the remaining vertices on the remaining positions. There are at most $2m<\frac{2}{\varepsilon}$ such vertices, so the enumeration takes constant time for fixed $\varepsilon$.
Thus minimizing the $\varphi_G$\nobreakdash-sum of left-degrees for $\varphi_G(z)=\frac{z^2}{(m+1)^2}+z$ admits a PTAS.
Hence, unless $\mathrm{P}=\mathrm{NP}$, APX-hardness cannot hold for every graph-dependent non-negative discrete strictly convex objective.
Nevertheless, exact optimization remains NP-hard for arbitrary graph-dependent discrete convex functions with a prescribed breakpoint, even without imposing non-negativity.
\begin{theorem}\label{thm:graphDependentBreakpointNPH}
  Fix an integer $k\geq2$.
  For every loop-free multigraph $G=(V,E)$, let $\varphi_G:\Z_+\to\R$ be an arbitrary discrete convex function with a breakpoint at $k$.
  Then minimizing the $\varphi_G$\nobreakdash-sum of left-degrees over the vertex orders of $G$ is NP-hard, even when restricted to connected simple graphs satisfying $\left\lceil \frac{|E|}{|V|}\right\rceil=k$.
\end{theorem}
\begin{proof}
  We first consider $k=2$.
  We reduce from Minimum Vertex Cover on $3$\nobreakdash-regular simple graphs, which is NP-hard~\cite{alimonti2000some}.
  Let $H=(U,F)$ be such an instance, and repeat the construction and notation from the proof of Theorem~\ref{thm:fixedStrictlyConvexAPXHardness}.
  The construction of $G$ depends only on $H$, and all statements through~\eqref{eq:fixedConvexCoverExcess} use only convexity and the breakpoint at $2$, so they remain valid with $\varphi$ replaced by the graph-dependent function $\varphi_G$.
  If $\sigma$ is an optimal order, then the right-hand side of~\eqref{eq:fixedConvexCoverExcess} is $0$, and hence $|C_\pi|=\tau$.
  Thus, an optimal order yields a minimum vertex cover of $H$.
  The final paragraph of the proof of Theorem~\ref{thm:fixedStrictlyConvexAPXHardness} shows that the constructed graph is connected and satisfies $|V(G)|<|E(G)|<2|V(G)|$.

  Now suppose that $k\geq3$, put $s=k-2$, and fix an integer $t>\frac{s(s+3)}{2}$.
  For every graph $G$, let $T(G)$ be the graph obtained from $t$ disjoint copies of $G$ by adding a clique of $s$ new vertices adjacent to every vertex in every copy, and define $\widetilde{\varphi}_G(z)=\varphi_{T(G)}(z+s)$ for $z\in\Z_+$.
  Since $k$ is a breakpoint of $\varphi_{T(G)}$, the function $\widetilde{\varphi}_G$ is discrete convex and has a breakpoint at $2$.
  Applying the preceding $k=2$ reduction to the graph-dependent family $\widetilde{\varphi}_G$ gives NP-hard instances $G$ satisfying $|V(G)|<|E(G)|<2|V(G)|$.
  Corollary~\ref{cor:cliqueShiftDensity} shows that $T(G)$ is connected and simple and satisfies $\left\lceil \frac{|E(T(G))|}{|V(T(G))|}\right\rceil=k$.
  If an optimal order of $T(G)$ for $\varphi_{T(G)}$ is transformed so that the added clique comes first, the restrictions to the $t$ copies of $G$ have total $\widetilde{\varphi}_G$\nobreakdash-objective value $t\opt_{\widetilde{\varphi}_G}(G)$.
  Each restriction is therefore optimal for $\widetilde{\varphi}_G$ on $G$.
  Thus an optimal order of $T(G)$ yields an optimal order for the NP-hard $\widetilde{\varphi}_G$\nobreakdash-instance $G$.
\end{proof}

Taking $\varphi_G=\varphi$ for every graph $G$ in Theorem~\ref{thm:graphDependentBreakpointNPH} establishes the remaining NP-hardness case of Theorem~\ref{thm:fixedConvexDichotomy}, and hence completes its proof.

Theorem~\ref{thm:graphDependentBreakpointNPH} yields the following two natural consequences concerning important special cases.

\begin{corollary}\label{cor:minSumPhiStrictlyConvexNPH}
  Fix an integer $k\geq2$.
  For every loop-free multigraph $G=(V,E)$, let $\varphi_G : \Z_+ \to \R$ be an arbitrary discrete strictly convex function.
  Then minimizing the $\varphi_G$\nobreakdash-sum of left-degrees over the vertex orders of $G$ is NP-hard, even when restricted to connected simple graphs satisfying $\left\lceil \frac{|E|}{|V|}\right\rceil=k$.
\end{corollary}
\noindent
This immediately implies that Problem~\ref{prob:sumPhiRhoV} is NP-hard for arbitrary graph-dependent discrete strictly convex functions, including functions that may take both positive and negative values.
For every fixed integer $k\geq2$, the hardness holds even on connected simple graphs satisfying $\left\lceil \frac{|E|}{|V|}\right\rceil=k$.
Later we show, in Claim~\ref{cl:decMinIncMaxAreSpecial}, that Problems~\ref{prob:decMin} and~\ref{prob:incMax} correspond to the choices $\varphi_G(z)=(|V|+1)^z$ and $\varphi_G(z)=(|V|+1)^{-z}$, respectively, hence they are also NP-hard.
Moreover, choosing $\varphi_G(z)=-\log(z+1)$ yields the NP-hardness of maximizing $\prod_{v \in V} (\cev{d}(v)+1)$.

\begin{corollary}\label{cor:oneBreakPointSumPhiLooplessNPH}
  Fix an integer $k\geq2$.
  For every loop-free multigraph $G=(V,E)$, let $\varphi_G : \Z_+ \to \R$ be an arbitrary discrete convex function such that $k$ is its unique breakpoint.
  Then minimizing the $\varphi_G$\nobreakdash-sum of left-degrees over the vertex orders of $G$ is NP-hard, even when restricted to connected simple graphs satisfying $\left\lceil \frac{|E|}{|V|}\right\rceil=k$.
\end{corollary}
\noindent

These corollaries sharpen the contrast with Section~\ref{sec:convexHullOfIndegrees}.
The unrestricted separable convex problem is strongly polynomial-time solvable over all integer points of $P_{\varrho}(G)$, while the corner-restricted problem can be NP-hard, as the two corollaries show.

\subsection{A general approximation framework}\label{sec:approximationAlgorithms}
Despite the hardness established above, broad classes of convex and concave objectives admit strongly polynomial-time constant-factor approximation algorithms.
For $C > 0$, a feasible solution to a minimization problem is $C$-approximate if its objective value is at most $C$ times the optimum.
This convention permits $C < 1$ for non-positive objectives.
For a maximization problem with a non-negative objective and $C \geq 1$, we use the standard convention that a $C$\nobreakdash-approximate solution has objective value at least $\frac{1}{C}$ times the optimum.
In either case, $C = 1$ means exact optimization.

\begin{theorem}\label{thm:weightedEndpointApproximation}
  Let $\lambda \in [0,1]$ and $C > 0$, and let $\varphi : \Z_+ \to \R$ be a fixed discrete convex function.
  Suppose that
  \begin{equation}\label{eq:weightedEndpointCondition}
    \lambda\varphi(a+c)+(1-\lambda)\varphi(b+d)\leq C\varphi(a+b)
  \end{equation}
  for all $a,b,c,d \in \Z_+$ satisfying $c\leq a$ and $b\leq d$.
  Then minimizing the $\varphi$\nobreakdash-sum of left-degrees over the vertex orders of loop-free multigraphs admits a strongly polynomial-time $C$\nobreakdash-approximation algorithm.

  If~\eqref{eq:weightedEndpointCondition} holds for all $a,b,c,d \in \Z_+$ without the restrictions $c\leq a$ and $b\leq d$, then the order with smaller $\varphi$\nobreakdash-sum among any vertex order and its reverse is a $C$\nobreakdash-approximation.
  In this case, convexity is not required.
  Furthermore, the two orders and their left-degrees can be found in $O(|V|+|E|)$ time, and their objective values require $O(|V|)$ function evaluations.
\end{theorem}
\begin{proof}
  Let $G = (V, E)$ be a loop-free multigraph.
  Theorem~\ref{thm:nonAcyclicWithFlow} gives, in strongly polynomial time, an orientation $D = (V, A)$ minimizing $\sum_{v \in V}\varphi(\varrho_D(v))$.

  For a vertex order $\sigma'$, let $q_D(\sigma')$ denote the number of arcs of $D$ whose tail precedes their head in $\sigma'$.
  Starting with an arbitrary order, repeatedly move a vertex to the first or last position whenever this strictly increases $q_D$.
  Since $q_D$ is an integer between $0$ and $|E|$, the procedure terminates after at most $|E|$ moves.
  Let $\sigma$ be the resulting order and let $\bar{\sigma}$ denote its reverse.

  Fix $v \in V$.
  Let $a_v$ and $c_v$ be the numbers of arcs entering and leaving $v$, respectively, whose other endpoint precedes $v$ in $\sigma$; and let $b_v$ and $d_v$ be the numbers of arcs entering and leaving $v$, respectively, whose other endpoint follows $v$ in $\sigma$.
  Parallel arcs are counted with multiplicity.
  Then
  \begin{equation}\label{eq:weightedEndpointDecomposition}
    \varrho_D(v)=a_v+b_v,\qquad \cev{d}_{\sigma}(v)=a_v+c_v,\qquad \cev{d}_{\bar{\sigma}}(v)=b_v+d_v.
  \end{equation}
  Moving $v$ to the first position changes $q_D$ by $c_v-a_v$, while moving it to the last position changes $q_D$ by $b_v-d_v$.
  The choice of $\sigma$ therefore implies $c_v\leq a_v$ and $b_v\leq d_v$.

  Condition~\eqref{eq:weightedEndpointCondition} and~\eqref{eq:weightedEndpointDecomposition} give $\lambda\varphi(\cev{d}_{\sigma}(v))+(1-\lambda)\varphi(\cev{d}_{\bar{\sigma}}(v))\leq C\varphi(\varrho_D(v))$ for every $v \in V$.
  Hence
  \[
    \lambda\sum_{v \in V}\varphi(\cev{d}_{\sigma}(v))+(1-\lambda)\sum_{v \in V}\varphi(\cev{d}_{\bar{\sigma}}(v))\leq C\sum_{v \in V}\varphi(\varrho_D(v))\leq C\min_{\sigma'}\sum_{v \in V}\varphi(\cev{d}_{\sigma'}(v)).
  \]
  Thus, the order with smaller $\varphi$\nobreakdash-sum among $\sigma$ and $\bar{\sigma}$ is a $C$\nobreakdash-approximation.

  The changes in $q_D$ caused by moving a vertex to the first or last position can be computed in $O(|V|+|E|)$ time.
  Hence, the procedure is strongly polynomial.

\medskip

  Finally, fix an arbitrary vertex order $\sigma$, let $\bar{\sigma}$ be its reverse, and let $\sigma^*$ be an optimal order.
  Orient every edge from its earlier endpoint to its later endpoint in $\sigma^*$, and define $a_v,b_v,c_v,d_v$ relative to this orientation $D$ and order $\sigma$ as above.
  Then~\eqref{eq:weightedEndpointDecomposition} holds.
  If~\eqref{eq:weightedEndpointCondition} holds for all $a,b,c,d \in \Z_+$, summing it for $a_v,b_v,c_v,d_v$ over the vertices shows that the order with smaller $\varphi$\nobreakdash-sum among $\sigma$ and $\bar{\sigma}$ is a $C$\nobreakdash-approximation.
  The running-time statement follows by scanning the graph once to compute the two left-degree vectors.
\end{proof}

Notice that, when $C\neq1$, Condition~\eqref{eq:weightedEndpointCondition} implies that $\varphi$ cannot change sign over its domain.
To see this, apply~\eqref{eq:weightedEndpointCondition} with $a=d=z$ and $b=c=0$ for any $z\in\Z_+$, which gives $\varphi(z)\leq C\varphi(z)$.
Hence, $\varphi(z)\leq0$ for every $z\in\Z_+$ if $C<1$, while $\varphi(z)\geq0$ for every $z\in\Z_+$ if $C>1$.

As a direct corollary, we obtain the analogous result for maximizing the $\psi$\nobreakdash-sum for a discrete concave function $\psi$.
\begin{corollary}\label{cor:weightedEndpointMaxApproximation}
  Let $\lambda \in [0,1]$ and $C \geq 1$, and let $\psi : \Z_+ \to \R$ be a fixed discrete concave function.
  Suppose that
  \begin{equation}\label{eq:weightedEndpointMaxCondition}
    \lambda\psi(a+c)+(1-\lambda)\psi(b+d)\geq \frac{1}{C}\psi(a+b)
  \end{equation}
  for all $a,b,c,d \in \Z_+$ satisfying $c\leq a$ and $b\leq d$.
  Then maximizing the $\psi$\nobreakdash-sum of left-degrees over the vertex orders of loop-free multigraphs admits a strongly polynomial-time $C$\nobreakdash-approximation algorithm.

  If~\eqref{eq:weightedEndpointMaxCondition} holds for all $a,b,c,d \in \Z_+$ without the restrictions $c\leq a$ and $b\leq d$, then the order with larger $\psi$\nobreakdash-sum among any vertex order and its reverse is a $C$\nobreakdash-approximation.
  In this case, concavity is not required.
  Furthermore, the two orders and their left-degrees can be found in $O(|V|+|E|)$ time, and their objective values require $O(|V|)$ function evaluations.
\end{corollary}
\begin{proof}
  Apply Theorem~\ref{thm:weightedEndpointApproximation} to $\varphi=-\psi$ with $\frac{1}{C}$ in place of $C$.
  Under this substitution, condition~\eqref{eq:weightedEndpointMaxCondition} is precisely condition~\eqref{eq:weightedEndpointCondition}, and the returned order $\sigma$ satisfies
  $\sum_{v \in V}\psi(\cev{d}_{\sigma}(v))\geq\frac{1}{C}\max_{\sigma'}\sum_{v \in V}\psi(\cev{d}_{\sigma'}(v))$.
  The unrestricted version and its running-time bound follow from the corresponding part of Theorem~\ref{thm:weightedEndpointApproximation}.
\end{proof}

Now we state several useful corollaries of Theorem~\ref{thm:weightedEndpointApproximation}~and its concave counterpart, Corollary~\ref{cor:weightedEndpointMaxApproximation}.
\begin{corollary}\label{cor:twoSidedConvexApproximation}
  Let $\varphi : \Z_+ \to \R_+$ be a fixed discrete convex function.
  Suppose that there exists a constant $C_\varphi \geq 1$ such that
  \begin{equation}\label{eq:twoSidedConvexGrowth}
    \max\{\varphi(0),\varphi(2z)\}\leq C_\varphi\varphi(z)
  \end{equation}
  for every $z \in \Z_+$.
  Then minimizing the $\varphi$\nobreakdash-sum of left-degrees over the vertex orders of loop-free multigraphs admits a strongly polynomial-time $C_\varphi$\nobreakdash-approximation algorithm.
\end{corollary}
\begin{proof}
  Apply Theorem~\ref{thm:weightedEndpointApproximation} with $\lambda=1$ and $C=C_\varphi$.
  Condition~\eqref{eq:weightedEndpointCondition} holds because, whenever $c\leq a$, we have $0\leq a+c\leq2(a+b)$, and convexity and~\eqref{eq:twoSidedConvexGrowth} give $\varphi(a+c)\leq\max\{\varphi(0),\varphi(2(a+b))\}\leq C_\varphi\varphi(a+b)$.
\end{proof}
Together with Theorem~\ref{thm:fixedConvexDichotomy}, this gives a complexity classification within the class covered by Corollary~\ref{cor:twoSidedConvexApproximation}.
If $\varphi$ is linear on the positive integers, exact optimization is possible in linear time; otherwise the problem is APX-hard while admitting the constant-factor approximation above.
The following two corollaries follow from Corollary~\ref{cor:twoSidedConvexApproximation}.
First, the corollary simplifies to the following form when monotonicity of $\varphi$ is assumed.
\begin{corollary}\label{cor:convexApproximation}
  Let $\varphi : \Z_+ \to \R_+$ be a fixed non-decreasing discrete convex function for which there exists a constant $C_\varphi \geq 1$ such that $\varphi(2z)\leq C_\varphi\varphi(z)$ for every $z \in \Z_+$.
  Then minimizing the $\varphi$\nobreakdash-sum of left-degrees over the vertex orders of loop-free multigraphs admits a strongly polynomial-time $C_\varphi$\nobreakdash-approximation algorithm.
\end{corollary}
\begin{proof}
  The monotonicity of $\varphi$ implies that $\varphi(0)\leq \varphi(2z)$ for every $z\in \Z_+$.
  Hence, $\varphi(2z)\leq C_\varphi\varphi(z)$ implies~\eqref{eq:twoSidedConvexGrowth}, and Corollary~\ref{cor:twoSidedConvexApproximation} gives the stated approximation factor.
\end{proof}
In particular, if $\varphi(1)\geq\varphi(0)$, then convexity implies that $\varphi$ is non-decreasing. If, in addition, $\sup_{z\geq1}\frac{\varphi(2z)}{\varphi(z)}<\infty$, then Corollary~\ref{cor:convexApproximation} yields a strongly polynomial-time constant-factor approximation algorithm.
For example, for any fixed integer $p\geq1$, the function $\varphi(z)=z^p$ ($z\in\Z_+$) has the property above, and $\sup_{z\geq1}\frac{(2z)^p}{z^p}=2^p$, hence Corollary~\ref{cor:convexApproximation} gives a $2^p$-approximation.

We state another natural consequence of Corollary~\ref{cor:twoSidedConvexApproximation}, without assuming the monotonicity of $\varphi$.
\begin{corollary}\label{cor:positiveNonMonotoneConvexApproximation}
  Let $\varphi : \Z_+ \to \R_+$ be a fixed discrete convex function.
  Suppose that $\ell_\varphi = \inf_{z \in \Z_+}\varphi(z)>\nolinebreak0$ and that there exists a constant $D_\varphi \geq 1$ such that $\varphi(2z)\leq D_\varphi\varphi(z)$ for every $z \in \Z_+$.
  Then minimizing the $\varphi$\nobreakdash-sum of left-degrees over the vertex orders of loop-free multigraphs admits a strongly polynomial-time $\max\{\frac{\varphi(0)}{\ell_\varphi},D_\varphi\}$\nobreakdash-approximation algorithm.
  In particular, the conclusion applies when $\varphi$ is neither non-increasing nor non-decreasing.
\end{corollary}
\begin{proof}
  Set $C_{\varphi}=\max\{\frac{\varphi(0)}{\ell_\varphi},D_\varphi\}$.
  Since $\varphi(z)\geq\ell_\varphi$ and $\varphi(2z)\leq D_\varphi\varphi(z)$ for every $z \in \Z_+$, we have $\max\{\varphi(0),\varphi(2z)\}\leq C_{\varphi}\varphi(z)$.
  Thus,~\eqref{eq:twoSidedConvexGrowth} holds, and Corollary~\ref{cor:twoSidedConvexApproximation} gives the stated approximation factor.
\end{proof}
For example, for a fixed integer $p\geq 1$, and fixed rationals $q>0$, and $r\geq0$, let $\varphi(z)=|z-r|^p+q$.
The function $\varphi$ is discrete convex and satisfies
$\ell_\varphi\geq q>0$. We claim that the condition of the corollary holds with
$D_\varphi=2^{2p}\left(1+\frac{r^p}{q}\right)$.

For $x,y\in\R_+$, Jensen's inequality applied to $t\mapsto t^p$ gives $(x+y)^p\leq 2^{p-1}(x^p+y^p)$. Substituting $x=2|z-r|$ and $y=r$ and using $|2z-r|=|2(z-r)+r|\leq 2|z-r|+r$,
we obtain
\[
|2z-r|^p\leq (2|z-r|+r)^p
\leq 2^{p-1}\bigl(2^p|z-r|^p+r^p\bigr)
\leq
2^{2p-1}\left(1+\frac{r^p}{q}\right)(|z-r|^p+q).
\]
Therefore,
\[
\varphi(2z)=|2z-r|^p+q\leq
\left(2^{2p-1}\left(1+\frac{r^p}{q}\right)+1\right)(|z-r|^p+q)\leq D_\varphi\varphi(z).
\]
Since $\frac{\varphi(0)}{\ell_\varphi}
\leq \frac{r^p+q}{q}
=1+\frac{r^p}{q}
\leq D_\varphi$, the corollary gives a
$2^{2p}\left(1+\frac{r^p}{q}\right)$-approximation.
Theorem~\ref{thm:fixedConvexDichotomy} shows that this problem is APX-complete when $p\geq2$, or when $p=1$ and $r>1$.
When $p=1$ and $0\leq r\leq1$, we have $\varphi(z)=z-r+q$ for every positive integer $z$, so the problem is solvable exactly in linear time.

When the relevant inequality holds without the restrictions $c\leq a$ and $b\leq d$, the final clause of Theorem~\ref{thm:weightedEndpointApproximation} yields an algorithm running in $O(|V|+|E|)$ time and requiring $O(|V|)$ function evaluations, by returning the order with smaller $\varphi$\nobreakdash-sum among any vertex order and its reverse.

\begin{corollary}\label{cor:nonIncreasingConvexApproximation}
  Let $\varphi : \Z_+ \to \R_+$ be a fixed non-increasing discrete convex function, and suppose that $\ell_\varphi = \inf_{z \in \Z_+}\varphi(z)>0$.
  Then minimizing the $\varphi$\nobreakdash-sum of left-degrees over the vertex orders of loop-free multigraphs admits a $\frac{\varphi(0)+\ell_\varphi}{2\ell_\varphi}$\nobreakdash-approximation algorithm, by returning the order with smaller $\varphi$\nobreakdash-sum among any vertex order and its reverse.
  The algorithm uses $O(|V|+|E|)$ time and $O(|V|)$ evaluations of $\varphi$.
\end{corollary}
\begin{proof}
  Apply the final clause of Theorem~\ref{thm:weightedEndpointApproximation} with $\lambda=\frac{1}{2}$ and $C=\frac{\varphi(0)+\ell_\varphi}{2\ell_\varphi}$.
  For all $a,b,c,d \in \Z_+$, convexity implies $\varphi(a)+\varphi(b)\leq \varphi(0)+\varphi(a+b)$, and with monotonicity this gives
  \[
    \frac{\varphi(a+c)+\varphi(b+d)}{2}\leq\frac{\varphi(a)+\varphi(b)}{2}\leq\frac{\varphi(0)+\varphi(a+b)}{2}\leq C\varphi(a+b),
  \]
  hence Condition~\eqref{eq:weightedEndpointCondition} holds.
\end{proof}
For example, for fixed rationals $q>0, r>1$, minimizing the $\varphi$\nobreakdash-sum of left-degrees for $\varphi(z)=r^{-z}+q$ ($z\in \Z_+$) admits a $\left(1+\frac{1}{2q}\right)$-approximation.
This objective is discrete strictly convex, so Theorem~\ref{thm:fixedConvexDichotomy} shows that the problem is APX-complete.

\begin{corollary}\label{cor:nonNegativeNonDecreasingConcaveApproximation}
  Let $\psi : \Z_+ \to \R_+$ be a fixed non-decreasing discrete concave function.
  Then maximizing the $\psi$\nobreakdash-sum of left-degrees over the vertex orders of loop-free multigraphs admits a $2$\nobreakdash-approximation algorithm, by returning the order with larger $\psi$\nobreakdash-sum among any vertex order and its reverse.
  The algorithm uses $O(|V|+|E|)$ time and $O(|V|)$ evaluations of $\psi$.
\end{corollary}
\begin{proof}
  Apply the final clause of Corollary~\ref{cor:weightedEndpointMaxApproximation} with $\lambda=\frac{1}{2}$ and $C=2$.
  For all $a,b,c,d \in \Z_+$, concavity implies $\psi(a)+\psi(b) \geq \psi(0)+\psi(a+b)$ and with monotonicity this gives
  \[
    \frac{\psi(a+c)+\psi(b+d)}{2}\geq\frac{\psi(a)+\psi(b)}{2}\geq\frac{\psi(0)+\psi(a+b)}{2}\geq\frac{1}{2}\psi(a+b),
  \]
  hence Condition~\eqref{eq:weightedEndpointMaxCondition} holds.
\end{proof}

For example, maximizing the $\psi$-sum of left-degrees admits a $2$-approximation for $\psi(z)=\log(z+1)$ and for $\psi(z)=z^q$ with any fixed rational $0<q<1$, including $\psi(z)=\sqrt{z}$.

\subsection{Exact exponential-time optimization and a matching lower bound}\label{sec:DinProg}
Theorem~\ref{thm:fixedConvexDichotomy} shows that exact optimization of the $\varphi$\nobreakdash-sum objective is NP-hard for every fixed discrete convex function $\varphi$ that is not linear on the positive integers.
In this subsection, we provide an exact method for solving the following more general problem, which includes the acyclic version of Problem~\ref{prob:sumPhi_VRhoV}.

Let $G = (V, E)$ be a loop-free multigraph, and let $\varphi_v : \Z_+ \to \R$ be a discrete (not necessarily convex) function given for each $v \in V$.
Our goal is to find an order that minimizes $\sum_{v \in V} \varphi_v(\cev{d}(v))$.
The natural brute-force approach is to try all $|V|!$ permutations of the vertices and select one that minimizes the objective value.
Now we present a dynamic programming algorithm and an implementation whose running time is $O\left(2^{|V|}|V|+\poly(|V|,|E|)\right)$.

Let $f:2^V \to \R$ be defined as follows.
Let $f(\emptyset) = 0$, and for each non-empty $V'\subseteq V$, let
\begin{equation}\label{eq:defF}
  f(V') = \min_{v \in V'} \left\{ f(V' \setminus \{v\}) + \varphi_v(d(v, V')) \right\}.
\end{equation}
Furthermore, let $g:2^V\setminus\{\emptyset\}\to V$ be any function that, for each non-empty $V'\subseteq V$, satisfies
\begin{equation}\label{eq:defG}
  g(V') \in \argmin_{v \in V'} \left\{ f(V' \setminus \{v\}) + \varphi_v(d(v, V')) \right\}.
\end{equation}

We claim that, for each non-empty subset $V' \subseteq V$, the minimum value of $\sum_{v \in V'} \varphi_v(\cev{d}(v))$ for the graph $G[V']$ is $f(V')$ as defined in~\eqref{eq:defF}, and there exists an optimal order in which the last vertex is $g(V')$ as defined in~\eqref{eq:defG}.
We prove this by induction on $|V'|$.
The case $V'=\emptyset$ follows from $f(\emptyset)=0$.
Suppose $V'\neq\emptyset$ and the claim holds for all subsets of size smaller than $|V'|$.
Let $n'=|V'|$, and consider an optimal order $\sigma = (v_1, \dots, v_{n'})$ for $G[V']$.
Then $(v_1, \dots, v_{{n'}-1})$ is an optimal order for $G[V' \setminus \{v_{n'}\}]$, so its objective value is $f(V' \setminus \{v_{n'}\})$ by induction.
Since $v_{n'}$ is last, its contribution is $\varphi_{v_{n'}}(d(v_{n'}, V'))$.
Thus $\sigma$ has objective value $f(V' \setminus \{v_{n'}\}) + \varphi_{v_{n'}}(d(v_{n'}, V'))$.
This shows that $f(V')$ is at most the optimal value.
Conversely, let $v_{n'} = g(V')$.
By induction, $f(V' \setminus \{v_{n'}\})$ is the optimal value for $G[V' \setminus \{v_{n'}\}]$.
Appending $v_{n'}$ to an optimal order for $G[V' \setminus \{v_{n'}\}]$ achieves objective value $f(V')$, proving the claim.

Therefore, once the functions $f$ and $g$ have been computed, an optimal order can be recovered iteratively as follows.
Starting from $V'=V$, we place the vertex $g(V')$ in the last free position and remove it from the current graph.
Repeating this until no vertices remain yields an optimal order.
It remains to compute $f(V')$ for every subset $V' \subseteq V$ and $g(V')$ for every non-empty subset $V' \subseteq V$.
First compute $\varphi_v(j)$ for every $v\in V$ and $j\in\{0,\dots,d(v)\}$.
This requires $\sum_{v\in V}(d(v)+1)=2|E|+|V|$ function evaluations in total.
Precompute $e(S)=|E(G[S])|$ for every $S\subseteq V$, where parallel edges are counted with multiplicity.
These values can be computed in $O(2^{|V|}|V|)$ time and $O(2^{|V|})$ space by processing the subsets in increasing order of size.
For each non-empty $S\subseteq V$, choose a vertex $v\in S$ and use $e(S)=e(S\setminus\{v\})+d(v,S\setminus\{v\})$, where the final term can be computed by scanning the at most $|V|-1$ vertices of $S\setminus\{v\}$.
For every $v\in S$, $d(v,S)=e(S)-e(S\setminus\{v\})$, so every candidate in~\eqref{eq:defF} and~\eqref{eq:defG} can then be evaluated in constant time from the tabulated values.
The number of candidates over all non-empty subsets of $V$ is $\sum_{S\subseteq V}|S|=2^{|V|-1}|V|$.
The preceding analysis gives the following running-time bound.
\begin{theorem}\label{thm:exactExponentialAlgorithm}
  Let $G=(V,E)$ be a loop-free multigraph, and let $\varphi_v:\Z_+\to\R$ be given for every $v\in V$.
  An order minimizing $\sum_{v\in V}\varphi_v(\cev{d}(v))$ can be found in $O\left(2^{|V|}|V|+\poly(|V|,|E|)\right)$ time and $O\left(2^{|V|}+\poly(|V|,|E|)\right)$ space.
\end{theorem}

We now show that the exponential dependence on $|V|$ is unavoidable under the Exponential Time Hypothesis throughout the hard side of Theorem~\ref{thm:fixedConvexDichotomy}.

\begin{theorem}\label{thm:exactExponentialETHLowerBound}
  Let $\varphi:\Z_+\to\R$ be a fixed discrete convex function that is not linear on the positive integers.
  Assuming the Exponential Time Hypothesis, there is no $2^{o(|V|)}$\nobreakdash-time algorithm for finding an order that minimizes the $\varphi$\nobreakdash-sum of left-degrees on connected simple graphs.
  Moreover, for every breakpoint $k\geq2$ of $\varphi$, the same lower bound holds when restricted to connected simple graphs satisfying $k=\left\lceil \frac{|E|}{|V|}\right\rceil$.
\end{theorem}
\begin{proof}
  Maximum Independent Set on $3$\nobreakdash-regular $n$\nobreakdash-vertex graphs has no $2^{o(n)}$\nobreakdash-time exact algorithm under the Exponential Time Hypothesis~\cite{amiri2021note}.
  Since the complement of a maximum independent set is a minimum vertex cover, the same lower bound holds for Minimum Vertex Cover on $3$\nobreakdash-regular graphs.

  It suffices to prove the final, restricted statement.
  Fix a breakpoint $k\geq2$ of $\varphi$ and put $s=k-2$.
  Specializing Theorem~\ref{thm:graphDependentBreakpointNPH} to the fixed family $\varphi_G=\varphi$, the exact-hardness reduction first applies the construction from Theorem~\ref{thm:fixedStrictlyConvexAPXHardness} to the shifted function $\widetilde{\varphi}(z)=\varphi(z+s)$.
  The resulting graph has $15n+3$ vertices for a source graph on $n$ vertices.
  If $s=0$, this graph is already the final instance.
  If $s\geq1$, the reduction takes a fixed number $t>\frac{s(s+3)}{2}$ of disjoint copies and adds $s$ new clique vertices, so the final graph has $t(15n+3)+s=O(n)$ vertices.
  In both cases the resulting graph is connected and simple, satisfies $\left\lceil \frac{|E|}{|V|}\right\rceil=k$, and an optimal order yields a minimum vertex cover of the source graph.

  Hence a $2^{o(|V|)}$\nobreakdash-time exact algorithm for the fixed $\varphi$\nobreakdash-objective on these instances would solve Minimum Vertex Cover on $3$\nobreakdash-regular $n$\nobreakdash-vertex graphs in $2^{o(n)}$ time, contradicting the Exponential Time Hypothesis.
\end{proof}

\section{Power objectives over acyclic orientations}\label{sec:powerObjectives}
In this section, we further investigate the minimization of the $p$\nobreakdash-power sum of left-degrees for an integer $p$.
Corollary~\ref{cor:convexApproximation} gives a strongly polynomial-time $2^p$\nobreakdash-approximation algorithm for every fixed integer $p\geq1$.
Together with Theorem~\ref{thm:fixedConvexDichotomy}, which applies to every breakpoint of $\varphi(z)=z^p$ for each fixed integer $p\geq2$, this yields the following classification of the problem for fixed exponents.
\begin{corollary}\label{cor:fixedPowerNoPTAS}\label{thm:squaresumNPForSimple}
  For every fixed integer $p \geq 2$, minimizing the $p$\nobreakdash-power sum of left-degrees is APX-complete even on connected simple graphs.
  On loop-free multigraphs, the problem admits a strongly polynomial-time $2^p$\nobreakdash-approximation algorithm.
  Thus, unless $\mathrm{P} = \mathrm{NP}$, the problem admits no polynomial-time approximation scheme.
  Moreover, for every fixed integer $k\geq2$, APX-completeness remains valid when restricted to connected simple graphs satisfying $k=\left\lceil \frac{|E|}{|V|}\right\rceil$.
\end{corollary}

We next turn to the dependence of the approximation ratio on the exponent $p$. We first show that the exponential dependence on $p$ in the approximation ratio is near-tight in the exponent when the positive integer $p$ is part of the input. We then study the approximation performance of the particularly simple smallest-last algorithm for the quadratic objective.

\subsection{Approximation hardness when the exponent is part of the input}\label{sec:variablePowerApproximationHardness}

Corollary~\ref{cor:fixedPowerNoPTAS} gives a $2^p$\nobreakdash-approximation for every fixed integer exponent $p\geq2$ and rules out a polynomial-time approximation scheme unless $\mathrm{P}=\mathrm{NP}$.
We now consider the setting in which the positive integer exponent $p$ is part of the input and show that the exponential dependence on $p$ of the approximation ratio is essentially unavoidable, in the sense that the exponent of $2$ cannot be reduced to $p^{1-\varepsilon}$ for any fixed $\varepsilon>0$.

\begin{theorem}\label{thm:noPolynomialPowerApproximation}
  For any fixed real number $\varepsilon>0$, there is no deterministic polynomial-time approximation algorithm for
  minimizing the $p$\nobreakdash-power sum of left-degrees over the vertex orderings of a simple graph
  with approximation ratio at most $2^{p^{1-\varepsilon}}$, where the positive integer $p$ is part of the input, unless $\mathrm{P}=\mathrm{NP}$.
  This holds even for simple graphs of maximum degree at most $18$.
\end{theorem}
\begin{proof}
  It suffices to prove the statement for $0<\varepsilon\leq\frac{1}{2}$, since if $\varepsilon>\frac{1}{2}$, then $2^{p^{1-\varepsilon}}\leq2^{p^{\frac{1}{2}}}$ for every $p\geq1$.
  Thus, an algorithm with approximation ratio at most $2^{p^{1-\varepsilon}}$ would also have approximation ratio at most $2^{p^{\frac{1}{2}}}$.
  Fix a real number $\varepsilon$ with $0<\varepsilon\leq\frac{1}{2}$, and suppose, for a contradiction, that such an approximation algorithm exists.
  For every simple graph $G=(V,E)$ of maximum degree at most $18$ and every positive integer $p$, the algorithm returns an order $\sigma$ satisfying
  \begin{equation}\label{eq:polyPowerApproximationAssumption}
    \sum_{v\in V}\cev{d}_{\sigma}(v)^p
    \leq
    2^{p^{1-\varepsilon}}\opt_p(G),
  \end{equation}
  where $\opt_p(G)=\min_{\sigma}\sum_{v\in V}\cev{d}_{\sigma}(v)^p$.

  We reduce from Minimum Monotone Satisfying Assignment at depth three ($\mathrm{MMSA}_3$).
  An $\mathrm{MMSA}_3$ instance is a monotone Boolean formula of the form
  \[
    \Phi
    =
    \bigwedge_{a=1}^{A}
    \bigvee_{b=1}^{B_a}
    \bigwedge_{c=1}^{C_{a,b}} x_{\lambda(a,b,c)}
  \]
  over variables $x_1,\dots,x_n$, where $\lambda(a,b,c)\in[n]$, and the objective is to find a satisfying assignment that minimizes the number of true variables.
  Let $m$ denote the encoding size of the instance and put
  \begin{equation}\label{eq:mmsaFormulaOptimum}
    \tau=\min\{|X|:X\subseteq[n],\ \Phi(X)=1\},
  \end{equation}
  where $\Phi(X)$ is the value of $\Phi$ when $x_i$ is true exactly for $i\in X$.
  Dinur and Safra proved that $\mathrm{MMSA}_3$ is NP-hard to approximate within a factor $2^{(\log m)^{1-o(1)}}$, where $m$ is the encoding size; see~\cite[p.~249]{dinur2004labelcover}.
  Thus, there is a function $\gamma(m)$ tending to $0$ as $m$ tends to infinity such that a deterministic polynomial-time $2^{(\log m)^{1-\gamma(m)}}$\nobreakdash-approximation for $\mathrm{MMSA}_3$ would imply $\mathrm{P}=\mathrm{NP}$.
  Fix real numbers $\eta>0$ and $c>0$.
  For all sufficiently large $m$, we have $\gamma(m)<\frac{\eta}{2}$ and
  \[
    \frac{(\log m)^{1-\gamma(m)}}{c(\log m)^{1-\eta}}
    =
    \frac{(\log m)^{\eta-\gamma(m)}}{c}
    \geq
    \frac{(\log m)^{\frac{\eta}{2}}}{c}.
  \]
  The last expression tends to infinity with $m$.
  Consequently, for every fixed $\eta>0$ and $c>0$, $\mathrm{MMSA}_3$ has no deterministic polynomial-time $2^{c(\log m)^{1-\eta}}$\nobreakdash-approximation unless $\mathrm{P}=\mathrm{NP}$.

  We may assume that every variable occurs and that $\tau\geq1$.
  Variables that do not occur can be deleted, while the case $\Phi(\emptyset)=1$ has optimum $\tau=0$ and can be solved exactly.
  Suppress gates of fan-in one and replace gates of larger fan-in by binary trees of gates of the same type.
  This expansion preserves every satisfying assignment and hence preserves $\tau$.
  We henceforth also denote the expanded monotone binary formula tree by $\Phi$.
  The expanded formula has $O(m)$ gates and variable occurrences.
  Every variable occurs in $\Phi$, and the all-true assignment satisfies $\Phi$, so
  \begin{equation}\label{eq:mmsaCardinalityBounds}
    1\leq\tau\leq n\leq m.
  \end{equation}

  We construct a simple graph $G$ from the expanded formula and later choose $p=O(\log m)$.
  From the order returned by the assumed algorithm, we will obtain a satisfying assignment of size less than $\bigl(2^{p^{1-\varepsilon}}+1\bigr)\tau$ for all sufficiently large $m$.
  
  We identify a vertex order with the sequence in which its vertices are removed from right to left, so the degree at removal is the resulting left-degree.

\medskip
  We begin by introducing the gadgets used in the construction of the graph.
  For $1\leq t\leq d\leq 12$, put $k=d-t+1$ and define a \emph{threshold gadget} $T(d,t)$ as follows.
  Its internal graph has a root $r$; in the final graph, $r$ is adjacent to $d$ designated roots of other gadgets, called its \emph{signal neighbors}.
  For odd $k$, take two copies of $K_{13}$, join $r$ to $\frac{13-k}{2}$ vertices in each, and perfectly match the remaining vertices of the two cliques.
  For even $k$, take two copies of $K_{13}$ and an additional vertex $w$.
  Join $r$ to $(13-k)$ vertices of the first clique, join $w$ to $\frac{k}{2}$ other vertices of the first clique and to $\left(13-\frac{k}{2}\right)$ vertices of the second clique, and perfectly match the remaining $\frac{k}{2}$ vertices of each clique.
  In both constructions, every non-root clique vertex has its $12$ clique neighbors and exactly one additional neighbor, while in the even case $w$ has degree $\frac{k}{2}+\left(13-\frac{k}{2}\right)=13$.
  The matching connects the two cliques after $r$ is deleted in the odd case, and $w$ connects them in the even case.
  Thus
  \begin{equation}\label{eq:polyPowerActivationDegrees}
    d_{T(d,t)}(r)=13-k=12-d+t
    \qquad \text{and }\qquad
    d_{T(d,t)}(v)=13
    \quad\text{for every }v\neq r,
  \end{equation}
  and $T(d,t)-r$ is connected. 
  The construction of $T(d,t)$ is illustrated in Figure~\ref{fig:polyPowerThresholdConstruction}.
\begin{figure}[h]
  \centering
  \begin{subfigure}[t]{.49\textwidth}
    \centering
    \begin{tikzpicture}[yscale=1.7, xscale=1.7]
      \tikzset{VertexStyle/.append style = {minimum size = 20pt,inner sep=0pt}}
      \SetVertexMath
      \Vertex[x=1.5, y=4.5, L= r]{r}
      
      \Vertex[x=0, y=4, L=$ $]{a1}
      \node(sa1) at (0,3.5) {$\cdot$};
      \node(sa2) at (0,3.25) {$\cdot$};
      \node(sa3) at (0,3) {$\cdot$};
      \Vertex[x=0, y=2.5, L=$ $]{a2}
      \Vertex[x=0, y=1.8, L=$ $]{a3}
      \node(ta1) at (0,1.23) {$\cdot$};
      \node(ta2) at (0,0.9) {$\cdot$};
      \node(ta3) at (0,0.56) {$\cdot$};
      \Vertex[x=0, y=0, L=$ $]{a4}

      \begin{pgfonlayer}{background}
        \highlight{10mm}{black!15}{(a1.center) -- (a4.center)}
      \end{pgfonlayer}
      \node(k) at (-0.75,2) {$K_{13}$};

      \Vertex[x=3, y=4, L=$ $]{b1}
      \node(sb1) at (3,3.5) {$\cdot$};
      \node(sb2) at (3,3.25) {$\cdot$};
      \node(sb3) at (3,3) {$\cdot$};
      \Vertex[x=3, y=2.5, L=$ $]{b2}
      \Vertex[x=3, y=1.8, L=$ $]{b3}
      \node(tb1) at (3,1.23) {$\cdot$};
      \node(tb2) at (3,0.9) {$\cdot$};
      \node(tb3) at (3,0.56) {$\cdot$};
      \Vertex[x=3, y=0, L=$ $]{b4}

      \begin{pgfonlayer}{background}
        \highlight{10mm}{black!15}{(b1.center) -- (b4.center)}
      \end{pgfonlayer}
      \node(k) at (3.75,2) {$K_{13}$};

      \draw[wavy] (r) to (1.5,5.33);
      \draw[wavy] (r) to (1.25,5.33);
      \draw[wavy] (r) to (1.75,5.33);
      \path[circleAroundEdges={.12}{.33}{0.33}] (r)--(1.5,5.33) node [right=.25cm, pos=0.33] {$d$};
      
      \draw (r)--(a1);
      \path [draw third paths={draw=none}{dotted}] (r)--(sa1);
      \path [draw third paths={draw=none}{dotted}] (r)--(sa2);
      \path [draw third paths={draw=none}{dotted}] (r)--(sa3);
      \draw (r)--(a2);
      \path[circleAroundEdges={.15}{.45}{0.25}] (r)--(0,3.4) node [above=.6cm, pos=0.5] {$\frac{13-k}{2}$};

      \draw (r)--(b1);
      \path [draw third paths={draw=none}{dotted}] (r)--(sb1);
      \path [draw third paths={draw=none}{dotted}] (r)--(sb2);
      \path [draw third paths={draw=none}{dotted}] (r)--(sb3);
      \draw (r)--(b2);
      \path[circleAroundEdges={.15}{.45}{0.25}] (r)--(3,3.4) node [above=.6cm, pos=0.5] {$\frac{13-k}{2}$};

      \node(m) at (1.5,0.9) {matching of size $\frac{13+k}{2}$};
      \draw (a3)--(b3);
      \path [draw third paths={draw=none}{dotted}] (ta1)--(tb1);
      \path [draw third paths={draw=none}{dotted}] (ta3)--(tb3);
      \draw (a4)--(b4);
    \end{tikzpicture}
    \caption{The case when $k$ is odd.}
  \end{subfigure}
  \hfill
  \begin{subfigure}[t]{.49\textwidth}
    \centering
    \begin{tikzpicture}[yscale=1.7, xscale=1.7]
      \tikzset{VertexStyle/.append style = {minimum size = 20pt,inner sep=0pt}}
      \SetVertexMath
      \Vertex[x=1.5, y=4.5, L= r]{r}
      \Vertex[x=1.5, y=2.775, L=w]{w}
      
      \Vertex[x=0, y=4, L=$ $]{a1}
      \node(da1) at (0,3.62) {$\vdots$};
      \Vertex[x=0, y=3.1, L=$ $]{a2}
      \Vertex[x=0, y=2.45, L=$ $]{a3}
      \node(da1) at (0,2.07) {$\vdots$};
      \Vertex[x=0, y=1.55, L=$ $]{a4}
      \Vertex[x=0, y=0.9, L=$ $]{a5}
      \node(da1) at (0,0.52) {$\vdots$};
      \Vertex[x=0, y=0, L=$ $]{a6}

      \begin{pgfonlayer}{background}
        \highlight{10mm}{black!15}{(a1.center) -- (a6.center)}
      \end{pgfonlayer}
      \node(k) at (-0.75,2) {$K_{13}$};

      \Vertex[x=3, y=4, L=$ $]{b1}
      \node(s1) at (3,3.275) {$\cdot$};
      \node(s2) at (3,2.775) {$\cdot$};
      \node(s3) at (3,2.275) {$\cdot$};
      \Vertex[x=3, y=1.55, L=$ $]{b2}
      \Vertex[x=3, y=0.9, L=$ $]{b3}
      \node(db3) at (3,0.52) {$\vdots$};
      \Vertex[x=3, y=0, L=$ $]{b4}

      \begin{pgfonlayer}{background}
        \highlight{10mm}{black!15}{(b1.center) -- (b4.center)}
      \end{pgfonlayer}
      \node(k) at (3.75,2) {$K_{13}$};

      \draw[wavy] (r) to (1.5,5.33);
      \draw[wavy] (r) to (1.25,5.33);
      \draw[wavy] (r) to (1.75,5.33);
      \path[circleAroundEdges={.12}{.33}{0.33}] (r)--(1.5,5.33) node [right=.25cm, pos=0.33] {$d$};
      
      \draw (r)--(a1);
      \path [draw third paths={draw=none}{dotted}] (r)--(0,3.55);
      \draw (r)--(a2);
      \path[circleAroundEdges={.1}{.3}{0.25}] (r)--(0,3.6) node [above=.6cm, pos=0.5] {$13-k$};
      \draw (w)--(a3);
      \path [draw third paths={draw=none}{dotted}] (0, 2)--(w);
      \draw (w)--(a4);
      \path[circleAroundEdges={.1}{.33}{0.25}] (w)--(0,2.05) node [above=.4cm, pos=0.4] {$\frac{k}{2}$};
      \draw (w)--(b1);
      \path [draw third paths={draw=none}{dotted}] (w)--(s1);
      \path [draw third paths={draw=none}{dotted}] (w)--(s2);
      \path [draw third paths={draw=none}{dotted}] (w)--(s3);
      \draw (w)--(b2);
      \path[circleAroundEdges={.2}{.7}{0.2}] (w)--(s2) node [above=.8cm, pos=0.1] {$13-\frac{k}{2}$};

      \node(m) at (1.5,0.45) {matching of size $\frac{k}{2}$};
      \draw (a5)--(b3);
      \draw (a6)--(b4);
    \end{tikzpicture}
    \caption{The case when $k$ is even.}
  \end{subfigure}
  \caption{The threshold gadget $T(d,t)$, where $k = d - t + 1$. The wavy edges represent the edges connecting the root $r$ to its $d$ neighbors outside the gadget.}\label{fig:polyPowerThresholdConstruction}
  \end{figure}
  
  In the final graph, the root $r$ has degree $12+t$ including its $d$ signal neighbors. We say that the gadget receives one signal whenever one of its signal neighbors is removed.
  As long as no vertex of the threshold gadget has been removed, after exactly $s$ signals its root has degree
  \begin{equation}\label{eq:thresholdRootDegree}
    12+t-s.
  \end{equation}

  \begin{claim}\label{cl:thresholdGadgetBehavior}
    In an arbitrary removal order, a threshold gadget $T(d,t)$ has the following properties.
    \begin{enumerate}[label=(\roman*)]
      \item After any first gadget vertex has been removed, all remaining gadget vertices can be removed at degrees at most $12$ without first removing a vertex outside the gadget.\label{item:thresholdGagdetBehavior:i}
      \item After at least $t$ signal neighbors have been removed, the entire gadget can be removed at degrees at most $12$, starting with the root of the gadget.\label{item:thresholdGagdetBehavior:ii}
      \item Before $t$ signal neighbors have been removed, the first removed gadget vertex has degree at least $13$.\label{item:thresholdGadgetBehavior:iii}
    \end{enumerate}
  \end{claim}
  \begin{subproof}
    For~\ref{item:thresholdGagdetBehavior:i}, if the first removed gadget vertex is $r$, remove the vertices of $T(d,t)-r$ in parent-before-child order along a spanning tree rooted at a former neighbor of $r$.
    Every removed vertex has then lost an incident edge before removal.
    If the first removed gadget vertex is $z\neq r$, every component of $T(d,t)-\{r,z\}$ contains a neighbor of $z$, because $T(d,t)-r$ is connected.
    In each component, remove the vertices in parent-before-child order along a spanning tree rooted at such a neighbor, and remove $r$ last.
    Every non-root vertex again loses an incident edge before removal, while $r$ has degree at most $d\leq12$ when it is removed.

    For~\ref{item:thresholdGagdetBehavior:ii}, after $t$ signal neighbors have been removed the root has degree at most $12$ by~\eqref{eq:thresholdRootDegree}. Remove the root and apply assertion~\ref{item:thresholdGagdetBehavior:i}.
    
    For~\ref{item:thresholdGadgetBehavior:iii}, immediately before the first gadget vertex is removed the gadget is untouched.
    Before the $t^\text{th}$ signal, its root has degree at least $13$ by~\eqref{eq:thresholdRootDegree}, and every non-root vertex has degree $13$ by~\eqref{eq:polyPowerActivationDegrees}.
  \end{subproof}

  Claim~\ref{cl:thresholdGadgetBehavior} implies that if no vertex of a threshold gadget is removed at degree at least $13$, then the first removed vertex of the gadget is its root and at least $t$ signal neighbors have already been removed.

  A \emph{unit link} between marked roots $r_1$ and $r_2$ is a copy of $T(2,1)$ whose root is adjacent to $r_1$ and $r_2$.
  Once either endpoint is removed, the unit link can be removed at degrees at most $12$, and removing its root provides one signal to the other endpoint.
  We use unit links in place of direct edges so that multiple connections between the same pair of marked roots do not create parallel edges.

  \begin{claim}\label{cl:polyPowerUnitLinkTiming}
    Suppose every vertex of a unit link joining roots $r_1$ and $r_2$ is removed at degree at most $12$.
    Then no vertex of the unit link is removed while both endpoints remain.
    In particular, if the root of the unit link is removed before $r_2$, then $r_1$ was removed before that root.
  \end{claim}
  \begin{subproof}
    While both endpoints remain, the unit link has received no signal.
    Since its threshold is $1$, Claim~\ref{cl:thresholdGadgetBehavior}\ref{item:thresholdGadgetBehavior:iii} implies that its first removed vertex would have degree at least $13$.
    Therefore, one endpoint is removed before any vertex of the unit link.
    If the root of the unit link is removed before $r_2$, the endpoint removed earlier must be $r_1$.
  \end{subproof}

\medskip
  Next, we construct a graph $G$ on marked roots, and then replace each marked root by a threshold gadget of the same degree.
  During the construction, we assign a threshold to each marked root; this becomes the threshold of the gadget that replaces it.

  First, construct the auxiliary trees.
  Take $(m+1)^2$ copies of the expanded formula tree, called \emph{formula copies}.
  For each $i\in [n]$, construct a full binary tree whose leaves are matched bijectively to the occurrences of the variable $x_i$ in the formula copies, called the \emph{variable tree of $x_i$}.
  Each variable tree has at least $(m+1)^2\geq 4$ leaves because each variable occurs in $\Phi$.
  We now construct $G$ using these trees.
  
  First, construct the subgraphs of $G$ corresponding to the formula copies.
  For each vertex $h$ in a formula copy, add two marked roots $c_h$ and $o_h$, called the \emph{core-root} and \emph{output-root of $h$}, and connect them with six unit links.
  Let the threshold of $c_h$ be $1$ if $h$ is an OR gate or a variable occurrence, and $2$ if $h$ is an AND gate; let the threshold of $o_h$ be $6$.
  For each non-root vertex $h$ with parent $g$ in a formula copy, connect $o_h$ and $c_g$ with one unit link.
  Furthermore, connect $o_g$ and $c_h$ with one unit link if $h$ is an OR gate or a variable occurrence, and connect them with two unit links if $h$ is an AND gate.

  Second, construct the subgraph of $G$ corresponding to each variable $x_i$, where $i\in[n]$.
  For each vertex $v$ in the variable tree of $x_i$, add two marked roots $c_v$ and $o_v$, called the core-root and output-root of $v$, and connect them with six unit links.
  Let the threshold of $c_v$ be $1$ if $v$ is a leaf and $2$ if $v$ is a non-leaf vertex; let the threshold of $o_v$ be $6$.
  For each non-root vertex $v$ with parent $u$, connect $o_v$ and $c_u$ with one unit link.
  Connect $o_u$ and $c_v$ with one unit link if $v$ is a leaf and with two unit links if $v$ is a non-leaf vertex.

  Finally, connect the variable subgraphs to the formula-copy subgraphs.
  Each leaf $\ell$ in a variable tree is matched to one variable occurrence $h$ in a formula copy.
  Connect $o_{\ell}$ and $c_h$ with one unit link, and connect $c_{\ell}$ and $o_h$ with one unit link.

  Figure~\ref{fig:localConstructions} illustrates the local constructions for an OR gate, an AND gate, and a variable occurrence in a formula copy.
  The local construction for a leaf of a variable tree is isomorphic to that for a variable occurrence, while the local construction for a non-leaf vertex is isomorphic to that for an AND gate.

  \begin{figure}[h]
  \centering
  \begin{subfigure}[t]{.32\textwidth}
    \centering
    \begin{tikzpicture}[yscale=1.2, xscale=1.2]
      \tikzset{VertexStyle/.append style = {minimum size = 20pt,inner sep=0pt, shape=circle}}
      \SetVertexMath
      \Vertex[x=0.5, y=2, L=o_g]{og}
      \Vertex[x=-0.5, y=2, L=c_g]{cg}
      \Vertex[x=0.5, y=0, L=o_h]{oh}
      \Vertex[x=-0.5, y=0, L=c_h]{ch}
      \Vertex[x=-1, y=-2, L=o_a]{oa}
      \Vertex[x=-2, y=-2, L=c_a]{ca}
      \Vertex[x=2, y=-2, L=o_b]{ob}
      \Vertex[x=1, y=-2, L=c_b]{cb}

      \node() at (-0.9,0.25) {$1$};
      \node() at (0.9,0.25) {$6$};

       \draw (ch) to [out=5,in=175,looseness=1] (oh);
       \draw (ch) to [out=15,in=165,looseness=1] (oh);
       \draw (ch) to [out=25,in=155,looseness=1] (oh);
       \draw (ch) to [out=-5,in=-175,looseness=1] (oh);
       \draw (ch) to [out=-15,in=-165,looseness=1] (oh);
       \draw (ch) to [out=-25,in=-155,looseness=1] (oh);

       \draw (cg) to [out=5,in=175,looseness=1] (og);
       \draw (cg) to [out=15,in=165,looseness=1] (og);
       \draw (cg) to [out=25,in=155,looseness=1] (og);
       \draw (cg) to [out=-5,in=-175,looseness=1] (og);
       \draw (cg) to [out=-15,in=-165,looseness=1] (og);
       \draw (cg) to [out=-25,in=-155,looseness=1] (og);

       \draw (ca) to [out=5,in=175,looseness=1] (oa);
       \draw (ca) to [out=15,in=165,looseness=1] (oa);
       \draw (ca) to [out=25,in=155,looseness=1] (oa);
       \draw (ca) to [out=-5,in=-175,looseness=1] (oa);
       \draw (ca) to [out=-15,in=-165,looseness=1] (oa);
       \draw (ca) to [out=-25,in=-155,looseness=1] (oa);

       \draw (cb) to [out=5,in=175,looseness=1] (ob);
       \draw (cb) to [out=15,in=165,looseness=1] (ob);
       \draw (cb) to [out=25,in=155,looseness=1] (ob);
       \draw (cb) to [out=-5,in=-175,looseness=1] (ob);
       \draw (cb) to [out=-15,in=-165,looseness=1] (ob);
       \draw (cb) to [out=-25,in=-155,looseness=1] (ob);

       \draw (cg)--(oh);
       \draw (og)--(ch);
       \draw (ch)--(oa);
       \draw[dashed] (oh) to[bend left=7.5] (ca);
       \draw (oh) to[bend right=7.5] (ca);
       \draw (ch)--(ob);
       \draw[dashed] (oh) to[bend left=10] (cb);
       \draw (oh) to[bend right=10] (cb);

    \end{tikzpicture}
    \caption{If $h$ is an OR gate with parent $g$ and children $a$ and $b$.}\label{subfig:ORGadget}
  \end{subfigure}
  \hfill
  \begin{subfigure}[t]{.32\textwidth}
    \centering
    \begin{tikzpicture}[yscale=1.2, xscale=1.2]
      \tikzset{VertexStyle/.append style = {minimum size = 20pt,inner sep=0pt, shape=circle}}
      \SetVertexMath
      \Vertex[x=0.5, y=2, L=o_g]{og}
      \Vertex[x=-0.5, y=2, L=c_g]{cg}
      \Vertex[x=0.5, y=0, L=o_h]{oh}
      \Vertex[x=-0.5, y=0, L=c_h]{ch}
      \Vertex[x=-1, y=-2, L=o_a]{oa}
      \Vertex[x=-2, y=-2, L=c_a]{ca}
      \Vertex[x=2, y=-2, L=o_b]{ob}
      \Vertex[x=1, y=-2, L=c_b]{cb}

      \node() at (-0.9,0.25) {$2$};
      \node() at (0.9,0.25) {$6$};

      \draw (ch) to [out=5,in=175,looseness=1] (oh);
       \draw (ch) to [out=15,in=165,looseness=1] (oh);
       \draw (ch) to [out=25,in=155,looseness=1] (oh);
       \draw (ch) to [out=-5,in=-175,looseness=1] (oh);
       \draw (ch) to [out=-15,in=-165,looseness=1] (oh);
       \draw (ch) to [out=-25,in=-155,looseness=1] (oh);

       \draw (cg) to [out=5,in=175,looseness=1] (og);
       \draw (cg) to [out=15,in=165,looseness=1] (og);
       \draw (cg) to [out=25,in=155,looseness=1] (og);
       \draw (cg) to [out=-5,in=-175,looseness=1] (og);
       \draw (cg) to [out=-15,in=-165,looseness=1] (og);
       \draw (cg) to [out=-25,in=-155,looseness=1] (og);

       \draw (ca) to [out=5,in=175,looseness=1] (oa);
       \draw (ca) to [out=15,in=165,looseness=1] (oa);
       \draw (ca) to [out=25,in=155,looseness=1] (oa);
       \draw (ca) to [out=-5,in=-175,looseness=1] (oa);
       \draw (ca) to [out=-15,in=-165,looseness=1] (oa);
       \draw (ca) to [out=-25,in=-155,looseness=1] (oa);

       \draw (cb) to [out=5,in=175,looseness=1] (ob);
       \draw (cb) to [out=15,in=165,looseness=1] (ob);
       \draw (cb) to [out=25,in=155,looseness=1] (ob);
       \draw (cb) to [out=-5,in=-175,looseness=1] (ob);
       \draw (cb) to [out=-15,in=-165,looseness=1] (ob);
       \draw (cb) to [out=-25,in=-155,looseness=1] (ob);

       \draw (cg)--(oh);
       \draw (og) to[bend left=10] (ch);
       \draw (og) to[bend right=10] (ch);
       \draw (ch)--(oa);
       \draw[dashed] (oh) to[bend left=7.5] (ca);
       \draw (oh) to[bend right=7.5] (ca);
       \draw (ch)--(ob);
       \draw[dashed] (oh) to[bend left=10] (cb);
       \draw (oh) to[bend right=10] (cb);
    \end{tikzpicture}
    \caption{If $h$ is an AND gate with\\ parent $g$ and children $a$ and $b$.}\label{subfig:ANDGadget}
  \end{subfigure}
  \hfill
  \begin{subfigure}[t]{.32\textwidth}
    \centering
    \begin{tikzpicture}[yscale=1.2, xscale=1.2]
      \tikzset{VertexStyle/.append style = {minimum size = 20pt,inner sep=0pt, shape=circle}}
      \SetVertexMath
      \Vertex[x=0.5, y=2, L=o_g]{og}
      \Vertex[x=-0.5, y=2, L=c_g]{cg}
      \Vertex[x=0.5, y=0, L=o_h]{oh}
      \Vertex[x=-0.5, y=0, L=c_h]{ch}

      \node() at (-0.9,0.25) {$1$};
      \node() at (0.9,0.25) {$6$};
      
      \Vertex[x=-0.5, y=-2, L=c_{\ell}]{cl}
      \Vertex[x=0.5, y=-2, L=o_{\ell}]{ol}

\draw (ch) to [out=5,in=175,looseness=1] (oh);
       \draw (ch) to [out=15,in=165,looseness=1] (oh);
       \draw (ch) to [out=25,in=155,looseness=1] (oh);
       \draw (ch) to [out=-5,in=-175,looseness=1] (oh);
       \draw (ch) to [out=-15,in=-165,looseness=1] (oh);
       \draw (ch) to [out=-25,in=-155,looseness=1] (oh);

       \draw (cg) to [out=5,in=175,looseness=1] (og);
       \draw (cg) to [out=15,in=165,looseness=1] (og);
       \draw (cg) to [out=25,in=155,looseness=1] (og);
       \draw (cg) to [out=-5,in=-175,looseness=1] (og);
       \draw (cg) to [out=-15,in=-165,looseness=1] (og);
       \draw (cg) to [out=-25,in=-155,looseness=1] (og);

       \draw (cl) to [out=5,in=175,looseness=1] (ol);
       \draw (cl) to [out=15,in=165,looseness=1] (ol);
       \draw (cl) to [out=25,in=155,looseness=1] (ol);
       \draw (cl) to [out=-5,in=-175,looseness=1] (ol);
       \draw (cl) to [out=-15,in=-165,looseness=1] (ol);
       \draw (cl) to [out=-25,in=-155,looseness=1] (ol);

       \draw (cg)--(oh);
       \draw (og)--(ch);
       \draw (ch)--(ol);
       \draw (oh)--(cl);
    \end{tikzpicture}
    \caption{If $h$ is a variable occurrence with parent $g$ and matched leaf~$\ell$.}\label{subfig:variableGagdet}
  \end{subfigure}
  \caption{The local constructions in $G$ for an OR gate, an AND gate or a variable occurrence $h$ in the formula copies. Each edge represents a unit link, and the numbers represent the thresholds of $c_h$ and $o_h$. Each core-output pair is joined by six unit links. In subfigures (a) and (b), the dashed edges incident to $c_a$ or $c_b$ are only present if $a$ or $b$ is an AND gate.}\label{fig:localConstructions}
  \end{figure}

  For each marked core-root or output-root, let $d$ be the number of incident unit links, and let $t$ be its assigned threshold.
  Replace the marked root with the root of a copy of $T(d,t)$, retaining its incident edges.
  For every vertex $z$ in a formula copy or a variable tree, the gadgets replacing $c_z$ and $o_z$ are called its core and output gadgets and are denoted by $C_z$ and $O_z$, respectively.

  It is straightforward to verify that a core root is incident with at most four unit links in addition to the six unit links connecting it to its output root, whereas an output root is incident with at most five additional unit links besides these six. See Figure~\ref{fig:localConstructions} for the local constructions in $G$ corresponding to the different types of formula-copy nodes. As noted above, the local construction for each variable-tree node is isomorphic to one of these cases.
  Every core-root and output-root is incident with the six unit links of its core-output pair, so $d\geq6$ and its assigned threshold satisfies $t\leq6\leq d$.
  Together with the preceding upper bounds, this gives $1\leq t\leq d\leq11$ for every core or output gadget.
  Every unit link has $d=2$ and $t=1$, so all copies of $T(d,t)$ used in the construction satisfy $1\leq t\leq d\leq12$.
  Equation~\eqref{eq:thresholdRootDegree} shows that every gadget root has initial degree at most $12+6=18$, while every non-root gadget vertex has degree $13$.
  Distinct unit links use distinct copies of $T(2,1)$, and no direct edge joins two marked roots.
  Hence, $G$ is simple and has maximum degree at most $18$.

  For each $i\in [n]$, the \emph{gadgets associated with $x_i$} are the core and output gadgets of the vertices in the variable tree of $x_i$, together with every unit link incident with the root of at least one of these gadgets.
  The \emph{gadgets associated with a formula copy} are the core and output gadgets of its vertices, together with every unit link whose two endpoints are roots corresponding to vertices of that formula copy.
  The sets for distinct variables are disjoint, the sets for distinct formula copies are disjoint, and no set associated with a variable intersects a set associated with a formula copy.
  Thus, all these gadget sets are pairwise disjoint.

  The number of gadgets associated with a formula copy is $O(m)$, so all $(m+1)^2$ copies together contain $O(m^3)$ gadgets.
  The total number of leaves in all variable trees equals the total number of variable occurrences in the $(m+1)^2$ formula copies and is $O(m^3)$.
  Since a full binary tree has fewer than twice as many vertices as leaves, all subgraphs corresponding to the variables together also contain $O(m^3)$ core and output gadgets and unit links.
  The unit links joining the variable subgraphs to the $(m+1)^2$ formula-copy subgraphs contribute another $O(m^3)$ gadgets, so the total number of gadgets associated with variables is also $O(m^3)$.
  Every gadget has constant size, so $G$ can be constructed in polynomial time and its number of vertices is
  \begin{equation}\label{eq:polyPowerGraphSize}
    N=O(m^3).
  \end{equation}

  \FloatBarrier
  \medskip
  
  We next show that $G$ has a $13$-bounded order with at most $\tau$ vertices of left-degree $13$.
  Each output gadget has threshold $6$, and its output-root is incident with at most five unit links outside the six unit links of its core-output pair.
  This gives the following removal-order property.
  \begin{claim}\label{cl:polyPowerCoreOutputBarrier}
    If the gadgets $C_z$ and $O_z$ of a core-output pair and the six unit links joining them are removed at degrees at most $12$, then the core-root $c_z$ is removed before the output-root $o_z$.
  \end{claim}
  \begin{subproof}
    Since no vertex of $O_z$ is removed at degree at least $13$, Claim~\ref{cl:thresholdGadgetBehavior}\ref{item:thresholdGadgetBehavior:iii} implies that its threshold has been met before its first vertex is removed.
    Every non-root vertex of $O_z$ has degree $13$ until a vertex of $O_z$ is removed by~\eqref{eq:polyPowerActivationDegrees}, so this first vertex is the output-root $o_z$.
    Suppose, for a contradiction, that $o_z$ is removed before $c_z$.
    If the root of one of the six unit links were removed before $o_z$, Claim~\ref{cl:polyPowerUnitLinkTiming} would imply that $c_z$ had been removed earlier.
    Hence, none of these six unit links supplies a signal to $O_z$ before $o_z$ is removed.
    The output gadget can receive at most five signals from its other incident unit links, which does not meet its threshold $6$.
  \end{subproof}

  After a core-root is removed, the six unit links joining it to its output-root can be removed at degrees at most $12$.
  Their roots provide six signals to the output gadget, so the output gadget can then be removed at degrees at most $12$.

  We now construct a $13$-bounded order from a satisfying set of true variables $X$.
  
  \begin{lemma}\label{lemma:polyPowerCompleteness}
    If $X\subseteq[n]$ satisfies $\Phi(X)=1$, then $G$ has a $13$\nobreakdash-bounded order with at most $|X|$ vertices of left-degree $13$.
  \end{lemma}
  \begin{subproof}
    For each $i\in X$, remove one non-root vertex from the core gadget $C_r$ corresponding to the root $r$ of the variable tree of $x_i$.
    These removals are performed first, and every chosen vertex has degree $13$ by~\eqref{eq:polyPowerActivationDegrees}.
    We remove all remaining vertices at degrees at most $12$ using the following rules.
    If a vertex of a threshold gadget has been removed, remove the remaining vertices of that gadget at degrees at most $12$ using Claim~\ref{cl:thresholdGadgetBehavior}\ref{item:thresholdGagdetBehavior:i}.
    Whenever a threshold gadget has met its threshold, remove its vertices at degrees at most $12$ using Claim~\ref{cl:thresholdGadgetBehavior}\ref{item:thresholdGagdetBehavior:ii}.
    When the root of a threshold gadget has been removed, remove every remaining incident unit link at degrees at most $12$, since the removed endpoint meets the threshold $1$ of each such link.
    In particular, after a core gadget has been removed, remove the six unit links joining it to its output gadget. Their roots provide the six signals needed to remove the output gadget at degrees at most $12$.

    First, process the variable subgraphs for the indices $i\in X$.
    The initial removal in $C_r$ allows the remaining vertices of $C_r$ to be removed at degrees at most $12$, followed by all incident unit links and the output gadget $O_r$.
    Continue from the root of the variable tree toward its leaves.
    Once the output gadget $O_u$ of a parent $u$ has been removed, remove the incident unit links.
    For a child $v$, this removes one unit link from $o_u$ to $c_v$ if $v$ is a leaf and two such unit links if $v$ is a non-leaf vertex.
    These removals meet the threshold of $C_v$, so $C_v$, its incident unit links, and $O_v$ can be removed in this order at degrees at most $12$.
    Hence, this process removes all gadgets associated with $x_i$ for each $i\in X$.

    Next, process each formula copy.
    First process the formula vertices that are true under $X$ in an order in which children precede their parents.
    A true variable occurrence $g$ is matched to a leaf $\ell$ of the variable tree of $x_i$ for some $i\in X$.
    The unit link from $o_{\ell}$ to $c_g$ has therefore already been removed, meeting the threshold $1$ of $C_g$.
    Thus, $C_g$, its incident unit links, and $O_g$ can be removed at degrees at most $12$.
    For a true OR gate $g$, choose a true child $h$.
    The removed unit link from $o_h$ to $c_g$ meets the threshold $1$ of $C_g$, so $C_g$, its incident unit links, and $O_g$ can be removed at degrees at most $12$.
    For a true AND gate $g$, the unit link from each child output-root $o_h$ to $c_g$ has been removed.
    These two removals meet the threshold $2$ of $C_g$, so $C_g$, its incident unit links, and $O_g$ can be removed at degrees at most $12$.
    Since $\Phi(X)=1$, this process removes the core and output gadgets of the formula root in every copy.

    Process the remaining formula vertices from the root toward the leaves.
    When a remaining non-root vertex $h$ is considered, the output gadget $O_g$ of its parent $g$ has already been removed.
    If $h$ is an OR gate or a variable occurrence, the unit link from $o_g$ to $c_h$ has also been removed; if $h$ is an AND gate, both such unit links have been removed.
    These removals meet the threshold of $C_h$, so $C_h$, its incident unit links, and $O_h$ can be removed at degrees at most $12$.

    It remains to process the variable subgraphs for the indices $i\notin X$.
    Process each such variable tree from its leaves toward its root.
    Each leaf $\ell$ is matched to a variable occurrence $h$ in a formula copy, so the unit link from $o_h$ to $c_{\ell}$ has already been removed.
    This meets the threshold $1$ of $C_{\ell}$, allowing $C_{\ell}$, its incident unit links, and $O_{\ell}$ to be removed at degrees at most $12$.
    When a non-leaf vertex $u$ is processed, the unit link from each child output-root $o_v$ to $c_u$ has already been removed.
    These two removals meet the threshold $2$ of $C_u$, so $C_u$, its incident unit links, and $O_u$ can be removed at degrees at most $12$.
    Hence, this process removes all gadgets associated with $x_i$ for each $i\notin X$.

    The only removals of degree $13$ are the $|X|$ initial removals in the root-core gadgets associated with variables in $X$; every other vertex is removed at degree at most $12$.
  \end{subproof}

  Apply Lemma~\ref{lemma:polyPowerCompleteness} to a satisfying assignment of size $\tau$, and let $s\leq\tau$ be the number of vertices of left-degree $13$ in the resulting order.
  Every other vertex has left-degree at most $12$, so
  \begin{equation}\label{eq:polyPowerObjectiveUpperBound}
    \opt_p(G)
    \leq s13^p+(N-s)12^p
    \leq \tau13^p+(N-\tau)12^p.
  \end{equation} 

  \medskip
  Consider an arbitrary removal order $\sigma$ of $G$.
  Let $n_{\geq 13}=|\{v\in V:\cev{d}_{\sigma}(v)\geq 13\}|$, and let $X_{\geq 13}\subseteq[n]$ consist of the indices $i$ for which a gadget associated with $x_i$ contains a vertex removed at degree at least $13$.
  The gadget sets associated with distinct variables are disjoint, so choosing one such vertex for each $i\in X_{\geq 13}$ gives
  \begin{equation}\label{eq:selectedVariablesAtMostTopLayer}
    |X_{\geq 13}|\leq n_{\geq 13}.
  \end{equation}
  We show that $X_{\geq 13}$ satisfies $\Phi$ whenever $n_{\geq 13}<(m+1)^2$.

  We first establish how the removal order in a variable tree constrains the matched variable occurrences in the formula copies.
  \begin{claim}\label{cl:polyPowerBroadcastTiming}
    Suppose that no gadget associated with $x_i$ contains a vertex removed at degree at least $13$.
    If a leaf $\ell$ in the variable tree of $x_i$ is matched to a variable occurrence $h$, then $o_h$ is removed before $o_{\ell}$.
  \end{claim}
  \begin{subproof}
    Let $v_0,v_1,\dots,v_k=\ell$ be the path from the root of the variable tree to $\ell$.
    We first show that $c_{\ell}$ is the first vertex removed among $c_{v_0},o_{v_0},c_{v_1},o_{v_1},\dots,c_{v_k},o_{v_k}$.
    For every $j\in\{0,\dots,k\}$, the gadgets $C_{v_j}$ and $O_{v_j}$ and the six unit links joining them are associated with $x_i$, so all their vertices are removed at degrees at most $12$.
    Claim~\ref{cl:polyPowerCoreOutputBarrier} therefore gives that $c_{v_j}$ is removed before $o_{v_j}$.
    Hence, the first vertex removed in the displayed list is $c_{v_j}$ for some $j\in\{0,\dots,k\}$.

    Suppose that $j<k$.
    Then $v_j$ is a non-leaf vertex, and the threshold of $C_{v_j}$ is $2$.
    Since no vertex of $C_{v_j}$ is removed at degree at least $13$, Claim~\ref{cl:thresholdGadgetBehavior}\ref{item:thresholdGadgetBehavior:iii} implies that its threshold has been met before its first vertex is removed.
    Every non-root vertex of $C_{v_j}$ has degree $13$ until a vertex of $C_{v_j}$ is removed by~\eqref{eq:polyPowerActivationDegrees}, so this first vertex is $c_{v_j}$.
    Immediately before that removal, both endpoints of each of the six unit links between $c_{v_j}$ and $o_{v_j}$ remain, so Claim~\ref{cl:polyPowerUnitLinkTiming} implies that none of those links has been removed.
    If $j>0$, both endpoints of each of the two unit links from $o_{v_{j-1}}$ to $c_{v_j}$ also remain, so Claim~\ref{cl:polyPowerUnitLinkTiming} implies that neither link has been removed.
    The output-root $o_{v_{j+1}}$ of the child on the path to $\ell$ remains as well, so Claim~\ref{cl:polyPowerUnitLinkTiming} implies that the unit link from $o_{v_{j+1}}$ to $c_{v_j}$ has not been removed.
    The only remaining unit link incident with $c_{v_j}$ joins it to the output-root of the other child of $v_j$, so $C_{v_j}$ can have received at most one signal.
    This does not meet the threshold $2$, a contradiction.
    Therefore, $j=k$ and $c_{\ell}$ is the first vertex removed in the displayed list.

    No vertex of $C_{\ell}$ is removed at degree at least $13$, so Claim~\ref{cl:thresholdGadgetBehavior}\ref{item:thresholdGadgetBehavior:iii} implies that its threshold $1$ has been met before its first vertex is removed.
    Every non-root vertex of $C_{\ell}$ has degree $13$ until a vertex of $C_{\ell}$ is removed by~\eqref{eq:polyPowerActivationDegrees}, so this first vertex is $c_{\ell}$.
    Since every variable tree has at least four leaves, $\ell$ is not the root of its variable tree and has a parent.
    At that time, both endpoints of the six unit links between $c_{\ell}$ and $o_{\ell}$ and of the unit link from the output-root of the parent of $\ell$ to $c_{\ell}$ remain.
    Claim~\ref{cl:polyPowerUnitLinkTiming} implies that none of these unit links has been removed.
    Consequently, the root of the unit link joining $o_h$ and $c_{\ell}$ was removed before $c_{\ell}$.
    This unit link is associated with $x_i$ and contains no vertex removed at degree at least $13$.
    Therefore, Claim~\ref{cl:polyPowerUnitLinkTiming} gives that $o_h$ was removed before the root of this unit link.
    Claim~\ref{cl:polyPowerCoreOutputBarrier} gives that $c_{\ell}$ was removed before $o_{\ell}$.
    Hence, $o_h$ was removed before $o_{\ell}$.
  \end{subproof}

  We next relate the removal order within a formula copy to the Boolean value of its formula vertices.
  \begin{claim}\label{cl:justifiedFormulaNode}
    Let $h$ be a vertex in a formula copy, and suppose that no gadget associated with this formula copy contains a vertex removed at degree at least $13$.
    Suppose further that $h$ is the formula root, or that $h$ has parent $g$ and $o_h$ is removed before $o_g$.
    Then the following statements hold.
    \begin{enumerate}[label=(\roman*)]
      \item If $h$ is an OR gate, then $o_a$ is removed before $o_h$ for some child $a$ of $h$.\label{item:justifiedFormulaNode:i}
      \item If $h$ is an AND gate, then $o_a$ is removed before $o_h$ for each child $a$ of $h$.\label{item:justifiedFormulaNode:ii}
      \item If $h$ is an occurrence of $x_i$, then $i\in X_{\geq 13}$.\label{item:justifiedFormulaNode:iii}
    \end{enumerate}
  \end{claim}
  \begin{subproof}
    Fix a vertex $h$ satisfying the hypotheses.
    Claim~\ref{cl:polyPowerCoreOutputBarrier} gives that $c_h$ is removed before $o_h$.
    If $h$ is not the formula root, the additional hypothesis gives $c_h$ before $o_h$ before the output-root $o_g$ of its parent.
    Therefore, immediately before $c_h$ is removed, both endpoints of each of the six unit links between $c_h$ and $o_h$ remain, and
    if $h$ is not the formula root, both endpoints of every unit link from $o_g$ to $c_h$ also remain.
    All these unit links are associated with the formula copy and contain no vertex removed at degree at least $13$.
    Claim~\ref{cl:polyPowerUnitLinkTiming} therefore implies that none of them has been removed by that time.
    Moreover, no vertex of $C_h$ is removed at degree at least $13$.
    Claim~\ref{cl:thresholdGadgetBehavior}\ref{item:thresholdGadgetBehavior:iii} therefore implies that its threshold has been met before its first vertex is removed.
    Every non-root vertex of $C_h$ has degree $13$ until a vertex of $C_h$ is removed by~\eqref{eq:polyPowerActivationDegrees}, so this first vertex is $c_h$.

    If $h$ is an OR gate, its core gadget has threshold $1$.
    The preceding paragraph rules out every incident unit link except the links from the output-roots of its children, so the root of at least one unit link joining $c_h$ to a child output-root $o_a$ was removed before $c_h$.
    This unit link is associated with the formula copy and contains no vertex removed at degree at least $13$.
    Claim~\ref{cl:polyPowerUnitLinkTiming} gives that $o_a$ was removed before the root of the unit link, and hence before $o_h$.
    This proves~\ref{item:justifiedFormulaNode:i}.

    If $h$ is an AND gate, its core gadget has threshold $2$.
    The formula is binary, and the only incident unit links not ruled out above are the two links joining $c_h$ to the output-roots of its two children.
    Both unit links must therefore have been removed before $c_h$.
    Applying Claim~\ref{cl:polyPowerUnitLinkTiming} to each link shows that the output-root of each child was removed before $c_h$, and hence before $o_h$.
    This proves~\ref{item:justifiedFormulaNode:ii}.

    Suppose that $h$ is an occurrence of $x_i$, and let $\ell$ be its matched leaf in the variable tree of $x_i$.
    The core gadget $C_h$ has threshold $1$, and the only incident unit link not ruled out above is the link joining $c_h$ and $o_{\ell}$.
    Hence, the root of this unit link was removed before $c_h$.
    If the unit link contains a vertex removed at degree at least $13$, then it is associated with $x_i$, and the definition of $X_{\geq 13}$ gives $i\in X_{\geq 13}$.
    Otherwise, Claim~\ref{cl:polyPowerUnitLinkTiming} gives that $o_{\ell}$ was removed before the root of the unit link, and hence before $c_h$ and $o_h$.
    If $i\notin X_{\geq 13}$, no gadget associated with $x_i$ contains a vertex removed at degree at least $13$, so Claim~\ref{cl:polyPowerBroadcastTiming} gives that $o_h$ was removed before $o_{\ell}$.
    This contradicts the preceding order.
    Therefore, $i\in X_{\geq 13}$, proving~\ref{item:justifiedFormulaNode:iii}.
  \end{subproof}

  \begin{claim}\label{cl:satisfiedCopy}
    If no gadget associated with a formula copy contains a vertex removed at degree at least $13$, then $\Phi(X_{\geq 13})=1$.
  \end{claim}
  \begin{subproof}
    Apply Claim~\ref{cl:justifiedFormulaNode} to the formula root.
    Whenever the claim is applied to an OR gate, continue with one child supplied by Claim~\ref{cl:justifiedFormulaNode}\ref{item:justifiedFormulaNode:i}.
    Whenever it is applied to an AND gate, continue with both children supplied by Claim~\ref{cl:justifiedFormulaNode}\ref{item:justifiedFormulaNode:ii}.
    Each child reached in this process satisfies the hypothesis of Claim~\ref{cl:justifiedFormulaNode}, because its output-root was removed before the output-root of its parent.
    The process ends at variable occurrences, and Claim~\ref{cl:justifiedFormulaNode}\ref{item:justifiedFormulaNode:iii} shows that every occurrence reached in the process is true under $X_{\geq 13}$.
    Moving from these occurrences toward the formula root, every OR gate reached in the process has a true child and every AND gate reached in the process has two true children.
    Hence, the formula root is true under $X_{\geq 13}$, so $\Phi(X_{\geq 13})=1$.
  \end{subproof}

  This implies the following key lemma.
  
  \begin{lemma}\label{lemma:fewLargeDegrees}
    If $n_{\geq 13}<(m+1)^2$, then $\Phi(X_{\geq 13})=1$.
  \end{lemma}
  \begin{subproof}
    The gadget sets associated with the $(m+1)^2$ formula copies are pairwise disjoint.
    If every formula copy contained a vertex removed at degree at least $13$, these copies would contain at least $(m+1)^2$ distinct such vertices.
    Since $n_{\geq 13}<(m+1)^2$, at least one formula copy contains no vertex removed at degree at least $13$.
    Claim~\ref{cl:satisfiedCopy} then gives $\Phi(X_{\geq 13})=1$.
  \end{subproof}

  \medskip
  We choose $p$ so that, after multiplication by the assumed approximation ratio $2^{p^{1-\varepsilon}}$, the total contribution of all vertices removed at degrees at most $12$ is less than $13^p$.
  Choose an integer $K$, depending only on $\varepsilon$, so large that $K\log_2\left(\frac{13}{12}\right)>2^{-\varepsilon}K^{1-\varepsilon}+1$.
  Such an integer exists because $(2K)^{-\varepsilon}+K^{-1}$ tends to $0$ as $K$ tends to infinity.
  Let $L=\lceil\log_2(N+2)\rceil$ and $p=KL$.
  Since $L\geq2$ and $N<2^L$,
  \begin{equation}\label{eq:lowDegreePowerGap}
    \begin{aligned}
      2^{p^{1-\varepsilon}}N\left(\frac{12}{13}\right)^p
      &<2^{(KL)^{1-\varepsilon}}2^L\left(\frac{12}{13}\right)^{KL}
      =2^{L\left(K^{1-\varepsilon}L^{-\varepsilon}+1-K\log_2(\frac{13}{12})\right)}\\
      &\leq2^{L\left(2^{-\varepsilon}K^{1-\varepsilon}+1-K\log_2(\frac{13}{12})\right)}
      <1.
    \end{aligned}
  \end{equation}
  Moreover, $p=O(\log N)$ and can be computed in polynomial time, and the bit-length of $p$ is $O(\log\log N)$.

  Run the assumed approximation algorithm on $(G,p)$, let $\sigma$ be the returned order, and construct $n_{\geq 13}$ and $X_{\geq 13}$ from $\sigma$ as above.
  By~\eqref{eq:polyPowerApproximationAssumption} and~\eqref{eq:polyPowerObjectiveUpperBound},
  \[
    n_{\geq 13}13^p
    \leq
    \sum_{v\in V}\cev{d}_{\sigma}(v)^p
    \leq
    2^{p^{1-\varepsilon}}\opt_p(G)
    \leq
    2^{p^{1-\varepsilon}}\bigl(\tau 13^p+N12^p\bigr).
  \]
  Dividing by $13^p$ and using~\eqref{eq:lowDegreePowerGap} gives
  \[
    n_{\geq 13}
    \leq2^{p^{1-\varepsilon}}\tau+2^{p^{1-\varepsilon}}N\left(\frac{12}{13}\right)^p
    <2^{p^{1-\varepsilon}}\tau+1
    \leq\bigl(2^{p^{1-\varepsilon}}+1\bigr)\tau,
  \]
  where the last inequality uses $\tau\geq1$.

  By~\eqref{eq:polyPowerGraphSize}, $p=O(\log m)$.
  Since $\varepsilon>0$,
  \[
    2^{p^{1-\varepsilon}}+1=2^{O((\log m)^{1-\varepsilon})}=m^{o(1)}.
  \]
  Together with~\eqref{eq:mmsaCardinalityBounds}, this gives
  \[
    n_{\geq 13}
    <\bigl(2^{p^{1-\varepsilon}}+1\bigr)\tau
    \leq\bigl(2^{p^{1-\varepsilon}}+1\bigr)m
    =m^{1+o(1)}
    <(m+1)^2
  \]
  for all sufficiently large $m$.
  Lemma~\ref{lemma:fewLargeDegrees} then gives $\Phi(X_{\geq 13})=1$.
  Equation~\eqref{eq:selectedVariablesAtMostTopLayer} and the preceding bound give
  \[
    |X_{\geq 13}|\leq n_{\geq 13}<\bigl(2^{p^{1-\varepsilon}}+1\bigr)\tau.
  \]
  Thus, for all sufficiently large instances, the assumed algorithm yields in polynomial time a $\bigl(2^{p^{1-\varepsilon}}+1\bigr)$\nobreakdash-approximation for $\mathrm{MMSA}_3$.

  Since $p=O(\log m)$, there is a fixed constant $c>0$ such that
  \[
    2^{p^{1-\varepsilon}}+1
    \leq
    2^{c(\log m)^{1-\varepsilon}}
  \]
  for all sufficiently large $m$.
  All lower bounds on $m$ used above are constants independent of the input instance.
  For the finitely many smaller encoding sizes, an optimal satisfying assignment can be found by exhaustive search over all subsets of the variable set.
  Consequently, the assumed algorithm would give a deterministic polynomial-time $2^{c(\log m)^{1-\varepsilon}}$\nobreakdash-approximation for $\mathrm{MMSA}_3$.
  This contradicts the consequence of the Dinur--Safra inapproximability result stated above with $\eta=\varepsilon$.
\end{proof}

\subsection{The quadratic objective and the smallest-last algorithm}\label{sec:minSquareSum}
Next, we focus on the natural special case $p=2$, that is, minimizing the sum of squared left-degrees over the vertex orders of a graph.
Corollary~\ref{cor:fixedPowerNoPTAS} shows that this problem is APX-complete even on connected simple graphs and gives a strongly polynomial-time $4$\nobreakdash-approximation algorithm.
We now analyze the approximation performance of the natural smallest-last greedy algorithm: at each iteration, the algorithm selects a vertex of minimum degree, assigns it to the last available position in the order, and removes it from the graph.
Observe that when multiple vertices share the minimum degree, the output may depend on which vertex is chosen.
We note that this algorithm is essentially the same as the algorithm for computing the degeneracy $\cev{d}_{\min}$~\cite{matula1983smallest}.

In what follows, we establish two different upper bounds on the approximation ratio of this algorithm for minimizing the sum of squared left-degrees.
The first bound provides a stronger guarantee for sparse graphs, while the second is tighter for dense graphs.

\begin{theorem}\label{thm:degApprox}
  Let $G = (V, E)$ be a loop-free multigraph with $E\neq\emptyset$.
  The smallest-last greedy algorithm finds a $\cev{d}_{\min}(G)$\nobreakdash-ap\-prox\-i\-mate order for minimizing the sum of squared left-degrees.
\end{theorem}
\begin{proof}
  Consider the order produced by the smallest-last greedy algorithm.
  For each edge $e$, let $\pi(e) = \cev{d}(v)$, where $v$ is the endpoint of $e$ that appears later in the order.
  Then, $\sum_{e \in E} \pi(e) = \sum_{v \in V} \cev{d}(v)^2$.

  Since the algorithm is known to compute the degeneracy number by finding a $\cev{d}_{\min}(G)$-bounded order, we have $\cev{d}(v) \leq \cev{d}_{\min}(G)$ for every $v$.
  Therefore, $\pi(e) \leq \cev{d}_{\min}(G)$ for all $e \in E$, and the objective value satisfies $\operatorname{obj} \leq |E| \cdot \cev{d}_{\min}(G)$.
  On the other hand, each edge must contribute at least $1$ to the sum of squared left-degrees in any order, so $\opt \geq |E|$.

  Thus, the approximation ratio is bounded by $\frac{\operatorname{obj}}{\opt} \leq \frac{|E| \cdot \cev{d}_{\min}(G)}{|E|} = \cev{d}_{\min}(G)$.
\end{proof}

We now derive another upper bound on the approximation ratio of the smallest-last greedy algorithm.

\begin{theorem}
  Let $G = (V, E)$ be a loop-free multigraph with $n=|V|$.
  The smallest-last greedy algorithm produces a $4\eta_n$\nobreakdash-ap\-prox\-i\-mate order for minimizing the sum of squared left-degrees, where $\eta_n = \sum_{i=1}^n \frac{1}{i}$ is the $n^{\text{th}}$ harmonic number.
\end{theorem}
\begin{proof}
  Let $(v_1, \dots, v_n)$ denote the order produced by the smallest-last greedy algorithm, and let $\operatorname{obj}$ denote its objective value.
  For $i = 1, \dots, n$, define $V_i = \{ v_1, \dots, v_i \}$, and let $G[V_i] = (V_i, E_i)$ be the subgraph of $G$ induced by $V_i$.
  Let $m_i = |E_i|$ and $\opt(V_i)$ denote the minimum sum of squared left-degrees in $G[V_i]$.

  We first give a lower bound:
  \begin{equation}\label{eq:lowerBoundForSquareSum}
    \opt(V) \geq \opt(V_i) \geq \sum_{v \in V_i} \left( \frac{m_i}{i} \right)^2 = \frac{m_i^2}{i},
  \end{equation}
  where the first inequality follows because any optimal order for $G$ induces an order on $V_i$, and the second by Jensen's inequality and the convexity of the square function.

  For the smallest-last order, we have
  \[
    \operatorname{obj} = \sum_{i=1}^n \cev{d}(v_i)^2 \leq \sum_{i=1}^n \left( \frac{2 m_i}{i} \right)^2 = 4 \sum_{i=1}^n \frac{1}{i} \frac{m_i^2}{i},
  \]
  where the inequality follows because $v_i$ is a vertex with minimum degree in $G[V_i]$, in which the average degree is $\frac{2m_i}{i}$.

  Applying~\eqref{eq:lowerBoundForSquareSum}, this gives
  \[
    \operatorname{obj} \leq 4 \sum_{i=1}^n \frac{1}{i} \opt(V) = 4 \eta_n \opt(V).
  \]
  \vspace{-8mm}

\end{proof}

The graph family shown in Figure~\ref{fig:squareSum9/7Approx} provides an example where the worst-case approximation ratio of the smallest-last greedy algorithm approaches $\frac{9}{7}$, which is the best lower bound we have found so far.
\begin{figure}[h]
  \centering
  \begin{tikzpicture}[yscale=1.4,xscale=0.9]
    \tikzset{VertexStyle/.append style = {minimum size = 20pt,inner sep=0pt}}
    \SetVertexMath
    \Vertex[x=0,y=0, L=v_1^1]{v1}
    \Vertex[x=1,y=1, L=v_1^2]{v2}
    \Vertex[x=2,y=0, L=v_1^3]{v3}
    \Vertex[x=4,y=0, L=u_2]{v4}
    \Vertex[x=6,y=0, L=v_2^1]{v5}
    \Vertex[x=7,y=1, L=v_2^2]{v6}
    \Vertex[x=8,y=0, L=v_2^3]{v7}

    \draw (v1)--(v2);
    \draw (v2)--(v3);
    \draw (v3)--(v1);
    \draw (v5)--(v6);
    \draw (v6)--(v7);
    \draw (v7)--(v5);
    \draw (v3)--(v4);
    \draw (v4)--(v5);

    \draw [thick, decorate,decoration={brace, amplitude=4pt}, xshift=0cm,yshift=-0.3cm]
    (8.4,0) -- (2.4,0) node [black,midway,xshift=0cm,yshift=-0.4cm]
    {$\times (k - 1)$};

  \end{tikzpicture}
  \caption{Illustration of the graph sequence $G_k$ for which the approximation ratio of the smallest-last greedy algorithm approaches $\frac{9}{7}$.
    Here, $k$ denotes the number of triangles.
    Each $v_i^j$ represents the $j^\text{th}$ vertex of the $i^{\text{th}}$ triangle ($j \in \{1, 2, 3\}$), and $u_i$ connects $v_{i-1}^3$ to $v_i^1$ for $i \in \{2, \dots, k\}$.}\label{fig:squareSum9/7Approx}
\end{figure}
To analyze the worst case, we compare the performance of the smallest-last greedy algorithm against the optimal order for the graph shown in Figure~\ref{fig:squareSum9/7Approx}, consisting of $k$ triangles and paths of length two between them.
Consider the execution of the smallest-last greedy algorithm that orders the vertices as follows: it begins with the vertices $v_i^1, v_i^2, v_i^3$ for $i = 1, \dots, k$, followed by the vertices $u_i$ for $i = 2, \dots, k$.
Note that every vertex in this order has minimum degree in the subgraph induced by it and the preceding vertices.
Thus, this order is a possible output of the smallest-last greedy algorithm, and the objective value is $k(0^2 + 1^2 + 2^2) + (k-1) (2^2) = 9k - 4$.
On the other hand, consider the following optimal order: list the vertices $v_1^1$, $v_1^2$, $v_1^3$ first, and then for each $i = 2, \dots, k$, list $u_i$, $v_i^1$, $v_i^2$, $v_i^3$ in this order.
The corresponding objective value is $0^2 + 1^2 + 2^2 + (k-1)(1^2 + 1^2 + 1^2 + 2^2) = 7k - 2$.
Therefore, the approximation ratio is $\frac{9k - 4}{7k - 2}$, which tends to $\frac{9}{7}$ as $k$ goes to infinity.

We tested the approximation ratio of the worst solution found by the smallest-last greedy algorithm for every simple graph on at most $12$ vertices, and found that the ratio for the graph $G_k$ shown in Figure~\ref{fig:squareSum9/7Approx} serves as an upper bound for every simple graph on at most $(4k - 1)$ vertices.
It remains an open question whether the worst-case approximation ratio of the smallest-last greedy algorithm for simple graphs is indeed $\frac{9}{7}$.

\section{Lexicographically optimal acyclic orientations}\label{sec:decMinIncMax}
This section studies Problems~\ref{prob:decMin} and~\ref{prob:incMax}. We first show that both problems are special cases of Problem~\ref{prob:sumPhiRhoV}, with strictly convex functions $\varphi_G$ that depend on the input graph $G=(V,E)$.
Consequently, Corollary~\ref{cor:minSumPhiStrictlyConvexNPH} implies that both problems are NP-hard.
We then show that, in contrast to the unconstrained setting, the sets of optimal solutions for the dec-min and inc-max objectives need not coincide when restricted to acyclic orientations. In particular, the optimal acyclic orientations minimizing the $\varphi$\nobreakdash-sum of indegrees may depend on the choice of the discrete strictly convex function $\varphi$.

Next, we investigate natural first steps of these optimization problems.
First, we give a strongly polynomial-time algorithm for minimizing the maximum (weighted) indegree over acyclic orientations, thereby identifying a tractable first step of the dec-min optimization problem.
We then consider the problem of minimizing the number of vertices attaining the maximum indegree and show that this problem is APX-hard.
This line of investigation is motivated by~\cite{borradaile2017egalitarian}, where the authors posed the problem of approximating the dec-min objective under acyclicity.
Similarly, we establish APX-hardness for a natural first step of the inc-max problem.

Finally, we turn to the complementary increasingly minimal (inc-min) and decreasingly maximal (dec-max) objectives.

\subsection{Separable-objective formulations and divergence of dec-min and inc-max}
We begin with a simple observation that connects the dec-min and inc-max objectives to minimizing the $\varphi$\nobreakdash-sum of indegrees for a discrete strictly convex function~$\varphi$.
\begin{claim}\label{cl:decMinIncMaxAreSpecial}
  For a loop-free multigraph $G=(V,E)$, the dec-min and inc-max acyclic orientation problems are special cases of minimizing the $\varphi_G$\nobreakdash-sum of indegrees over acyclic orientations of $G$, corresponding to the choices $\varphi_G(z)=(|V|+1)^z$ and $\varphi_G(z)=(|V|+1)^{-z}$, respectively.
\end{claim}
\begin{proof}
  We present the proof for the dec-min case only, as the other one follows analogously.
  Let $n=|V|$.
  Consider two orientations $D$ and $D'$ of $G$ with non-increasingly ordered indegree sequences $\ell_1 \geq \dots \geq \ell_n$ and $\ell'_1 \geq \dots \geq \ell'_n$, respectively, and suppose that $D$ is better than $D'$ for the dec-min objective.
  We show that $\sum_{i=1}^n (n+1)^{\ell_i} < \sum_{i=1}^n (n+1)^{\ell'_i}$.
  Let $k$ denote the smallest index such that $\ell_i = \ell'_i$ for all $i < k$, and $\ell_k < \ell'_k$.
  Observe that the following upper bound holds for the objective value of $D$:
  \begin{equation}\label{eq:objDecMin}
    \sum_{i=1}^n (n+1)^{\ell_i} \leq \sum_{i=1}^{k-1} (n+1)^{\ell_i}+(n-k+1)(n+1)^{\ell_k} < \sum_{i=1}^{k-1} (n+1)^{\ell_i}+(n+1)^{\ell_k+1}.
  \end{equation}
  The first inequality follows by increasing the last $(n-k+1)$ indegrees to $\ell_k$, and the strict inequality follows from $n-k+1<n+1$.
  For the objective value of $D'$, we obtain a complementary lower bound:
  \begin{equation}\label{eq:objDecRandomOrder}
    \sum_{i=1}^n (n+1)^{\ell'_i}>\sum_{i=1}^k (n+1)^{\ell'_i} = \sum_{i=1}^{k-1} (n+1)^{\ell_i}+(n+1)^{\ell'_k}\geq \sum_{i=1}^{k-1} (n+1)^{\ell_i}+(n+1)^{\ell_k+1}.
  \end{equation}
  The strict inequality holds because the two indegree sequences have the same sum and therefore differ at some position after $k$, so $k<n$; the other two steps follow by the definition of $k$ and by simple rearrangements.
  Combining~\eqref{eq:objDecMin} and~\eqref{eq:objDecRandomOrder}, we obtain that $D$ is also strictly better than $D'$ for the objective $\min\sum_{v\in V} \varphi_G(\varrho(v))$ with $\varphi_G(z) = (n+1)^z$.
  Thus the dec-min objective and this $\varphi_G$\nobreakdash-sum induce the same ordering on indegree sequences and have the same optimal orientations.
\end{proof}
Note that this proof also shows that the claim holds not only over acyclic orientations but also over an arbitrary subset of orientations.

Corollary~\ref{cor:minSumPhiStrictlyConvexNPH} applied to the discrete strictly convex functions $\varphi_G(z)=(|V|+1)^z$ and $\varphi_G(z)=(|V|+1)^{-z}$, and Claim~\ref{cl:decMinIncMaxAreSpecial} imply the following NP-hardness statement.
\begin{corollary}\label{cor:decMinIncMaxNPH}
  The dec-min and inc-max acyclic orientation problems are NP-hard even on connected simple graphs.
  Moreover, for every fixed integer $k\geq2$, the hardness remains valid when restricted to connected simple graphs satisfying $k=\left\lceil \frac{|E|}{|V|}\right\rceil$.
\end{corollary}
We note that NP-hardness of the acyclic dec-min problem was already shown in~\cite{borradaile2017egalitarian}.

\medskip
Recall that, when orientations need not be acyclic, the dec-min and inc-max orientation problems are equivalent.
This raises the natural question of whether this equivalence continues to hold for acyclic orientations.
We answer this in the negative by exhibiting a counterexample.
\begin{figure}[h]
  \centering
  \begin{tikzpicture}[yscale=0.75]
    \tikzset{VertexStyle/.append style = {minimum size = 20pt,inner sep=0pt}}
    \SetVertexMath
    \Vertex[x=0, y=0, L=v_9]{v8}
    \SetVertexNoLabel
    \grEmptyCycle[form=1,x=0,y=0,RA=2, rotation=-67.5, prefix=v]{8}
    \AssignVertexLabel{v}{$v_1$, $v_2$, $v_3$, $v_4$, $v_5$, $v_6$, $v_7$, $v_8$}

    \draw (v0)--(v1);
    \draw (v2)--(v3);
    \draw (v3)--(v4);
    \draw (v4)--(v5);
    \draw (v5)--(v6);
    \draw (v6)--(v7);
    \draw (v7)--(v0);
    \draw (v0)--(v2);
    \draw (v0)--(v3);
    \draw (v1)--(v3);
    \draw (v4)--(v6);
    \draw (v5)--(v7);

    \draw (v8)--(v1);
    \draw (v8)--(v2);
    \draw (v8)--(v4);
    \draw (v8)--(v5);
    \draw (v8)--(v6);
    \draw (v8)--(v7);
  \end{tikzpicture}
  \caption{The smallest simple graph for which the acyclic dec-min and inc-max problems are not the same.}\label{fig:decMinIncMaxDifference}
\end{figure}
\begin{claim}\label{cl:decMinIncMaxDifference}
  For the simple graph shown in Figure~\ref{fig:decMinIncMaxDifference}, the sets of optimal acyclic orientations for the dec-min and inc-max problems are disjoint.
\end{claim}
\begin{proof}
  We adopt the vertex ordering perspective, and proceed in two steps.
  First, we show that the non-increasing left-degree sequence of dec-min orders is $3,3,3,3,2,2,1,1,0$.
  Second, we prove that there exists an inc-max order whose non-decreasing left-degree sequence is $0,1,2,2,2,2,2,3,4$.
  Since the latter is better for inc-max than the (reverse of the) former, the set of optimal acyclic orientations for the two problems must be~disjoint.

  For the first part, observe that there exist $3$\nobreakdash-bounded orders for this graph; consequently, any dec-min order must be $3$\nobreakdash-bounded.
  Initially, only $v_2$ and $v_3$ have degree at most $3$; deleting either of them reduces the degrees of both $v_1$ and $v_4$ to $3$.
  Therefore, the last two vertices in any $3$\nobreakdash-bounded order must belong among $v_1,v_2,v_3,v_4$, and both of these last two vertices have left-degree $3$.
  Next, consider the subgraph induced by the vertices $v_5, \dots, v_9$.
  Observe that each vertex has at least three neighbors within this induced subgraph; only $v_5$ and $v_8$ have exactly three, and deleting one of them does not create a vertex with degree smaller than $3$.
  Therefore, in any $3$\nobreakdash-bounded order, the last two vertices among $v_5,\dots,v_9$ must have left-degree at least $3$.
  Combining these observations with the fact that the first vertex has left-degree $0$, we conclude that the non-increasing left-degree sequence of any dec-min order is lexicographically at least $3,3,3,3,2,2,1,1,0$.
  The order $v_9, v_8, v_7, v_6, v_5, v_4, v_3, v_1, v_2$ attains this, so this left-degree sequence is the optimal dec-min sequence.

  For the second part, one can verify directly that the order $v_1, v_4, v_2, v_3, v_9, v_5, v_6, v_7, v_8$ produces the non-decreasing left-degree sequence $0,1,2,2,2,2,2,3,4$.
  This sequence is lexicographically larger than the non-decreasing sequence of dec-min orders, namely $0,1,1,2,2,3,3,3,3$.
  Therefore, the sets of optimal dec-min and inc-max orientations are disjoint.
\end{proof}

Note that the same example demonstrates that the optimal solutions for $\varphi(z)=z^2$ differ from those for the dec-min and inc-max problems over acyclic orientations.
It is also worth noting that the maximum indegree in every dec-min acyclic orientation is at most the degeneracy number $\cev{d}_{\min}(G)$, but the example in Figure~\ref{fig:decMinIncMaxDifference} demonstrates that this does not necessarily apply for inc-max acyclic orientations.

\begin{remark}
  Beyond the exponential functions representing dec-min and inc-max, we also observed divergence for polynomial objectives.
  For instance, the optimal acyclic orientations for $\varphi(z)=z^2$ and for $\varphi(z)=z^3$ do not coincide in general.
  We conjecture that, for every pair of distinct integers $p,q\geq2$, there exists a graph for which the sets of optimal acyclic orientations minimizing the $p$\nobreakdash-power sum and the $q$\nobreakdash-power sum of indegrees differ.
  Surprisingly, the cubic objective does not simply refine the quadratic one: there exist graphs where an optimal solution for $\varphi(z)=z^3$ fails to be optimal for $\varphi(z)=z^2$.
  This underlines the richness of the family of problems defined by minimizing the $\varphi$\nobreakdash-sum of indegrees under acyclicity, and highlights a clear distinction from the universality in the unconstrained setting.~$\bullet$
\end{remark}

 Now we consider some natural first steps toward dec-min and inc-max optimization.

\subsection{Minimizing the maximum (weighted) indegree}\label{sec:decMinIncMax:relax}
The first step toward finding a dec-min orientation is minimizing the maximum indegree.
This subsection focuses on this problem, which has been studied extensively in the context of unconstrained (possibly cyclic) orientations~\cite{asahiro2007graph, de1995regular, venkateswaran2004minimizing}.

A widely used approach begins with an arbitrary orientation and iteratively reverses directed paths to reduce the indegree of a vertex with highest indegree.
In a natural generalization, we consider edge-weighted graphs: given a graph $G = (V, E)$ with a non-negative weight function $w : E \to \R_+$, our goal is to find an orientation of the edges that minimizes the maximum weighted indegree across all vertices.
Unlike the unweighted case, this problem is weakly NP-hard even for simple bipartite planar graphs~\cite{asahiro2007graph}.
Furthermore, this result was strengthened in~\cite{asahiro2011approximation}, which shows NP-hardness even when the weight function is restricted to values in $\{1, k\}$ for some fixed $k \geq 2$.

We now turn to the problem of minimizing the maximum (weighted) indegree under the additional constraint that the orientation is acyclic.
Recall from Section~\ref{sec:introduction} that an unweighted graph is $k$\nobreakdash-degenerate if it admits an acyclic orientation with maximum indegree at most $k$~\cite{lick1970k}, and the smallest such $k$ is the degeneracy of the graph, denoted by $\cev{d}_{\min}(G)$.
A linear-time algorithm for computing $\cev{d}_{\min}(G)$ was given in~\cite{matula1983smallest}.
The algorithm repeatedly removes a vertex of minimum degree from the current subgraph and places it at the last available position.
The resulting order is $\cev{d}_{\min}(G)$\nobreakdash-bounded.
Orienting the edges from earlier to later vertices yields an acyclic orientation that minimizes the maximum indegree.

We now extend this approach to the weighted setting.
Given a graph $G = (V, E)$ with a non-negative weight function $w : E \to \R_+$, our goal is to find an acyclic orientation that minimizes the maximum weighted indegree.
Equivalently, we seek a vertex ordering that minimizes the maximum weighted left-degree.
Algorithm~\ref{alg:weightedMinMaxIndegreeAcyclic} solves this problem in strongly polynomial time.

\begin{algorithm}
  \caption{\hspace{0.5cm}\textsc{Weighted smallest-last ordering}}\label{alg:weightedMinMaxIndegreeAcyclic}
  \begin{algorithmic}[1]
    \Input A graph $G = (V, E)$ with a weight function $w : E \to \R_+$.
    \Output A vertex order of $G$ that minimizes $\max \{\cev{d}^w(v) : v \in V\}$.
    \State $V' \coloneqq V$; $n \coloneqq |V|$
    \For {$i = n, \dots, 1$}
    \State Let $\sigma_i$ be any vertex in $\argmin \{d^w(v,V') : v \in V'\}$.\label{alg:line:select}
    \State $V' \coloneqq V' \setminus \{\sigma_i\}$
    \EndFor
    \State \Return $(\sigma_1, \dots, \sigma_n)$
  \end{algorithmic}
\end{algorithm}
\FloatBarrier

The algorithm constructs the vertex order from right to left.
At each step, it selects a vertex with the smallest weighted degree in the current subgraph and places it at the last available position.

We now show the correctness of this algorithm.

\begin{theorem}\label{thm:weightedMinMaxIndegreeAcyclic}
  Given a loop-free multigraph $G = (V,E)$ with a non-negative weight function $w: E \to \R_+$, Algorithm~\ref{alg:weightedMinMaxIndegreeAcyclic} produces a vertex ordering that minimizes the maximum weighted left-degree.
\end{theorem}
\begin{proof}
  Let $n=|V|$.
  Let $\sigma = (\sigma_1, \dots, \sigma_n)$ be the order returned by the algorithm, and let $\sigma' = (\sigma'_1, \dots, \sigma'_n)$ be an arbitrary vertex order.
  Take an index $i$ such that $\sigma_i$ is a vertex with maximum weighted left-degree in $\sigma$.
  Let $\sigma'_j$ be the last vertex in $\sigma'$ that appears among ${\sigma_1, \dots, \sigma_i}$.
  Then
  \[
    \cev{d}_{\sigma}^w(\sigma_i) = d^w(\sigma_i, \{\sigma_1, \dots, \sigma_i\}) \leq d^w(\sigma'_j, \{\sigma_1, \dots, \sigma_i\}) \leq d^w(\sigma'_j, \{\sigma'_1, \dots, \sigma'_j\}) = \cev{d}_{\sigma'}^w(\sigma'_j).
  \]
  The first and last equations follow by definition of the weighted left-degree.
  The first inequality holds because $\sigma'_j$ belongs to $\{\sigma_1, \dots, \sigma_i\}$, and $\sigma_i$ was selected as the vertex of minimum weighted degree in this set.
  The second inequality follows from the non-negativity of $w$ and the inclusion $\{\sigma_1, \dots, \sigma_i\} \subseteq \{\sigma'_1, \dots, \sigma'_j\}$.

  Thus, the maximum weighted left-degree in $\sigma$ is no larger than in $\sigma'$, which completes the proof.
\end{proof}

Hence, we can compute an acyclic orientation that minimizes the maximum weighted indegree in strongly polynomial time.
This stands in contrast to the unconstrained (possibly cyclic) case, where the same problem is NP-hard.

The complexity results discussed in this subsection are summarized in Table~\ref{tab:minMaxIndegreeComplexities}.

\setlength{\tabcolsep}{6pt}
\renewcommand{\arraystretch}{1.5}
\begin{table}[h]
  \centering
  \begin{tabular}{|c|c|c|}
    \hline
    & min max $\varrho$ & min max $\varrho^w$  \\
    \hline
    orientation & in P~\cite{asahiro2007graph, de1995regular, venkateswaran2004minimizing} & NP-hard~\cite{asahiro2011approximation, asahiro2007graph} \\
    \hline
    acyclic orientation & in P~\cite{matula1983smallest} & in P, Thm~\ref{thm:weightedMinMaxIndegreeAcyclic}\\
    \hline
  \end{tabular}
  \caption{Complexities of finding an orientation that minimizes the maximum (weighted) indegree}\label{tab:minMaxIndegreeComplexities}
\end{table}

It may seem natural to consider the analogous first step of the inc-max problem, that is, finding an orientation which maximizes the minimum (weighted) indegree.
However, this is not meaningful in the acyclic case, since every acyclic orientation necessarily includes at least one source vertex with indegree~$0$.

\subsection{Hardness of lexicographic refinements}\label{sec:decMinIncMax:complexity}
Since minimizing the maximum left-degree is solvable in polynomial-time, it is natural to ask the complexity of the next step of the dec-min objective, that is, minimizing the number of vertices with that left-degree.
Unfortunately, this step already turns out to be APX-hard, which could also be proved based on the NP-hardness reduction in~\cite{borradaile2017egalitarian} from Minimum Set Cover.
\begin{theorem}\label{thm:decMinAPX}
  Minimizing the number of vertices with left-degree exactly $\cev{d}_{\min}(G)$ in a $\cev{d}_{\min}(G)$\nobreakdash-bounded order is APX-hard even for simple graphs with $\cev{d}_{\min}(G)=3$.
\end{theorem}
\begin{proof}
  We give an approximation-preserving reduction from Minimum Vertex Cover on $3$\nobreakdash-regular simple graphs, which is APX-complete~\cite{alimonti2000some}.
  Let $H = (U, F)$ be such an instance.
  The empty instance is trivial, so assume that $U$ is non-empty.

  For each $v \in U$, introduce a disjoint copy $G_v$ of the gadget shown in Figure~\ref{fig:decMinAPXGadget}.
  \begin{figure}[ht]
  \centering
  \begin{tikzpicture}[scale=1]
    \centering
    \tikzset{VertexStyle/.append style = {minimum size = 20pt,inner sep=0pt}}
    \SetVertexMath
    \Vertex[x=-1.5, y=2, L=t^v_1]{t1}
    \Vertex[x=-1.5, y=0, L=t^v_2]{t2}
    \Vertex[x=-1.5, y=-2, L=t^v_3]{t3}
    \Vertex[x=1.5, y=-2, L=t^v_4]{t4}
    \Vertex[x=1.5, y=0, L=t^v_5]{t5}
    \Vertex[x=1.5, y=2, L=t^v_6]{t6}

    \Vertex[x=-1.5, y=3.5, empty]{t7}
    \Vertex[x=.5, y=3.5, empty]{t8}
    \Vertex[x=2.5, y=3.5, empty]{t9}

    \draw[wavy] (t1) to (t7);
    \draw[wavy] (t6) to (t8);
    \draw[wavy] (t6) to (t9);

    \draw (t1)--(t2);
    \draw (t2)--(t3);
    \draw (t3)--(t4);
    \draw (t4)--(t5);
    \draw (t5)--(t6);
    \draw (t1)--(t5);
    \draw (t6)--(t2);
    \draw (t2)--(t4);
    \draw (t5)--(t3);
  \end{tikzpicture}
  \caption{The gadget $G_v$ for $k=3$ corresponding to $v\in U$ in the proof of Theorem~\ref{thm:decMinAPX}. The wavy edges are the three non-gadget edges representing the edges of $H$ incident to $v$.}\label{fig:decMinAPXGadget}
\end{figure}
  For each $v \in U$, choose one of the three edges of $H$ incident to $v$ to be represented at $t^v_1$, and let the other two be represented at $t^v_6$.
  For each edge $uv \in F$, add a non-gadget edge between the vertex of $G_u$ representing $uv$ and the vertex of $G_v$ representing $uv$.
  The gadget vertex sets are disjoint, every internal gadget edge is a single edge, and the simplicity of $H$ implies that at most one non-gadget edge joins any two gadgets.
  Thus the resulting graph $G$ is simple, and the construction has size polynomial in the size of $H$.

  Let $\tau$ denote the minimum vertex-cover size of $H$, and let $\opt(G)$ denote the minimum number of vertices of left-degree $3$ in a $3$\nobreakdash-bounded order of $G$.
  
  Given a vertex cover $X \subseteq U$ of $H$, we construct a $3$\nobreakdash-bounded order of $G$ with at most $|X|$ vertices of left-degree $3$, by specifying the vertices from right to left.
  Process the vertices of $X$ in an arbitrary order.
  Whenever $t^v_3$ is still present for the current $v \in X$, remove $t^v_3$ and then repeatedly remove vertices of current degree at most $2$ until no such vertex remains. We argue that this process removes all vertices, and thus yields a valid order of $V$. It is easy to confirm that removing $t^v_3$ triggers the removal of the gadget $G_v$, hence every gadget $G_v$ for $v\in X$ is removed.
  If $u \in U\setminus X$, then every edge of $H$ incident to $u$ has its other endpoint in $X$.
  Consequently, all three non-gadget edges incident to $G_u$ disappear when the gadgets corresponding to vertices of $X$ are removed. This triggers the removal of $G_u$.
  Therefore, all vertices of $G$ are removed, forming a $3$-bounded order. The only vertices of left-degree $3$ arise from removing $t^v_3$ for $v\in X$, so at most $|X|$ such vertices appear. Applying this to an optimal vertex cover implies $\opt(G) \leq \tau$.
  Moreover, taking $X=U$ in this construction shows that $G$ has a $3$\nobreakdash-bounded order, so $\cev{d}_{\min}(G)\leq 3$.
  On the other hand, every vertex of $G$ has degree at least $3$, so the last vertex in any order has left-degree at least $3$.
  Hence $\cev{d}_{\min}(G)\geq 3$, and therefore $\cev{d}_{\min}(G)=3$.

  Conversely, given a $3$\nobreakdash-bounded order $\sigma$ of $G$ with $k$ vertices of left-degree $3$, we construct a vertex cover $X\subseteq U$ of $H$ whose size is at most $k$.
  Let $X \subseteq U$ consist of those vertices $v$ for which some vertex of $G_v$ has left-degree $3$ in $\sigma$.
  We show that $X$ is a vertex cover of $H$.
  Suppose that an edge $uv \in F$ is not covered by $X$, and let $u'v'$ be the corresponding non-gadget edge with $u' \in V(G_u)$ and $v' \in V(G_v)$.
  Since $u \notin X$, every vertex of $G_u$ has left-degree at most $2$.
  Therefore, by the construction of the gadget, the last two vertices of $G_u$ must be $t^u_1$ and $t^u_6$, or else a vertex of $G_u$ would have left-degree at least $3$.
  For $t^u_1$ and $t^u_6$ to avoid left-degree $3$, all their non-gadget neighbors must appear later in the order.
  In particular, $u'$ precedes $v'$.
  Applying the same argument to $G_v$ gives that $v'$ precedes $u'$, a contradiction.
  Hence $X$ is a vertex cover of $H$ with size at most $k$. Applying this to an optimal order implies $\tau\leq \opt(G)$.

  Hence, the two constructions above imply $\opt(G)=\tau$.
  Moreover, from any $3$-bounded order with $k$ vertices of left-degree $3$, we can extract in polynomial time a vertex cover $X$ with $\tau\leq |X|\leq k$.
  Since $\opt(G)=\tau$, this gives $|X|-\tau\leq k-\opt(G)$.
  In particular, if $k\leq C\opt(G)$ for some $C\geq 1$, then $|X|\leq C\tau$.
  The polynomial-time construction and extraction therefore give an approximation-preserving reduction from Minimum Vertex Cover on $3$\nobreakdash-regular simple graphs, proving APX-hardness.
\end{proof}

Now we turn to the analogous first steps of the inc-max problem. As noted above, maximizing the minimum left-degree is trivial: in every order the first vertex has left-degree $0$. For a connected graph, it is also easy to construct an order with exactly one $0$-left-degree vertex. Thus, minimizing the number of vertices attaining the minimum left-degree is polynomial-time solvable. More generally, it suffices to consider connected graphs, since the optimization can be performed independently in each connected component.

A natural next step is to maximize the length of the initial interval of left-degree values that are each attained by exactly one vertex. More precisely, we seek to maximize the integer $\ell$ for which there exists an ordering of the vertices with exactly one vertex of left-degree $i$ for every $i\in\{0,\dots,\ell-1\}$, which turns out to be APX-hard.

\begin{theorem}
  Maximizing the largest integer $\ell$ such that there exists exactly one vertex with left-degree $i$ for each $i \in \{0, \dots, \ell - 1\}$ over the vertex orders of a simple graph is APX-hard.
\end{theorem}
\begin{proof}
We prove APX-hardness by a PTAS-reduction from Maximum Independent Set on graphs of maximum degree at most $3$, which is APX-complete~\cite{alimonti2000some}.

Let $H=(U,F)$ be an instance of Maximum Independent Set with maximum degree at most $3$. We construct an instance $G$ of our maximum increasing prefix problem by taking the complement of $H$, that is,
$G=\overline{H}=(U,\overline{F})$.
Let $\alpha$ denote the maximum cardinality of an independent set in $H$, and let $\opt(G)$ denote the optimum value of the maximum increasing prefix problem on $G$. We claim that $\alpha-3\leq \opt(G)\leq\alpha$.
For the upper bound, consider an ordering of $U$ with value $\ell$, and let $v_i$ be the unique vertex of left-degree $i$, for $i\in\{0,\dots,\ell-1\}$. By the definition, every other vertex has left-degree at least $\ell$. Since the vertex in position $j$ has left-degree at most $j-1$, the vertices $v_0,\dots,v_{\ell-1}$ are precisely the first $\ell$ vertices of the ordering. Their left-degrees sum to $\binom{\ell}{2}$, which is the maximum possible number of edges induced by $\ell$ vertices in a simple graph. Hence these vertices form a clique in $G$. As $G=\overline{H}$, they form an independent set in $H$. Consequently, $\ell\leq\alpha$ for any order of $G$, hence $\opt(G)\leq \alpha$.

For the lower bound, if $\alpha\leq 3$, then $\opt(G)\geq 0\geq\alpha-3$. Assume therefore that $\alpha\geq 4$, and let $S=\{v_0,\dots,v_{\alpha-1}\}$ be a maximum independent set in $H$. Then $S$ is a clique in $G$. Place the vertices of $S$ first in an arbitrary order. Their left-degrees are then $0,1,\dots,\alpha-1$, respectively.

Now consider any vertex $v\in U\setminus S$. Since $H$ has maximum degree at most $3$, the vertex $v$ has at most three neighbors in $H$. Equivalently, among the vertices of $S$, at most three are non-neighbors of $v$ in $G$. Thus, $v$ has at least $\alpha-3$ neighbors among the first $\alpha$ vertices. Hence every vertex outside $S$ has left-degree at least $\alpha-3$.
It follows that the vertices of $S$ with left-degrees $0,1,\dots,\alpha-4$ are the unique vertices of these left-degrees in the ordering.
Therefore, $\opt(G)\geq\alpha-3$,
and the claimed inequality follows.

We next show that this reduction is a PTAS-reduction. Let $\varepsilon>0$ be arbitrary, set $\delta=\frac{\varepsilon}{2}$, and suppose that there is a polynomial-time $(1+\delta)$-approximation algorithm for increasing prefix maximization.
Every graph of maximum degree at most $3$ has an independent set of size at least $\frac{|U|}{4}$, by a greedy argument. Hence $|U|\leq 4\alpha$.
If $|U|<\frac{24(1+\varepsilon)}{\varepsilon}$, then the instance can be solved exactly by exhaustive enumeration in time depending only on $\varepsilon$.
We may therefore assume that $|U|\geq\frac{24(1+\varepsilon)}{\varepsilon}$, which implies $\alpha\geq\frac{6(1+\varepsilon)}{\varepsilon}>3$.
Applied to $G$, the approximation algorithm produces an ordering whose value $\ell$ satisfies $\ell\geq\frac{\opt(G)}{1+\delta}$.
The first $\ell$ vertices of this ordering form a clique in $G$ and hence an independent set of size $\ell$ in $H$. Thus, the approximation ratio of the resulting independent set is at most
\[
\frac{\alpha}{\ell}
\leq
\frac{\alpha(1+\delta)}{\opt(G)}
\leq
\frac{\alpha(1+\delta)}{\alpha-3}
=
\frac{\alpha(1+\frac{\varepsilon}{2})}{\alpha-3}
\leq 1+\varepsilon.
\]
The last inequality is equivalent to $3(1+\varepsilon)\leq \frac{\varepsilon\alpha}{2}$, which follows from the lower bound on $\alpha$ above.

Thus, for every $\varepsilon>0$, a polynomial-time $(1+\frac{\varepsilon}{2})$-approximation algorithm for maximum increasing prefix would yield a polynomial-time $(1+\varepsilon)$-approximation algorithm for Maximum Independent Set on graphs of maximum degree at most $3$.
\end{proof}

\subsection{Inc-min and dec-max objectives}
We now introduce two natural counterparts to the dec-min and inc-max problems: the \emph{increasingly minimal (inc-min)} and \emph{decreasingly maximal (dec-max)} orientation problems.
The inc-min (acyclic) orientation problem asks for an (acyclic) orientation in which the indegree vector sorted in non-decreasing order is lexicographically minimal.
Similarly, in the dec-max (acyclic) orientation problem, the indegree vector sorted in non-increasing order is to be lexicographically maximized.
We prove that, for both the inc-min and dec-max problems, a natural first step is already APX-hard, regardless of whether the orientation is required to be acyclic. This implies the NP-hardness of these objectives as well.

As a first step, the inc-min problem maximizes the number of $0$-left-degree vertices, which we prove to be APX-hard.
\begin{theorem}\label{thm:incMinAPXH}
  Maximizing the number of vertices with indegree $0$ in an (acyclic) orientation is APX-hard even for $3$\nobreakdash-regular simple graphs.
\end{theorem}
\begin{proof}
  We give an approximation-preserving reduction from Maximum Independent Set on $3$\nobreakdash-regular simple graphs, which is APX-complete~\cite{alimonti2000some}.
  Let $G=(V,E)$ be such an instance, and let $\alpha$ denote the maximum cardinality of an independent set in $G$.

  First, the zero-indegree vertices of any orientation of $G$ form an independent set.
  Thus, every orientation, and in particular every acyclic orientation, has at most $\alpha$ vertices of indegree $0$.

  Conversely, let $I$ be an independent set in $G$.
  Place the vertices of $I$ first in an arbitrary vertex ordering and orient every edge from left to right.
  The resulting orientation is acyclic, and every vertex of $I$ has indegree $0$.
  Applying this construction to a maximum independent set shows that both in the unrestricted and in the acyclic setting, the optimum number of zero-indegree vertices is exactly $\alpha$.

  Moreover, from any orientation of $G$ we can extract its set of zero-indegree vertices in polynomial time.
  This set is an independent set whose cardinality is exactly the objective value of the orientation.
  Hence the reduction preserves both the optimum value and the value of every feasible solution.
  In particular, a $(1+\varepsilon)$\nobreakdash-approximation for either version of the orientation problem would give a $(1+\varepsilon)$\nobreakdash-approximation for Maximum Independent Set on $3$\nobreakdash-regular simple graphs.
  Therefore, both versions are APX-hard.
\end{proof}

We argue that the analogous first step for the dec-max problem is APX-hard as well.

\begin{corollary}\label{cor:decMaxAPXH}
  Maximizing the number of vertices with indegree $\Delta$ in an (acyclic) orientation is APX-hard, where $\Delta$ denotes the maximum degree in the input graph, even on $3$\nobreakdash-regular simple graphs.
\end{corollary}
\begin{proof}
  On a $3$\nobreakdash-regular graph, we have $\Delta=3$.
  Reversing every arc gives a bijection on the orientations of the graph and preserves acyclicity.
  Under this reversal, the indegree of each vertex changes from $\varrho(v)$ to $3-\varrho(v)$.
  Consequently, the vertices of indegree $3$ in an orientation are precisely the vertices of indegree $0$ in its reversal.
  The reversal therefore preserves the objective value when passing between the two maximization problems, so the result follows from the preceding theorem.
\end{proof}

Since an inc-min or dec-max (acyclic) orientation maximizes the number of $0$-indegree vertices or $\Delta$-indegree vertices, respectively, Theorem~\ref{thm:incMinAPXH} and Corollary~\ref{cor:decMaxAPXH} imply the following corollary.

\begin{corollary}\label{cor:incMinDecMaxNPH}
  The inc-min orientation and acyclic orientation problems and the dec-max orientation and acyclic orientation problems are NP-hard even for $3$\nobreakdash-regular simple graphs.
\end{corollary}

Table~\ref{tab:lexComplexities} summarizes the complexities of the considered orientation and acyclic orientation problems.

\setlength{\tabcolsep}{6pt}
\renewcommand{\arraystretch}{1.5}
\newcolumntype{L}{>{\centering\arraybackslash}m{90pt}}
\newcolumntype{M}{>{\centering\arraybackslash}m{60pt}}
\begin{table}[h]
  \centering
  \begin{tabular}{|L|M|M|M|M|}
    \hline
    & dec-min & inc-max & inc-min & dec-max \\
    \hline
    orientation
    & in P~\cite{borradaile2017egalitarian}
              & in P~\cite{frank2022decreasing2,frank2022decreasing1}
                        & \multirow{3}{*}{\makecell[c]{NP-hard,\\ Cor~\ref{cor:incMinDecMaxNPH}}}
    & \multirow{3}{*}{\makecell[c]{NP-hard,\\ Cor~\ref{cor:incMinDecMaxNPH}}} \\
    \cline{1-3}
    \multirow{2}{*}{acyclic orientation}
    & \multirow{2}{*}{\makecell[c]{NP-hard~\cite{borradaile2017egalitarian},\\ Cor~\ref{cor:decMinIncMaxNPH}}}
    & \multirow{2}{*}{\makecell[c]{NP-hard,\\ Cor~\ref{cor:decMinIncMaxNPH}}}
    &  &  \\
    &  &  &  &  \\
    \hline
  \end{tabular}
  \caption{Complexities of lexicographic orientation problems without indegree bounds.}
  \label{tab:lexComplexities}
\end{table}

\section{Concluding remarks and open questions}
Our results show that separable convex optimization over the integer points of the indegree polytope is strongly polynomial-time solvable by minimum-cost flow, while restricting the feasible indegree vectors to its corners yields a sharp dichotomy for every fixed discrete convex function $\varphi$.
If $\varphi$ is linear on the positive integers, the $\varphi$\nobreakdash-sum can be minimized over acyclic orientations in linear time; otherwise exact optimization is NP-hard on connected simple graphs, and the problem is APX-hard when $\varphi$ is non-negative.
More strongly, for every breakpoint $k\geq2$ of $\varphi$, the hard instances can be chosen so that $k=\left\lceil \frac{|E|}{|V|}\right\rceil$, the ceiling of their average indegree.
The exact $O\left(2^{|V|}|V|+\poly(|V|,|E|)\right)$-time dynamic program is essentially optimal throughout this hard regime, since Theorem~\ref{thm:exactExponentialETHLowerBound} rules out a $2^{o(|V|)}$\nobreakdash-time exact algorithm under the Exponential Time Hypothesis for every fixed function on the hard side of the dichotomy.

For the minimization of the sum of squared indegrees over acyclic orientations, Corollary~\ref{cor:fixedPowerNoPTAS} proves APX-completeness even on connected simple graphs and gives a strongly polynomial-time $4$\nobreakdash-ap\-prox\-i\-ma\-tion.
It remains open whether the smallest-last greedy algorithm analyzed in Section~\ref{sec:minSquareSum} achieves a constant approximation ratio.
Among the graph instances we examined, its worst-case approximation ratio tends to $\frac{9}{7}$.

In Section~\ref{sec:decMinIncMax}, we exhibited a simple graph for which the sets of optimal acyclic orientations for the dec-min and inc-max objectives are disjoint.
This highlights a clear distinction from the universality in the unconstrained setting.
We conjecture that, for every pair of distinct integers $p,q\geq2$, there exists a graph for which the sets of optimal acyclic orientations minimizing the $p$\nobreakdash-power sum and the $q$\nobreakdash-power sum of indegrees differ.

\section*{Acknowledgment}
We are grateful to Andr\'as Frank for insightful discussions and for drawing our attention to relevant literature.

\smallskip
This research has been implemented with the support provided by the Ministry of Innovation and Technology of Hungary from the National Research, Development and Innovation Fund, financed under the ELTE TKP 2021-NKTA-62 funding scheme, by the Ministry of Innovation and Technology NRDI Office within the framework of the Artificial Intelligence National Laboratory Program, and by the Lend\"ulet Programme of the Hungarian Academy of Sciences --- grant number LP2021-1/2021.
The first author was supported by the Ministry of Innovation and Technology of Hungary from the National Research, Development and Innovation Fund --- grant number ADVANCED 150556.
The second author was supported by the EK\"OP-24 University Excellence Scholarship Program of the Ministry for Culture and Innovation from the source of the National Research, Development and Innovation Fund.

\smallskip
OpenAI ChatGPT was used to assist with language editing.
The authors assume responsibility for all content.

\section*{Conflict of interest}
The authors have no competing interests to declare that are relevant to the content of this article.

\bibliographystyle{plain}
\bibliography{bibliography}
\end{document}